\documentclass[a4paper,11pt,oneside]{memoir}
\setlrmarginsandblock{3cm}{*}{1}
\setulmarginsandblock{3cm}{*}{1}
\checkandfixthelayout[nearest]
\usepackage[latin1]{inputenc}   
\usepackage[T1]{fontenc}              
\usepackage{lmodern} 
\usepackage{amsmath,amssymb,amscd,bm,mathtools,xypic,extpfeil,graphicx}    
\usepackage{tikz,tikz-cd}
\usetikzlibrary{arrows.meta}
\usepackage{relsize}
\AtEndDocument{%
  \ifnum\value{lastsheet}=1\thispagestyle{empty}\fi}
\usepackage{amsthm}
\newtheorem{thm}{Theorem}
\theoremstyle{definition}
\newtheorem{defn}[thm]{Definition}
\theoremstyle{example}

\theoremstyle{remark}
\newtheorem{remark}[thm]{Remark}
\theoremstyle{lemma}
\newtheorem{lemma}[thm]{Lemma}
\theoremstyle{corollary}
\newtheorem{corollary}[thm]{Corollary}

\newcommand {\NN}{{\mathbf N}}
\newcommand {\QQ}{{\mathbf Q}}

\newcommand {\RR}{{\mathbf R}}
\newcommand {\ZZ}{{\mathbf Z}}

\newcommand{\val}[2]{{\mathrm{val}_#1(#2)}}
\newcommand{\vk  }{{\mathbf{k}}}
\newcommand{\lvk}{{p^{\mathfrak k}}}

\newcommand{\HB}{{\mathcal{H}}}
\newcommand{\FS}{{\mathcal{S}}}

\newcommand{\Min}[1]{{\mathcal{H}^{min}_{#1}}}
\newcommand{\Max}[1]{{\mathcal{H}^{max}_{#1}}}

\newcommand{\HBE}[1]{{\mathcal{H}^E_#1}}

\newcommand{\FK}{{\mathbb{F}}}
   
\newcommand{\red}[1]{{\mathrm{Red}_{#1}}}
 
\newcommand{\CH} {{\mathcal{CR}}}

\newcommand{\im}{{\mathrm{Im}}}

\newcommand{\la}[2]{{\mathfrak #1}_#2}
\newcommand{\lao}[1]{{\mathfrak #1}}

\newcommand{\fund}[1]{{\{#1\}}}
\newcommand{\funddivZ}[1]{{\mathcal Div}{\{#1\}}}

\newcommand{\fundbd}[1]{{d^0\{#1\}}}

\author{Marcel B\"o{}kstedt}
\begin{document}

\title{Configuration graph cohomology}
\maketitle
\section{Introduction}
 The motivation for this paper comes from the study of the fundamental group of a type of configuration spaces, see \cite{Link}.  
The configuration spaces we consider depend on parameters. The most important parameter consists
of a graph $\Gamma$. An additional parameter consists of that to each vertex $v\in V(\Gamma)$ we  
associate a natural number $k_v$, that is a family of numbers $\vk=\{k_v\}_v$. The pair $(\Gamma,\vk)$ determines a configuration space of points on a manifold $M$. A point in the configuration space consists of the following data. For each vertex $v$ in the graph $\Gamma$ there is a point $Z_v$ in the symmetric power $S^{k_v}M$. These elements satisfy the condition that if $v$ and $w$ are connected by an edge in $\Gamma$, then $Z_v$ and $Z_w$ correspond to disjoint subsets of $M$. If $k_v=1$ for all $v$, the configuration space is a configuration space of ordered points in $M$, one point for each vertex of $\Gamma$.

Examples of the configuration spaces studied in \cite{Link} originally arose in the study of moduli spaces of stable configurations of minimal energy described by certain vortex equations.  They play a role in two-dimensional QFTs arising from a quantization of a supersymmetric extension of gauged sigma models(ibid) with toric targets. The configuration spaces considered in this paper most closely related to this situation are configuration spaces of the above type for $M$ a surface, and $\Gamma$ a class of graphs derived from the structure of toric varieties.  

In spite of this, the gentle reader should be warned that the present paper is  
purely algebraic in nature. 
 
The
fundamental group of such configuration space for $M$ an oriented, compact surface was studied in \cite{Link}.
In this paper we determined the fundamental group of this class of configuration spaces of points. There were no restrictions on the graph $\Gamma$, or on the weights $k_v$.

The structure of the fundamental groups depends on $M$, but also 
on a certain finitely generated Abelian group. In the special case $M=S^2$, the
fundamental group of the configuration space equals this group.
The group is considered under the name $E(\Gamma,\vk)$ in \cite{Link} and given by generators and relations    
there. The actual computation of the group is a simple exercise in solving
linear Diophantine equations. Given $(\Gamma,\vk)$ it is a trivial task for a computer to write down the
elementary divisors of the group. On the other hand, it is not so easy to describe how this group varies as we vary the weights or even the graph. The dependence on the parameters is the subject of this paper.

There is something about this situation that is unusual in algebraic topology. Sometimes, you study either a very big
class of spaces like ``all manifold'' or ``the algebraic $K$-theory of arbitrary rings $R$'' or ``classifying spaces of finite  groups''  
where you can only  make general statements about the structure of  
various algebraic invariants. On the other hand, you often study a small, comparatively regular and well behaved  
family of spaces like the Grassmannians, the surfaces, or
perhaps the classifying spaces of simple Lie groups. 

If you are in this "regular" situation, you can often collect the spaces you have into a filtering limit system.
In favorable cases this system  satisfies stabilization properties, and we are led to study the 
colimit of the system. This can be the source of much fun. 

The limit spaces we obtain are space
like $QS^0$ or  $BU$ or $K(\ZZ)$. In these cases, you can hope that you sooner or later can get precise numerical answers to questions
like ``what is the fundamental group'' or ``what is the homology''. 

The class of configuration spaces that we are interested in fall between
these two situations. There are many of them, but not overwhelmingly so. They follow regularities as you vary the parameters, but they are not so regular that they are boring. 
For each individual set of parameter values it is easy to find the answer, either by using pencil and paper or by using a machine, because
it's just linear algebra.

You can vary the data on two different levels.
The first question you can ask is ab out the family of spaces obtain by varying the
weights $\vk$ while keeping the graph    $\Gamma$ fixed.
We usually do not have canonical maps between the spaces in the family, but we can still
ask about what happens when you let parameters grow towards infinity. This is vaguely 
similar to situations common in analytic number theory or in statistical mechanics, 
but it seems to be an unusual point of view in algebraic topology. However, the recent preprint
\cite{FWW} is inspired by similar ideas.

 In the first part of  this paper we give a more structural
understanding of the how the fundamental group varies while we do not change $\Gamma$. We show that this group is closely related to the partially ordered set of
bipartite subgraphs of $\Gamma$. In doing so, we find that it is convenient to reinterpret $E(\Delta,\vk)$ as the first cohomology group of a cochain complex. That is, we are defining a cohomology theory for vertex weighted graphs. 

In the past, there has been various definitions of a homology theory of graphs. For instance,  the definition of graph cohomology in \cite{BS} (see also \cite{EH}) is clearly related to ours. They consider configuration spaces that are important special cases of the configuration spaces that motivated this study.
They are interested in the homology of these configuration spaces, 
while the algebraic questions we deal with in this paper are motivated by a study of the fundamental groups. 

There is an additional difference between our approach and the situation studied in \cite{BS}. We are interested in configuration spaces with multiple points of the same color, which for the cohomology groups corresponds to allowing vertex weights to differ from 1. 

Eventually one might want to study the cohomology of the configuration spaces we consider using a generalization of the methods of \cite{BS}, but we will not discuss this question in the present paper.   

Another question one might ask is for the cohomology of the universal
cover or the maximal Abelian cover of the configuration space. The very special case of this where you have only one color is treated in 
\cite{Curvature}. The case of 2 colors is discussed in \cite{RW}. For a similar question, see also \cite{Pochhammer}.

There is also the famous graph cohomology of Kontsevich (\cite{Kontsevich}, see also \cite{CV} and\cite{Igusa}). This theory takes coefficients in cyclic operads, and there does not seem to be an obvious direct relation between that theory and the theory considered in this paper. However, there does seem to be a relation between this theory and the graph cohomology of \cite{BS}. We will discuss this connection further in \cite{BM}.
 
Another homology theory of graphs is discussed in \cite{GLMY}. This homology is somewhat similar to the theory in \cite{BS} in that it uses oriented edges, does not consider vertex weights and has higher homology. The basic chains of the theory are ``regular paths'' in the graph. Such paths do not seem to play an important role in our theory, so probably there is no strong link to our theory.

The ``GKM'' graph cohomology defined in  \cite{GKM} has roots in the cohomology of a toric variety and generalizes this, just like our cohomology theory. The coefficients is a local coefficient system with coefficients in real vector bundles on the graph. There are higher dimensional cohomology groups, defined in a  way similar of the theory in \cite{BS}. This theory has been extensively studied, mainly for its applications in computing equivariant cohomology and $K$-theory of spaces with an action of a torus.  

Then there are cohomology theories of Khovanov type (\cite{KH},\cite{Viro}). According to the authors, this was one of the inspirations for \cite{BS}. Following \cite{HR}, there are two different but equivalent complexes that defines this theory. The reformulation by Viro using the ``enhanced state complex'' is similar to our definition of graph cohomology, while Khovanov's original definition seems analogous to our ``fundamental complex''. At present, this is only a loose analogy. 

It seems to be a difficult question to give a 
complete description of how the group $E(\Gamma,\vk)$ varies as we vary the weights $\vk$. We did try the computer, and are happy to acknowledge the use of the computer system ``sage''. Letting her examine thousands of examples bolstered our confidence in the theorems we prove in this paper, but it didn't lead to a precise conjecture on how the structure of the group $E(\Gamma,\vk)$ depends on   
the parameter $\vk$.    

In the last part of the paper, we try to get at least some results about the order of the torsion group group $T(\Gamma,\vk)\subset E(\Gamma,\vk)$.
For a given $\Gamma$ one can sometimes understand completely how this varies with $\vk$. We give examples of this, and
prove a general structure theorem for the function $\vk \mapsto \val p {\vert T(\Gamma,\vk)\vert}$. This is expressed as
a rational function in the max-plus ring on the variables $\val p{k_v}$. This means that it is related to tropical algebra. 
We don't know if this connection will lead anywhere.

The final question to consider is how the tropical rational function which gives the order of $T(\Gamma,\vk)$
depends on $\Gamma$. This is the highest level of parametrization, and to be honest, we are not able to say much about it.

From a technical point of view this paper is about some elementary questions in  linear algebra.
It is essentially self contained. We now describe the technical set-up and the basic definitions.

Let $\Gamma$ be a graph without  loops or multiple edges.
Let $V(\Gamma)$ be the set of vertices of $\Gamma$, and $E(\Gamma)$ the set of edges.

Let $(\Gamma,\vk)$ be a negative color scheme. That is, for every vertex 
$v\in V$ we have fixed a (positive) natural number $k_v$.
Let $\vk=\{k_v\}_v$, that is, $\vk$ is a vector of weights on the vertices of $\Gamma$.

A subgraph of the graph $\Gamma$ is given by subsets $V'\subset V(\Gamma)$ and
$E'\subset E(\Gamma)$ such that if $e\in E'$, the two endpoints of $e$ are contained in
$V'$. If $\vk$ is a negative color scheme on $\Gamma$ and $i:\Delta\subset \Gamma$ is a subgraph,
there is an induced negative color scheme $i^*(\vk)$. If no confusion is likely, we will write $(\Delta,i^*(\vk))$ as $(\Delta,\vk)$. 

We define the graph cochain complex to be the complex $C^*(\Gamma,\vk)$ whose only non-trivial groups are $
C^0=\ZZ[V]$ and $C^1=\ZZ[E]$. The only non-trivial differential is 
$d^0(v)=\sum_e k_{w}e$ 
where the sum is taken over all edges $e$  of $\Gamma$ which are incident to $v$, and where $w$ is the other endpoint of $e$.
We denote the unique edge in $\Gamma$ connecting the vertices $v,w$ by $e(v,w)$, and the two vertices incident
to an edge $e$ by $v(e)$ and $w(e)$. The edges are not assumed to be oriented, 
so there is a choice inherent in this notation. 
Whenever we  use this notation either the choice is irrelevant, or we explicitly specify it. 
In this notation, we define the differential of the chain complex as
\[
d^0(v)=\sum_wk_we(v,w)
\]
where the sum is understood to be taken over all vertices $w$ such that there is an edge $e(v,w)$ connecting $v$ and $w$. According to the description of the group $E(\Gamma,\vk)$ in proposition 30 of \cite{Link}, the group $E(\Gamma,\vk)$ discussed above agrees with $H^1(\Gamma,\vk)$.

The graph cohomology $H^*(\Gamma,\vk)$ considered in this paper is the cohomology of this chain complex.
It can only be non-trivial in the degrees 0 and 1. 
If $i:\Delta\subset \Gamma$ is the inclusion of a subgraph, there is an induced negative color scheme
$(\Delta,\vk)$, and a surjective restriction map $i^*:C^*(\Gamma,\vk)\to C^*(\Delta,\vk)$.
We define the relative graph cohomology $H^*(\Gamma,\Delta,\vk)$ to be the cohomology
of the kernel of the restriction. 
 
There is some room for generalizations. For instance, we could drop the condition that there are no self loops, and we could define the induced negative color scheme for a class of maps of graphs which is more general than the class of injective maps. This seems to be irrelevant for the applications to the topology of configuration spaces, and we will not pursue it here. 

The purpose of this paper is to study how the graph cohomology varies while we keep $\Gamma$
fixed and vary $\vk$. The cohomology in degree 0 is obviously a finitely generated free group. 
In degree 1, the cohomology is still finitely generated, but not necessarily free.
The rank of $H^*(\Gamma,\vk)$ will not depend on $\vk$  (corollary~\ref{rank1}). 
Therefore we focus on the torsion subgroup of $T(\Gamma,\vk)\subset H^1(\Gamma,\vk)$. 
Our point of view is that $\Gamma$ determines a function $T$ on the set of maps $\vk:V(\Gamma)\to \NN$,
namely the function that takes $\vk$ to the torsion subgroup of $H^1(\Gamma,\vk)$.
We consider this function as an invariant of $\Gamma$, which we intend to study.

 We will use the following notation: 
For any subset $A\subset \{0,1,\dots,r\}$
let us define $GCD_A(\vk)$ to be the greatest common divisor of the numbers $\{k_i\}$ for
$i\in A$.  
Eventually we express  $T(\vk)$ as a function of the numbers $GCD_A(\vk)$. The function depends on the structure
of the graph $\Gamma$.   

We give an overview of what is contained in the paper.

In section~\ref{sec:fixed} we take a first shot at giving a more structural description of the torsion group $T(\Gamma,\vk)$ for a fixed graph and a fixed weighing $\vk$. The main subject is a discussion of a concept of orientation. This concept is motivated by an analogy to the cohomology of manifolds. A difference to the manifold situation is that it turns out to be a subtle question to study this orientation at the prime two.

The main results of the section are conditions for orientability of graphs at odd primes in lemma~\ref{lemma.pn.orientation}, and at the prime 2 in lemma~\ref{lemma.two.orientation}. More precisely, these lemmas deal with $\ZZ/p^r$--orientations  of graphs satisfying an additional restriction, namely the condition that there are no edges such that the product of the edge weights of its incident vertices is divisible by $p^r$. 

In section~\ref{sec:not.reduced} we continue the study of the cohomology for a fixed graph and a fixed weight. In topology one can attempt to represent homology classes by bordism classes $f:M\to X$, that is, asking if every homology class can be written as a sum of classes $f_*([M])$ where $M$ is the fundamental class of $M$. This is a classical and difficult subject. In general, the answer depends on which cohomology theory and which bordism theory we consider. We do something similar  with graph cohomology. We try to represent cohomology classes as images of fundamental classes of orientable subgraphs.

Along these lines we prove theorem~\ref{generation}. The proof of this theorem is regrettably technical. We try to go through the argument slowly, and attempt to cut the it up into individually edible pieces.

Since we now know that we can represent cohomology classes by inclusions of oriented subgraphs, the obvious next project is to describe the cohomology in terms of the category of oriented subgraphs. It turns out that it is convenient to fix a prime $p$ at this point, and for this prime discuss the $p$--primary torsion subgroup of $H^1(\Gamma,\vk)$.

In section~\ref{sec:forest} we organize the fundamental classes into a graph, which we call the fundamental forest.
The fundamental forest depends on the prime $p$ and also on the 
$p$--valuation of the weights $\vk$.
If we are given the fundamental forest, we are able to reconstruct the chain complex defining our graph cohomology up to quasi equivalence. We don't try to get the optimal result in this direction, but we do get a description of the critical torsion group of the graph cohomology in terms of the fundamental forest in lemma~\ref{main.theorem}. This is the high point in our study of the graph cohomology for fixed graph and fixed weights.

One can think of this result as a function $F_p$
which can be described in graph theoretical terms.
The function $F_p$ orders to a family of non negative integers
$\{a_v\}_{v\in V(\Gamma)}$ a finite set of exponents
\[
F_p(\{a_v\}_{v\in v(\Gamma)})=\{b_i\}
\]
such that if $\{b_i\}=F_p(\{\val p k_v\})$, then  the
$p$-torsion subgroup of $H^1(\Gamma,\vk)$ is isomorphic to $\oplus_i \ZZ/p^{b_i}$. 
However, there are only two cases.
If $p,q$ are odd primes, $F_p=F_q$. On the other hand, the
function $F_2$ can be different from $F_p$ for odd $p$.  

In section~\ref{sec:varying}, we start discussing what happens when we fix the graph, but vary the weights. There are some easy cases that can be understood completely. In particular, if the graph $\Gamma$ is a tree, it is not so hard to compute the order of the torsion. In the general case, we cannot give closed formulas for the torsion, but restrict ourselves to trying to determine the order of the $p$-torsion group. In theorem~\ref{order.oriented.graph}, we use the theory we have developed in the preceding sections to give an algorithm for computing the order of the torsion. 

In section~\ref{sec:tropical-cohomology} we continue the discussion of how the order of the torsion changes when we keep the graph fixed, and vary the weights. For (our) convenience, we now restrict ourselves to the odd torsion. In theorem~\ref{tropical.interpretation} we interpret the order of the torsion in terms of tropical rational functions in the weights. That is, to a graph we give a tropical rational function in the vertex weights which computes the order of the $p$-torsion for us.

We finally specialize the preceding theory to the case of a complete graph. In this case, it is possible to write down an easy formula for the tropical function in terms of elementary tropical symmetric functions. This leads to the final question: How does the tropical function vary when we vary the graph? At present, we cannot provide a good answer to this question.

I'm happy to acknowledge that during this work I have benefited a lot from many discussion with my collaborator N. Rom\~ao. 

\section{Computations for fixed weights $\vk$}
\label{sec:fixed}

\subsection{The splitting into $p$-torsion parts. The fundamental chain}
The group we are mainly interested in is the torsion subgroup $T(\Gamma,\vk)$ of the first cohomology group $H^1(\Gamma,\vk)$. In this paragraph, we will discuss how to interpret elements of this group in terms of subgraphs of $\Gamma$. The first step is to 
discuss $p$-torsion for each prime $p$ separately. In order to do this, we 
consider cohomology groups with coefficients. This can simplify the situation
because of the following lemma.
\begin{lemma}
\label{invertible}
  Let $\vk,\vk'$ define two negative color schemes on the same graph $\Gamma$. Let $R$ be a ring. 
Suppose that there are invertible elements $x_v\in R$ such that
for each $v$ we have that $x_vk_v=k_v'\in R$. Then, $H^*(\Gamma,\vk;R)\cong H^*(\Gamma,\vk';R)$.
\begin{proof}
   Consider the map
$F^*:C^*(\Gamma,\vk;R)\to C^*(\Gamma,\vk';R)$ defined by
$F^0(v)=x_vv$ for $v\in V(\Gamma)$ and $F^1(e)=x_{v(e)}x_{w(e)}e$ for $e\in E(\Gamma)$. This is a chain map, since
\[
F^1d^0(v)=F^1(\sum_w k_w e(v,w))=\sum_wx_vk_wx_we(v,w)=\sum_w x_v k_w' e(v,w) =
d^{\prime 0}(x_v v) = d^{\prime 0}F^0(v) 
\]
It follows that $F^*$ is an isomorphism of chain complexes, so it induces an isomorphism of
cohomology groups.   
\end{proof}
\end{lemma}

The group $H^*(\Gamma,\vk)$ is a finitely generated Abelian group. It has a torsion subgroup, and 
a complementary free Abelian group. The free part is determined up to isomorphism by its rank. The rank is much easier to understand than the torsion part. In particular, it is insensitive to the weights $\vk$, as the following application of lemma~\ref{invertible} shows.

\begin{corollary}
\label{rank1}
The rank of $H^*(\Gamma,\vk)$ is independent of $\vk$.
\begin{proof}
The rank equals the dimension of the $\QQ$-vector space  $H^*(\Gamma,\vk;\QQ)$. If $\vk'$ is a different choice of weights,
we have that $k_v' = (k_v'/k_v)k_v$. Since $k_v'/k_v$ is invertible in  $\QQ$ the corollary follows from
lemma~\ref{invertible}.
\end{proof}
\end{corollary}

We now consider the $p$--primary torsion part of
$H^1(\Gamma,\vk)$ for a given prime $p$. The following application of lemma~\ref{invertible} 
helps to compute this.
Let $\vk_p$ denote the $p$-part of $\vk$, that is $\vk_p=\{(k_p)_v\}_v$ where $(k_p)_v=p^{\val p {k_v}}$. We can use $\vk_p$ as
a weighing of the graph $\Gamma$, and compute the corresponding torsion group $T(\Gamma,\vk_p)$
\begin{lemma}
  $T(\Gamma,\vk)=\prod_p T(\Gamma,\vk_p)$
  \begin{proof}
    Let $\ZZ_p$ be the localization of $\ZZ$ at $p$. 
The torsion of $H^*(\Gamma,\vk;\ZZ_p)$ equals the $p$-primary
part of the torsion of $H^*(\Gamma,\vk)$, so we only have to show that
the torsion of $H^*(\Gamma,\vk;\ZZ_p)$ equals the
torsion of $H^*(\Gamma,\vk_p;\ZZ_p)$. In the ring $\ZZ_p$, we can write
$k_v=(\mathrm{unit}) (k_v)_p$, so that
this follows from
lemma~\ref{invertible}.
\end{proof}

\end{lemma}

\subsubsection{Orientation of graphs}

We think of the cohomology of weighted graphs to be somewhat analogous to the cohomology of manifolds.  
This analogy is not close, but it does suggest a concept of orientability. For some rings $R$ we will define orientability of a negative color scheme with respect to the ring. After that we will re-interpret orientability in terms of the graph theoretical properties of the graph. In the analogous manifold situation there are essentially two possibilities. We can divide the class of rings into two disjoint classes, according to whether $2=0\in R$ or not. A manifold is orientable with respect to any ring in the first class (for example $\ZZ/2$), and for any ring $R$ in the second class it is orientable if and only if it is orientable for the ring $\ZZ$. The orientability of weighted graphs which we consider now is more subtle, especially when we consider the rings $R=\ZZ/2^s$. 
Given a  negative color scheme $(\Gamma,\vk)$ , a key question for us will be to understand for which rings $R$ this color scheme is $R$-oriented.

We note that this form of orientability has nothing to do with the concept of orientation used by Kontsevich in his definition of graph cohomology (\cite{Kontsevich}, see also \cite{CV}). 

The bipartite graphs will play a special role. Recall that a bipartitioning of $\Gamma$ is given by
a map $\alpha:V(\Gamma)\to \{\pm 1\}$, such that if $e$ is an edge, the values of $\alpha$ at the endpoints of $e$
are different. We say that $\Gamma$ is a bipartite graph if it has a bipartitioning. If $\Gamma$ is connected and bipartite, the map $\alpha$ is uniquely determined (up to sign) by $\Gamma$.

A map $\alpha:V(\Gamma)\to\{ \pm 1\}$ 
defines a \textit{fundamental chain}
\[
\fund{\Gamma,\vk}=\sum_v \alpha(v){k_v} v \in C^0(\Gamma,\vk).
\]
This chain is relevant for the cohomology, because
if $\alpha$ is a bipartitioning, then the fundamental chain of 
$\Gamma$ is a cycle in $C^0(\Gamma,\vk)$:
\[
d^0(\sum_v \alpha(v){k_v} v) = \sum_{e\in E(\Gamma)} (\alpha(v)+\alpha(w))k_vk_w=0.
\]

Each coefficient of $\fund{\Gamma,\vk}$ is divisible by the greatest common divisor $GCD(\vk)$, so it makes sense 
to define the divided fundamental classes 
\[
{\funddivZ {\Gamma,\vk}}=\frac 1{GCD(\vk)}\fund{\Gamma,\vk}\in C^0(\Gamma,\vk).
\] 
By the argument above, it follows that this class is also a cycle in
$C^0(\Gamma,\vk)$. 

Suppose that $\Delta\subset \Gamma$ is a bipartite subgraph. We restrict the
weights of $\Gamma$ to weights of $\Delta$. The inclusion of the
sets of vertices $V(\Delta) \subset V(\Gamma)$ defines an inclusion of 0-chains:
$C^0(\Delta,\vk)\subset C^0(\Gamma,\vk)$, but this inclusion is not compatible with the boundary map.

We consider the fundamental chain $\fund{\Delta,\vk}$ as an element of $C^0(\Gamma,\vk)$, but keep in mind that this chain is not necessarily a cycle. 
Similarly, we write the image of the divided fundamental class  in 
$C^0(\Gamma,\vk)$ as ${\funddivZ {\Gamma,\vk}}$.

The boundary of the fundamental chain of $\Delta$ is $\fundbd{\Delta,\vk} \in C^1(\Gamma,\vk)$. We write 
 $\fundbd{\Delta,\vk}=\sum x_e e$. If $e\in E(\Delta)$, then the coefficient $x_e=0$. The only non-trivial contributions to the sum are due to the edges  $e=e(v,w)$ such that $v\in V(\Delta)$, but $e\not\in E(\Delta)$. The set of such edges is known as the edge boundary of $\Delta$ in $\Gamma$.
By another slight abuse of notation, for any unitary ring $R$, we let $\fund{\Delta,\vk}$, $\funddivZ{\Delta,\vk)}$ and  $\fundbd{\Delta,\vk}$  denote the images of these classes in $C^0(\Gamma,\vk;R)$ respectively $C^1(\Gamma,\vk;R)$. 

We are now ready to define orientability for a color scheme with coefficients for certain rings $R$. 
This definition will depend on the ring of coefficients $R$. We will not give a unified
definition valid for all rings, but rather ad hoc definitions for those rings which concern us most. 
\begin{defn}
Suppose that the ring $R$ is either $\ZZ$ or a field.
Let $(\Gamma,\vk)$ be a negative color scheme. We say that
$(\Gamma,\vk)$ is $R$-oriented if $H^0(\Gamma,\vk)\cong R$. An $R$-orientation
of $(\Gamma,\vk)$ is a generator of $H^0(\Gamma,\vk)$ as an  $R$-module.
\end{defn}

This was rather straightforward. Next consider the ring $R=\ZZ/p^s$, where $p$ is a prime. There are classes in $H^0(\Gamma,\vk;\ZZ/p^s)$ that play a similar role to the orientations in the case $R=\ZZ$ or for fields, but unfortunately, the situation is less intuitive in this case.  We are going to formulate orientability with coefficients in the rings $\ZZ/p^s$ in a slightly different way. The definition is going to look weird at first glance, but you will thank us later. Let $i:\ZZ/p^{s-1}\subset \ZZ/p^s$ be the standard inclusion $i(x)=px$. 

\begin{defn}
 The critical cohomology $\CH(\Gamma,\vk;\ZZ/p^s)$ is
the cokernel of the induced map 
\[
i_*:H^0(\Gamma,\vk;\ZZ/p^{s-1})\to H^0(\Gamma,\vk;\ZZ/p^{s}).
\]  
\end{defn}
It follows from the long exact sequence of the cohomology associated to the short exact sequence of coefficients $0\to \ZZ/p^{s-1}\xrightarrow{i} \ZZ/p^s \to \ZZ/p \to 0$ that there is an injective map 
$\CH(\Gamma,\vk;\ZZ/p^s)\to H^0(\Gamma,\vk;\ZZ/p)$. In particular, $\CH(\Gamma,\vk;\ZZ/p^s)$ is a $\ZZ/p$-vector space.

\begin{defn}
 A connected negative color scheme $(\Gamma,\vk)$ is $\ZZ/p^s$-oriented if and only if $\CH(\Gamma,\vk;\ZZ/p^s) \cong \ZZ/p$. If $\Gamma$ is connected, a $\ZZ/p^s$-orientation class of $(\Gamma,\vk)$ is a class in
$H^0(\Gamma,\vk;\ZZ/p^{s})$ which maps to a generator of the 
cokernel.
If $\Gamma$ is not connected, we say that $\Gamma$ is $\ZZ/p^s$-oriented if each component of $\Gamma$
is $\ZZ/p^s$-oriented. We define an orientation class of $\Gamma$
to be family consisting of an orientation class for each component of $\Gamma$.
\end{defn}

We need to be able to recognize when a class in $C^0(\Gamma,\vk,\ZZ/p^s)$  is a $\ZZ/p^s$ - orientation class. Here is a criterion 
which will be useful.

\begin{lemma}
\label{lemma:orientation}
  Let $\Gamma$ be connected.
  A cocycle $z\in C^0(\Gamma,\vk;\ZZ/p^s)$ is a $\ZZ/p^s$-orientation class if and only if the following two
conditions are satisfied:
\begin{itemize}
\item The order of $z$ is $p^s$.
\item If $u\in C^0(\Gamma,\vk;\ZZ/p^s)$ is any cocycle, there is an integer $n$ 
such that $p^{s-1}(u-nz)=0$. 
\end{itemize}
\begin{proof}
  We first prove that an orientation satisfies the two conditions.
  If the order of $z=\sum_v z_v v\in C^0(\Gamma,\vk;\ZZ/p^s)$ is less than $p^s$, every $z_v\in i(\ZZ/p^{s-1})$,
so $z$ is in the image of $i_*:C^0(\Gamma,\vk;\ZZ/p^{s-1}) \subset C^0(\Gamma,\vk;\ZZ/p^s) $. But then the image
of $z$ in the critical cohomology is trivial, so that $z$ isn't an orientation. Also, if $z$ 
is an orientation and $u$ is a cocycle, there is an integer $n$ such that $u-nz$ is the image of 
a cycle under $i_*$. But every element in $C^0(\Gamma,\vk;\ZZ/p^{s-1})$ has order dividing $p^{s-1}$. It follows that $u-nz$ has order dividing $p^{s-1}$. 

For the converse implication, 
we prove that if $z$ satisfies the two conditions, then it is an orientation class.
Since the cokernel of $i_*$ is a $\ZZ/p$ vector space, 
the second condition ensures that the image of $z$ generates this cokernel, so that either $H^1(\Gamma,\vk,\ZZ/p)$ is
a 1-dimensional vector space generated by the image of $z$, or this group is trivial.
The first condition ensures that $z$ has non-trivial image in the cokernel
$C^0(\Gamma,\vk;\ZZ/p)$. Since there are no boundaries in
$C^0(\Gamma,\vk;\ZZ/p)$, it follows that the image of $z$ defines a non-trivial
cohomology class.
We deduce that $z$ is an orientation class.   
\end{proof}
\end{lemma}

\subsection{$R$-orientability and  properties of $\Gamma$}
We first consider the rings where the definition of orientability does not involve the critical cohomology.

\begin{lemma}
  \label{lemma:Z.orientation}
Let $R$ be either a field or $\ZZ$. 
If $\Gamma$ is bipartite,
any connected negative color scheme $(\Gamma,\vk)$ is $R$-orientable.
and $\funddivZ{\Gamma,\vk}$ is an $R$-orientation class of $(\Gamma,\vk)$.

Let $\Gamma$ be a connected graph and let $R$ be either $\ZZ$ or a field which is not of characteristic 2.
Assume that $k_v\not=0\in R$ for each $v\in V(\Gamma)$.
If a negative color scheme $(\Gamma,\vk)$ is $R$-orientable
then $\Gamma$ is bipartite.  

If $R$ is a field of characteristic 2,
any connected negative color scheme is $R$-oriented.  
  \begin{proof}    
We give the proof for $R=\ZZ$. 

If $z=\sum z_v v\in C^0(\Gamma,\vk)$ is
a cocycle, we have that for any edge $e(v,w)$ of $\Gamma$, the coefficient of $e(v,w)$ 
in $d^0(z)$ is trivial. This coefficient is $k_vz_w+k_wz_w$, so that
$k_{v}z_{w}=-k_{w}z_{v}$.
Because $\Gamma$ is connected, it follows that all numbers $z_{v}/k_{v}\in \QQ^*$ agree up to a sign. The sign determines a bipartitioning of the graph, so that if $\Gamma$ is not bipartite, there are no non-trivial cycles in $C^0(\Gamma,\vk)$ and $H^0(\Gamma,\vk)=0$.   

Conversely, assume that $\Gamma$ is bipartite. The divided fundamental class $\funddivZ{\Gamma,\vk}$ is a non-trivial cocycle. In order to show that it is a $\ZZ$-orientation class, we have to show that it generates $H^0(\Gamma,\vk)$.
A cocycle $z\in C^0(\Gamma,\vk)$ is a sum
$z=\sum_v z_v v$ where $k_vz_w=-k_wz_v$. It follows that there is a rational number $\frac mn$ ($m,n$ relatively prime) such that  
$z_v=\frac mn\alpha(v)k_v$ for every $v$. Since $z_v\in \ZZ$,  
${k_v}/n$ is an integer for all $v$. This is equivalent to saying that $n$ divides $GCD(\vk)$, 
and $z=m \frac {GCD(\vk)}n\funddivZ{\Gamma,\vk}$. It follows that the cocycles of $C^0(\Gamma,\vk)$ are precisely the integral multiples of $\funddivZ{\Gamma,\vk}$. This completes the proof the lemma in the case $R=\ZZ$.

Now assume that $R$ is a field not of characteristic 2.  The proof is essentially the same as the proof for $R=\ZZ$. We omit it.

If $R$ is a field of characteristic 2, the class 
$
\sum_{v\in V(\Gamma)}v
$
is a nontrivial cycle. By an argument similar to the argument in the case $R=\ZZ$, we see that every cycle is a scalar multiple of this class.
\end{proof}
\end{lemma}

\begin{corollary}
  \label{rank}
  Let $\Gamma$ be a connected graph.
\begin{gather*}
\mathrm{rank}(H^0(\Gamma,\vk))=
\begin{cases}
  1 &\text{ if $\Gamma$ is bipartite,}\\
  0 &\text{ if $\Gamma$ is not bipartite.}
\end{cases}\\
\mathrm{rank}(H^1(\Gamma,\vk))=
\begin{cases}
  \#E(\Gamma)-\#V(\Gamma)+1 &\text{ if $\Gamma$ is bipartite,}\\
  \#E(\Gamma)-\#V(\Gamma) &\text{ if $\Gamma$ is not bipartite.}
\end{cases}
\end{gather*}
\begin{proof}
  The result for $(H^0(\Gamma,\vk))$ is the previous lemma.
The difference in ranks between  $H^0(\Gamma,\vk)$ and $H^1(\Gamma,\vk)$
agrees with the negative Euler characteristic 
$\#E(\Gamma)-\#V(\Gamma)$ of $C^*(\Gamma,\vk)$. This gives the result for 
$H^1(\Gamma,\vk)$.
\end{proof}
\end{corollary}

\begin{defn}
  A negative color scheme $(\Gamma,\vk)$ is $R$--reduced if there are no edges
$e(v,w)\in E(\Gamma)$ such that  $k_vk_w=0 \in R$.
\end{defn}
\begin{remark}
  If $R$ is a field, $(\Gamma,\vk)$ connected and $R$-reduced, then either $\Gamma$ is a single vertex graph or for every $v\in V(\Gamma)$ there is a $w\in V(\Gamma)$ such that $k_vk_w\not=0$. It follows that $k_v\not=0$ for all $v$. So we can reformulate part of lemma~\ref{lemma:Z.orientation} to say that if $p$ is an odd prime, $(\Gamma,\ZZ/p)$ is oriented and $\ZZ/p$--reduced, then $\Gamma$ is bipartite.
\end{remark}

\begin{defn}
  Let $(\Gamma,\vk)$ be a negative color scheme. The $\ZZ/p^s$ reduction $\red{p^s}(\Gamma,\vk)$
  is the negative color scheme we obtain from $(\Gamma,\vk)$ by removing all edges $e(v,w)\in E(\Gamma)$ such that
$v_p(k_vk_w)\geq s$.
\end{defn}

We can express $\red{p^s}(\Gamma,\vk)$ as $(\Delta,\vk)$ where $\Delta\subset \Gamma$ is the subgraph with
$V(\Delta)=V(\Gamma)$ and $E(\Delta)=\{e\in E(\Gamma)\mid k_{v(e)k_{w(e)}}\not=0\mod p^s\}$.
$\red{p^s}(\Gamma,\vk)$ is $\ZZ/p^s$ reduced.
The inclusion of this sub-scheme defines a map of chain complexes 
$R^*:C^*(\Gamma,\vk;\ZZ/p^s) \to C^*(\red{p^s}(\Gamma,\vk);\ZZ/p^s)$ 
which is an isomorphism in degree 0.
If $(\Gamma,\vk)$ is already $\ZZ/p^s$-reduced, $R^*$ is an isomorphism
of chain complexes.  

The rest of this section will deal with $R$-reduced negative color schemes.
Our goal is to giving a graph theoretical
condition on such schemes that is equivalent to the condition that the color scheme is $R$-oriented.
We will see that in the case $R=\ZZ/p^s$, this is easier to deal with for odd primes $p$ than for $p=2$.

We are going to deal with certain arithmetic properties of graphs. Before we start discussing this, we make two preliminary remarks.
\begin{remark}
One type of argument we are going to use repeatedly is the following. If we want to prove 
that a certain statement about vertices in a connected graph is true for all vertices of the graph, it is sufficient to
prove the induction start: there is at least one vertex for which the statement holds, together with the induction step:
if $e(v,w)$ is an edge, and the statement is true for $v$, then it is also true for $w$. We will refer to this method as 
``connected induction''.   
\end{remark}

\begin{remark}
In the proofs of the following lemmas, we will use the $p$-adic valuation of a number $x\in\ZZ/p^s$. We consider this as $\val p x \in\{0,\dots,r-1,\infty\}$, with $\val p x=\infty$ if and only if $x=0$.  With the usual convention $\infty + x=\infty$, it follows that $\val p{a}+\val p b\leq \val p{ab}$, and that if $ab\not=0$, we have that $\val p a+\val p b=\val p {ab}$. 
Note that $\mathrm{val}_p$ does not have the full valuation property.  This is because in the case that $ab=0$, it can happen that $\val p a+\val p b < \infty = \val p {ab}$.
But a certain weaker property is satisfied, namely that if $ab\not=0$, then
$\val p {ab}=\val p a+\val p b$. This  means that if $ab\not=0$, then the following two
conclusions are valid:
$\val p {ab} = \val p {ac}$ if and only if $\val p b = \val p c$ and
$\val p {ab} < \val p {ac}$ if and only if $\val p b < \val p c$.
We refer to this as the  ``restricted valuation property''.
\end{remark}

\begin{lemma}[The principle of small cycles]
\label{lemma:subfundamental}
Let
$(\Gamma,\vk)$ be a connected and  $\ZZ/p^s$-reduced negative color scheme.
Let $z=\sum_v z_v v\in C^0(\Gamma,\vk;\ZZ/p^s)$ be a cocycle. Assume that there is a vertex
$v_0\in V(\Gamma)$ such that $\val p {z_{v_0}}\leq \val p {k_{v_0}}$. Then
$c=\val p {k_{v}} - \val p {z_{v}}$ does not depend on $v$. Moreover, $c\geq 0$.
\begin{proof}
Let $c=\val p {k_{v_0}} - \val p {z_{v_0}}\geq 0$. We use connected induction to prove that
$c=\val p {k_{v}} - \val p {z_{v}}$ for all $v\in V(\Gamma)$. The vertex $v_0$ provides the induction start. We have to check the
induction step.
  Let $e(v,w)\in E(\Gamma)$ be an edge. Assume that 
$\val p {k_{v}} - \val p {z_{v}}=c$. 

Since $(\Gamma,\vk)$ is $\ZZ/p^s$-reduced, $\val p {k_v} +\val p {k_w}\leq s-1$. It follows that
$\val p {z_v} +\val p {k_w}= \val p {k_v} -c +\val p {k_w}\leq r-c-1$. Using this, and since
$k_vz_w+ k_wz_v=0$, the restricted valuation property implies that $\val p {k_v} +\val p {z_w}=\val p {k_w} +\val p {z_v}$,
so that  $\val p {k_w} -\val p {z_w}=\val p {k_v} -\val p {z_v}=c$.
\end{proof}
\end{lemma}

We will now deal with the question of when a connected negative color scheme
$(\Gamma,\vk)$  is $\ZZ/p^s$ -- oriented. 
If the graph has only one vertex and no edges, it follows directly from
the definition that it is $\ZZ/p^s$ -- oriented, so we will assume that
$\Gamma$ has at least one edge.
For odd primes $p$, there are essentially two mutually exclusive possibilities.

\begin{lemma}
\label{lemma.pn.orientation}
Let $p$ be an prime and $\Gamma$ a connected, $\ZZ/p^s$-reduced  graph.
Assume that $\Gamma$ is bipartite. Then $\funddivZ{\Gamma,\vk}$ 
is an $\ZZ/p^s$-orientation of $(\Gamma,\vk)$.
If $p$ is odd and $\Gamma$ is not bipartite, $(\Gamma,\vk)$ is not $\ZZ/p^s$ - oriented. 
  \begin{proof}    
By its definition, at least one coefficient in $\funddivZ{\Gamma,\vk}$
is prime to $p$, so the order of $\funddivZ{\Gamma,\vk}$ in $C^0(\Gamma,\vk;\ZZ/p^s)$ is $p^s$.
According to lemma~\ref{lemma:orientation}, in order to prove that $\funddivZ{\Gamma,\vk}$
is an orientation class,  it suffices to prove that
if $z\in C^0(\Gamma,\vk;\ZZ/p^s)$ is any cocycle, there is an integer $n$ such that
$z-n\funddivZ{\Gamma,\vk}\subset pC^0(\Gamma,\vk;\ZZ/p^s)$.

Let $z =\sum_v z_v v \in C^0(\Gamma,\vk)$ be a cocycle.
If all the coefficients $z_v$ are divisible with $p$, we chose $n=0$, and we are already done.
So we can as well assume that there is a vertex $v_0$ such that $z_{v_0}$ is not divisible by $p$. 
In this case $\val p {z_{v_0}} =0 \leq \val p {k_{v_0}}$. According to
lemma~\ref{lemma:subfundamental}, the principle of small cycles,  it follows that
$\val p {k_{v}}- \val p {z_{v}}=c$ is independent of $v$ where $c\geq 0 $.
In particular, $\val p{z_{v}}\leq \val p{k_v}$ for all $v\in V(\Gamma)$. 
Now write $\funddivZ{\Gamma,\vk}=\sum_v a_v v$.

Let $v_1$ be the vertex where $\val p{k_v}$ attains its minimum, so that
the coefficient $a_{v_1}$ is invertible modulo $p$.
There is an $n\in \ZZ$ such that $\val p{z_{v_1}- na_{v_1}} > \val p {z_{v_1}}$.
The next step is to use connected induction to prove that 
for every $v\in V(\Gamma)$
\[
\tag{*}
\val p{z_{v}- na_{v}} > \val p{z_v}.
\]
The induction start is that $(*)$ is true for the vertex $v=v_1$.

For the induction step, assume that $e(v,w)\in E(\Gamma)$ and that            
$(*)$ is true for $v\in V(\Gamma)$. 
Since $\Gamma$ is $R$-reduced we have that
$\val p{ z_vk_w} \leq \val p { k_vk_w }<r$, so $z_vk_w\not =0$.
By $(*)$ and the restricted valuation property
$\val p {(z_v-na_v)k_w} > \val p {z_vk_w}$.  
Using that $a_vk_w=-k_va_w$ and $z_vk_w=-z_wk_v$ we get
that 
\[
\val p{k_vz_w-nk_va_w}=\val p{nk_wa_v-k_wz_v}=\val p{-(z_v-na_v)k_w}>\val p{k_vz_w}. 
\]
Since $z_vk_w\not =0$ the restricted valuation property shows that
$\val p{z_w-na_w}>\val p{z_w}$ as required.

It follows from $(*)$ that 
$z-n\funddivZ{\Gamma,\vk}$ is divisible by $p$ for every $v\in V(\Gamma)$.
We have now finished the proof that $\funddivZ{\Gamma,\vk}$ 
is an orientation class.

We  prove the converse statement.
Assume that $p$ is odd, and that $\Gamma$ is not bipartite.
We want to show that $(\Gamma,\vk)$ is not $\ZZ/p^s$ oriented.
Since $\Gamma$ is not bipartite, there exists a sequence consisting of an
odd number of vertices $v_0,v_1,\dots,v_{2m}$ such that every $v_i$ is connected
to $v_{i+1}$ by an edge, and $v_{2m}$ is connected to $v_0$. 
Assume that $z$ is an orientation of $(\Gamma,\vk)$. We can write
$z=\sum n_v p^{a_v} v$ and where each $n_v$ is prime to $p$. We also write
$k_v=m_v p^{b_v}$ where $m_v$ is prime to $p$. It follows that
$m_vn_w p^{a_v+b_w}+m_wn_vp^{a_w+b_v}=0 \mod p^s$.
Because of the principle of small cycles (lemma~\ref{lemma.pn.orientation}), there are two possibilities. Either
$a_v> b_v$ for all $v$, or $a_v\leq b_v$ for all vertices $v$.
Since $z$ has order $p^s$ by lemma~\ref{lemma:orientation}, there is at least one
$v$ such that $a_v=0$, which means that we can exclude the case $a_v> b_v$. 

Since we can thus assume that $a_v\leq b_v$ for all $v$, we see that
$a_v+b_w\leq b_v+b_w<s$. It follows from this and the equality
$m_vn_w p^{a_v+b_w}+m_wn_vp^{a_w+b_v}=0 \mod p^s$
that
$m_vn_w+m_wn_v=0\mod p$, and $m_v/n_v = - m_w/n_w \mod p$. Chasing this through the
vertices of the sequence, we get that 
\[
m_{v_0}/n_{v_0}=-m_{v_1}/n_{v_1}=m_{v_2}/n_{v_2}=\dots =-m_{v_0}/n_{v_0}\mod p
\]
Since $p$ is odd, this is a contradiction. We conclude that $(\Gamma,\vk)$ 
cannot have an orientation. 
\end{proof}
\end{lemma}

\subsection{The case $R=\ZZ/2^s$}
Lemma~\ref{lemma.pn.orientation}
shows that bipartite graphs are $\ZZ/2^s$-oriented, but
it's not true that a $\ZZ/2^s$-oriented graph is necessarily bipartite.

Let $\Gamma$ be a connected $\ZZ/2^s$-reduced graph. The reduction
$\Gamma'=\red {2^{s-1}}\Gamma$ is not necessarily connected.
Now assume that $\Gamma'$ has a bipartitioning $\alpha$. 
That is, we assume that there is 
a map $\alpha:V(\Gamma)\to \pm 1$ such that if $e(v,w)\in E(\Gamma)$  
and $\val 2 {k_vk_w}\leq s-2$, then $\alpha(v)=-\alpha(w)$. 
This amounts to choosing a bipartitioning of each component of $\Gamma'$.
Under this assumption, by a slight extension of notation,
we define the fundamental class and the divided fundamental class of $(\Gamma,\vk)$
to be
\begin{align*}
\fund{\Gamma,\vk}&=\sum_v \alpha(v){k_v} v \in C^0(\Gamma,\vk),\\
\funddivZ{\Gamma,\vk}&=\sum_v \alpha(v)\frac{k_v}{GCD(\vk)} v \in C^0(\Gamma,\vk).\\
\end{align*}
This fundamental class is a cocycle:
\[
d^0\fund{\Gamma,\vk}=\sum_{e(v,w)\in E(\Gamma')} 
(\alpha(v){k_vk_w} + \alpha(w)k_wk_v)+
\sum_{e(v,w)\in E(\Gamma)\setminus E(\Gamma')}
 2k_vk_w =0+0=0
\]
The fundamental class depends on the choice of bipartitioning of $\Gamma'$, but
if there is no good reason to do otherwise, we will suppress this dependence in the notation. 

We will need the following ugly lemma twice in the proof of lemma~\ref{lemma.two.orientation}.
\begin{lemma}
\label{little.argument}
Let $\Gamma$ be a graph. Let $a,b:V(\Gamma)\to \NN$ be two positive functions.
Assume that there is non-negative integer $q$ such that the following three
conditions are satisfied:
\begin{enumerate}
\item For every vertex
$v$ we have that $\val 2 {b_v}=q+\val 2 {a_v}$.
\item If $e(v,w)\in E(\Gamma)$, then $a_vb_w+a_wb_v=0 \mod 2^r$.
\item  If $e(v,w)\in E(\Gamma)$, then $\val 2{a_va_w} \leq r-q-2$.
\end{enumerate}
Then $\Gamma$ is bipartite.
\begin{proof}
Since $\val 2{b_v} = q+ \val 2{a_v}$ we can find odd
integers $n_v$ such that $n_v 2^q a_v=b_v \mod 2^r$.  
If $e(v,w)\in E(\Gamma)$,  it follows that  
\[
n_v 2^q a_va_w+n_w 2^q a_va_w = a_wb_v+a_vb_w= 0\mod 2^r,
\]
so that $\val  2 {n_v+n_w} +\val 2 {a_va_w}\geq r - q$.
We rearrange:
\[
\val 2 {n_v+n_w}\geq r-q -\val 2 {a_va_w} \geq  r-q -(r-q-2)=2,
\]
 so that $n_v+n_w=0\mod 4$. It follows that for $v,w$ in the same component of $\Gamma$, 
the reduction modulo 4 of the odd numbers $n_v,n_w$ agree
up to sign. On the other hand, if $e(v,w)\in E(\Gamma)$ we have that $n_v\not=n_w\mod 4$, since $n_v$ is odd. 
The value of $n_v=\pm 1 \mod 4$ defines a bipartitioning of $\Gamma$. 
\end{proof}
\end{lemma}

\begin{lemma}
\label{lemma.two.orientation}
\begin{enumerate}
\item 
\label{orientation.two.conditions}
Let $(\Gamma,\vk)$ be a connected, $\ZZ/2^s$-reduced negative color scheme
.
If $(\Gamma,\vk)$ is oriented, then either $(\Gamma,\vk)$ is bipartite,
or $(\Gamma,\vk)$ is not bipartite, but  
$\Gamma'=\red {\ZZ/2^{s-1}}(\Gamma,\vk)$ is, and  $\val 2 {GCD(\vk)}=0$.
\item If $(\Gamma,\vk)$ is bipartite the divided fundamental class
, $\funddivZ{\Gamma,\vk}$ is an orientation of $(\Gamma,\vk)$.
\item If $\Gamma'$ is bipartite and
$\val 2 {GCD(\vk)}=0$, then each divided fundamental class
$\funddivZ{\Gamma,\vk}$ is an orientation of $(\Gamma,\vk)$.
\end{enumerate}
\begin{proof}
We first prove statement 1. 
The logic of the proof is slightly convoluted. We first prove that if $(\Gamma,\vk)$ is oriented
then $\Gamma'=\red {\ZZ/2^{s-1}}(\Gamma,\vk)$ is bipartite. After that we prove that
if $(\Gamma,\vk)$ is  oriented and $\val 2 {GCD(\vk)}>0$, then $\Gamma$ is bipartite, which completes the
proof of statement 1 of the lemma.

Assume that $(\Gamma,\vk)$ is oriented with orientation $u$.
By lemma~\ref{lemma:orientation} the cocycle $u$ has order $2^s$, so that
there will be at least one vertex $v_0$ such that $\val 2 {u_{v_0}}=0$. 
Arguing by connected induction and lemma~\ref{lemma:subfundamental} as 
in the proof of  lemma~\ref{lemma.pn.orientation}, we see that 
there is a constant $c\geq 0$ such that $\val 2 {k_v} = c +\val 2{u_v}$
for any $v\in V(\Gamma)$. 

We claim that $\Gamma'$ is bipartite. It suffices show that
$r=s, q=c, a_v=u_v ,b_v=k_v$  satisfies the conditions of lemma~\ref{little.argument}
applied to the graph $\Gamma'$.
The first two conditions are obviously satisfied. 
The condition we have to check is that if $e(v,w)\in E(\Gamma')$, then 
$\val 2{u_vu_w}\leq s-c-2$. But since $\Gamma'$ is $2^{s-1}$-reduced
by definition, we have that 
$\val 2 {k_vk_w}\leq s-2$.
Therefore
\[
\val 2{u_vu_w}= \val 2 {k_vk_w}-2c \leq s-2-2c\leq s-c-2. 
\]
 
Next assume  that  in addition to that $(\Gamma,\vk)$ is oriented 
we have that
$\val 2 {GCD}>0$. We need to prove that $\Gamma$  is bipartite.
That $\val 2 {GCD}>0$ means that each $k_v$ is even. Since there is a $v_0$ with
$\val 2{u_{v_0}}=0$, we conclude that 
$c\geq 1$. Let $\bar k_v =k_v/2 \mod 2^{s-1}$. 
We claim that lemma~(\ref{little.argument}) applies to $\Gamma$ with
$r=s-1,q=c-1,a_v=u_v,b_v=\bar k_v$. 
The three conditions of lemma~(\ref{little.argument}) translate into
\begin{enumerate}
\item $\val 2{\bar k_v}=c-1 +\val 2 {u_v}$ with $c-1\geq 0$.
\item $u_v\bar k_w+u_w\bar k_v=0\mod 2^{s-1}$.
\item $\val 2 {u_vu_w} \leq s-c-2$.
\end{enumerate}
The first two conditions follow immediately from the assumptions.
We have to check the third. But $\val 2 {k_vk_w}\leq s-1$ so
\[
\val 2{u_vu_w} = \val 2{k_vk_w}-2c \leq s-2c-1 \leq s-c-2.
\]
This finishes the proof of statement (\ref{orientation.two.conditions})
of the lemma.

Statement 2 is a part of lemma~\ref{lemma.pn.orientation}.

We finally prove statement 3.
Assume that $\Gamma'$ is bipartite and $\val 2{GCD(\vk)}=0$. 
Choose a bipartitioning $\alpha$ of $\Gamma'$ with corresponding 
fundamental class $\funddivZ{\Gamma',\vk}$. 
Using lemma~\ref{lemma:orientation} again, in order to prove that the fundamental class is an orientation, we
need to show that if $z\in C^0(\Gamma,\vk)$ is a cocycle, there is 
an integer $n$ such that $z-n\funddivZ{\Gamma',\vk}$ is divisible by 2.
 
Let $v_0\in V(\Gamma)$ be such that $k_{v_0}$ is odd. We make the following two claims,
which constitute a 2--primary version of the principle of small cycles.
If $z_{v_0}$ is odd, then $\val 2 {k_v} =\val 2 {z_v}$ for all $v$.
If $z_{v_0}$ is even, then $\val 2 {k_v}< \val 2 {z_v}$ for all $v$. 

We prove the claims by connected induction. The assumption on $v_0$  
provides the induction start for each of the two statements.

Let  $e(v,w)\in E(\Gamma)$. 
The induction step for the first statement is that 
if $\val 2{k_v}=\val 2{z_v}$ then  $\val 2{k_w}=\val 2{z_w}$.
To prove this, we note that 
$\val 2{z_vk_w} = \val 2{z_v}+\val 2{k_w}=\val 2{k_v}+\val 2{k_w}\leq s-1$, since
$\Gamma$ is $\ZZ/2^s$-reduced. It follows from the restricted valuation property that
\[
\val 2 {k_v}+\val 2 {z_w}=\val 2 {k_vz_w}=\val 2 {k_w}+\val 2 {z_v}.
\]
This completes the proof of the induction step for the first statement.

The induction step for the second statement is that 
if $\val 2{k_v}<\val 2{z_v}$ then  $\val 2{k_w}<\val 2{z_w}$.
There are two possibilities. If 
$\val 2{z_vk_w} \leq s-1$ we can complete the proof of the induction statement
by  arguing as in the first case. If $\val 2{z_wk_v} \geq s$, we have that
$\val 2{k_v} \geq s-\val 2 {z_w}$, and 
\[
\val 2 {k_w}\leq s-1-\val 2 {k_v} \leq s-1-(s-\val 2 {z_w})=\val 2{z_w}-1 <\val 2{z_w}.  
\]
This completes the proof of our claims.

The two claims finish our proof of statement 3, since
in the first case, $z_v - \funddivZ{\Gamma',\vk}_v$ is always even, and in the second case $z_v$ is always even.
\end{proof}
\end{lemma}

\subsection{Exact sequences}

There are some exact sequences around which probably deserve closer attention than we give them in this paper. We will use them in a few special cases. 
Here is a first obvious observation.

\begin{lemma}
\label{first.properties}
  If $\Gamma=\Gamma_1\cup \Gamma_2$ is a disjoint union of two graphs, for any $\vk$ 
there is a natural isomorphism 
$H^*(\Gamma,\vk)\cong H^*(\Gamma_1,\vk)\oplus H^*(\Gamma_2,\vk)$.
\end{lemma}

As usual, we have various long exact sequences of cohomology. 

Suppose that $\Gamma$ has two subgraphs   $i_1:\Gamma_1\subset\Gamma$ and $i_2:\Gamma_2\subset \Gamma$. If
$V(\Gamma) = V(\Gamma_1) \cup V(\Gamma_1)$ and $E(\Gamma) = E(\Gamma_2) \cup E(\Gamma_2)$,  
we write that $\Gamma = \Gamma_1\cup \Gamma_2$. There is an intersection graph
$\Gamma_1\cap \Gamma_2$ defined by that 
$E(\Gamma_1\cap \Gamma_2)=E(\Gamma_1)\cap E(\Gamma_2)$
and $V(\Gamma_1\cap \Gamma_2)=V(\Gamma_1)\cap V(\Gamma_2)$, and we have inclusions of subgraphs
$j_1:\Gamma_1 \cap \Gamma_2 \subset \Gamma_1$ and  $j_2:\Gamma_1 \cap \Gamma_2 \subset \Gamma_2$.

\begin{lemma}
The canonical map $H^*(\Gamma,\Gamma_1,\vk) \to H^*(\Gamma_2,\Gamma_1\cap \Gamma_2,\vk)$ is an isomorphism.
\begin{proof}
  Already the map of chain complexes $C^*(\Gamma,\Gamma_1,\vk)\to C^*(\Gamma_2,\Gamma_1\cap \Gamma_2,\vk)$ 
is an isomorphism.
\end{proof}
  \end{lemma}

Using this lemma, one can in the usual way construct a long exact sequences of pairs of graphs, and Mayer-Vietoris sequences.
We will use this in \ref{NOcase} to compare two graphs which only differ by an edge $e$. That is, $V(\Gamma_1)=V(\Gamma_2)$ and 
$E(\Gamma_1)=E(\Gamma_2)\cup \{ e\}$. In this case, the relative chain complex $C^*(\Gamma_1,\Gamma_2,\vk)$ is just 
a copy of $\ZZ$ generated by $[e]$. We get an exact sequence of cohomology groups
\begin{equation}
  \label{eq:LES}
  0 \to H^0(\Gamma_1,\vk) \to H^0(\Gamma_2,\vk) \to \ZZ \to H^1(\Gamma_1,\vk)\to H^1(\Gamma_2,\vk) \to 0
\end{equation}

\subsection{The not orientable case.}
\label{NOcase}
We now turn to the case when $(\Gamma,\vk)$ is $\ZZ/p^s$--reduced but not $\ZZ/p^s$--orientable. We want to show that
the critical cohomology $\CH(\Gamma,\vk;\ZZ/p^s)$ is trivial.

Consider a  $\ZZ/p^s$-reduced graph $(\Gamma,\vk)$. Let $\Delta\subset \Gamma$
be a subgraph, obtained from $(\Gamma,\vk)$ by removing a single edge $e(v,w)$. There is a restriction map
$j^*:C(\Gamma,\vk) \to C(\Delta,\vk)$. The relative cochain complex $C(\Gamma,\Delta,\vk)$ is trivial in dimension 0, and 
generated by $[e(v,w)]$ in dimension 1. The restriction map induces a map of critical cohomology
$j^*:\CH(\Gamma,\vk;\ZZ/p^s)\to \CH(\Delta,\vk;\ZZ/p^s)$. 
\begin{lemma}
\label{one.more.edge}
$j^*$ is injective.   If $(\Gamma,\vk)$ and $(\Delta,\vk)$ are both connected, $\ZZ/p^s$--reduced and $\ZZ/p^s$-orientable, $j^*:\CH(\Gamma,\vk;\ZZ/p^s) \to \CH(\Delta,\vk;\ZZ/p^s)$ is an isomorphism.

If $(\Gamma,\vk)$ is connected but not $\ZZ/p^s$--orientable, the group $\CH(\Gamma,\vk)$  is trivial.
\begin{proof}
  The injectivity of $j^*$  follows from diagram chase in the following commutative diagram
with exact columns and rows. The exactness of two first rows comes from the long exact sequences belonging to the short exact sequence of coefficients $0\to \ZZ/p^{s-1}\to \ZZ/p^s\to \ZZ/p\to 0$.
\[
\begin{CD}
@. 0 @. 0 @.\\
@. @VVV @VVV \\
0 @>>>H^0(\Gamma;\ZZ/p^{s-1}) @>>> H^0(\Gamma ;\ZZ/p^s) @>>> \CH(\Gamma;\ZZ/p^s) @>>> 0\\ 
@.@VVV @VVV @V{j^*}VV \\
0 @>>>H^0(\Delta ;\ZZ/p^{s-1}) @>>> H^0(\Delta ;\ZZ/p^s) @>>> \CH(\Delta;\ZZ/p^s) @>>> 0\\
@. @VVV @VVV  \\
0 @>>>  H^1(\Gamma,\Delta ;\ZZ/p^{s-1}) @>>> H^1(\Gamma,\Delta ;\ZZ/p^s) \\
\end{CD}
\]
If $\Gamma$ and $\Delta$ are both $\ZZ/p^s$-orientable, $j^*:\CH(\Gamma;\ZZ/p^s)\to\CH(\Delta;\ZZ/p^s)$ is an injective map between 1-dimensional vector spaces, so it's an isomorphism. If $\Delta$ is $\ZZ/p^s$-orientable but $\Gamma$ is not, $j^*$ identifies $\CH(\Gamma;\ZZ/p^s)$ with a subgroup of $\CH(\Delta,\ZZ/p^s)$. Since we are assuming that $\Gamma$ is not $\ZZ/p^s$-orientable, this subgroup cannot be $\ZZ/p$. It follows that it has to be trivial. 


Now let $(\Gamma,\vk)$ be a connected negative color scheme.
  Let $\Delta \subset \Gamma$ be a maximal tree. Since $\Delta$ is bipartite, it is orientable. Find a chain of graphs
$\Delta=\Delta_0 \subset \Delta_1 \subset \dots \subset \Delta_m=\Gamma$ such that
$V(\Delta_i)=V(\Gamma)$  for every $i$ and $E(\Delta_{i+1})\setminus E(\Delta_i)$ has exactly one element.
For some $i$, $\Delta_i$ will be $\ZZ/p^s$ orientable but $\Delta_{i+1}$ is not $\ZZ/p^s$ orientable. 
Using lemma~\ref{one.more.edge} applied to $\Delta_i$, we see that $\CH(\Delta_{i+1};\ZZ/p^s)$ is a trivial group. Using inductively that $j^*$ is injective, it follows that $\CH(\Delta_k;\ZZ/p^s)=0$ for $k\geq i+1$. In particular,
$\CH(\Delta;\ZZ/p^s)$ is trivial.
\end{proof}  
\end{lemma}

\section{Graphs that are not $\ZZ/p^s$-reduced}
\label{sec:not.reduced}
In this section we will study $H^0(\Gamma,\vk;\ZZ/p^s)$ in the case that $(\Gamma,\vk)$
is not $\ZZ/p^s$--reduced. We will attempt to generate as many cocycles as possible
as images of fundamental classes of subgraphs $\Delta$. If $f:\Delta\subset \Gamma$ is
the inclusion of
an $\ZZ/p^s$--oriented, $\ZZ/p^s$--reduced subgraph, we will see that the image of its fundamental class
$f_*(\funddivZ{\Delta})$
is a cycle in $C^0(\Gamma,\vk;\ZZ/p^s)$, that is, it is an element of $H^0(\Gamma,\vk;\ZZ/p^s)$. But this is not all we can do. A more general situation is that
$i:\Delta\subset \Gamma$ is a $\ZZ/p^r$--oriented, $\ZZ/p^r$--reduced subgraph for some
$r<s$. In this case, $f_*(\funddivZ{\Delta})$ is still a chain in
$C^0(\Gamma,\ZZ/p^s)$, but not necessarily a cycle. Instead, 
$p^{s-r}f_*(\funddivZ{\Delta})$ is a cycle in $C^0(\Gamma,\ZZ/p^s)$.
Our main result in this section is theorem \ref{main.lemma.generation} which states that there
are enough classes of this type to generate $H^0(\Gamma,\vk;\ZZ/p^s)$. The proof is by
an induction over $s$. The basic idea of the induction is that if
you can generate $H^0(\Gamma,\vk;\ZZ/p^{s-1})$ by classes as those mentioned above, one
can also generate the image of $i_*:H^0(\Gamma,\vk;\ZZ/p^{s-1}) \to H^0(\Gamma,\vk;\ZZ/p^{s})$
in the same way. To complete the proof, it suffices to prove that we can generate all elements of the
cokernel $\CH(\Gamma,\vk;\ZZ/p^s)$ of $i_*$. To prove this, we have to look very closely at the divisibility
properties of the coefficient of a cycle $z=\sum z_vv$. 

Let $\gamma_{p^s}:\CH(\Gamma,\vk;\ZZ/p^s)\to H^0(\Gamma,\vk;\ZZ/p)$ be the injective
inclusion coming from the long exact sequence induced  by the coefficient sequence
of $0 \to \ZZ/p^{s-1}\to \ZZ/p^{s} \to \ZZ/p\to 0$. This map allows us to identify
$\CH(\Gamma,\vk;\ZZ/p^s)$ with a subgroup of $H^0(\Gamma,\vk;\ZZ/p)$.

There is a reduction map $\CH(\Gamma,\vk;\ZZ/p^s)\to \CH(\Gamma,\vk;\ZZ/p^{s-1})$.  
The obvious map of long exact sequences shows that we have inclusions   
$H^0(\Gamma,\vk;\ZZ/p)=\im(\gamma_p)\supset \im(\gamma_p^2) \supset \dots$. 
Since $H^0(\Gamma,\vk,\ZZ/p)$ is a finite group, this sequence stabilizes after a finite number of steps. 
Precisely, the universal coefficient theorem shows that it stabilizes to $\left(H^0(\Gamma,\vk)/\mathrm{torsion}\right)\otimes \ZZ/p$.
It follows from this and from corollary \ref{rank} that
if $\Gamma$ is connected and bipartite, it
stabilizes to $\ZZ/p$, and that if $\Gamma$ is connected and not bipartite, it stabilizes to 0.

We want to use the divided fundamental classes of oriented subgraphs to construct elements of $\CH(\Gamma,\vk;\ZZ/p^s)$. 
We first define the group of cycles in $C^0(\Gamma,\vk;\ZZ/p^{s})$ which can be obtained
from fundamental classes of subgraphs and introduce some notation which will also be useful in later sections.
Let $f:\Delta\subset \Gamma$ be the inclusion of a component of $\red {p^r}\Gamma$. Assume that $(\Delta,\vk)$ is
$\ZZ/p^{s-d}$--oriented, with orientation class $\funddivZ \Delta$.
\begin{defn}
  \begin{align*}
m(\Delta)&=\inf_{v\in V(\Delta)} \val p{k_v} \\    
r(\Delta)&=\inf_{e(v,w)\in B(\Delta)} \val p{k_vk_w},
  \end{align*}
where $B(\Delta)=\{e(v,w)\in E(\Gamma)\mid  v\in V(\Delta)\}\setminus E(\Delta)$ is the 
edge boundary of $\Delta$.
\end{defn}
We see that $r(\Delta)\leq r$ since $\Delta$ is a component of
$\red{r(\Delta)}\Gamma$. Moreover, $r(\Delta)=\infty$ if and only if $B(\Delta)=\emptyset$.
If $\Gamma$ is connected, this happens if and only if $\Delta=\Gamma$.

We consider the fundamental class $\funddivZ \Delta$ as a cochain in
$C^0(\Gamma,\vk;\ZZ/p^{s})$. It's boundary  will be
\[
d(\funddivZ \Delta)=\sum_{e(v,w)\in E(\Delta)}p^{-m(\Delta)}k_vk_w + \sum_{e(v,w)\in B(\Delta)}p^{-m(\Delta)}k_vk_w
\]
Now
$\val p{(p^{-m(\Delta)}k_vk_w)}\geq r(\Delta)-m(\Delta)\geq r-m(\Delta)$ for $e(v,w)\in B(\Delta)$. We infer that
if $r(\Delta)-m(\Delta)+d\geq s$, then $p^d\funddivZ \Delta$ defines
a cocycle in $C^0(\Gamma,\vk;\ZZ/p^{s})$. 
This class maps to
 an element of the critical cohomology $\CH(\Gamma,\vk;\ZZ/p^s)$, which we will also denote by
$\funddivZ \Delta$. 

\begin{defn}
  The group $D_s^d(\Gamma,\vk)$ is the subgroup of $C^0(\Gamma,\vk;\ZZ/p^s)$ generated by the classes $\funddivZ\Delta$ where $r(\Delta)-m(\Delta)\geq s-d$ and  $\Delta$ is an $\ZZ/p^{s-d}$--oriented component of $\red{r(\Delta)}(\Gamma,\vk)$. 
\end{defn}

By the argument above, we have inclusions $p^dD^d_s \subset H^0(\Gamma,\vk;\ZZ/p^s)$.
The following diagram commutes (where $p^{d-1}D^{d-1}_{s-1}\to p^{d}D^{d}_{s}$ is given by $ \funddivZ \Delta\mapsto p\funddivZ  \Delta$):
\begin{equation}
  \label{eq:DandC}
\begin{tikzcd}
    p^{d-1}D^{d-1}_{s-1}\arrow{r}\arrow[d,hook]&p^{d}D^{d}_{s}\arrow[d,hook]\\  
   H^0(\Gamma,\vk;\ZZ/p^ {s-1})\arrow[r,"i_*"]&H^0(\Gamma,\vk;\ZZ/p^s)\arrow[r]&\CH(\Gamma,\vk;\ZZ/p^s)\arrow[r]&0  
\end{tikzcd}
\end{equation}

Here is an easy motivational result:
\begin{lemma}
\label{fundamental.classes.critical.homology}
  Assume that $\Gamma$ is $\ZZ/p^r$-reduced for $r\leq s$. The group $\CH(\Gamma,\vk;\ZZ/p^s)$ is generated by the classes $\funddivZ \Delta$ where
$r(\Delta)-m(\Delta)\geq s$ and  $\Delta$ is an $\ZZ/p^{s}$-oriented component of $\red{r(\Delta)}(\Gamma,\vk)$.
\begin{proof}
If $\Gamma$ is $\ZZ/p^r$-reduced for $r\leq s$, then $\Gamma$ is $s$-reduced. So we can as well assume that $r=s$. 
Let $\Delta_i$ be the connected components of $\Gamma$. Then 
$\CH(\Gamma,\vk;\ZZ/p^s)\cong \oplus_i \CH(\Delta_i,\vk;\ZZ/p^s)$. According to 
lemma \ref{one.more.edge} $\CH(\Delta_i,\vk;\ZZ/p^s)$ vanishes if $\Delta_i$ is not
$\ZZ/p^{s}$-oriented. 
Moreover, if  $\Delta_i$ is $\ZZ/p^{s}$-oriented, then
by lemma~\ref{lemma.pn.orientation} respectively lemma~\ref{lemma.two.orientation}
$\CH(\Gamma_i,\vk;\ZZ/p^s)$ is generated by $\funddivZ{\Gamma_i}$.
\end{proof}
\end{lemma}

\begin{corollary}
    Assume that $\Gamma$ is $\ZZ/p^r$-reduced for $r\leq s$. The group $H^0(\Gamma,\vk;\ZZ/p^s)$ is generated by the classes
$p^dD^d_s$.
\begin{proof}
  Induction over $r$, using lemma~\ref{fundamental.classes.critical.homology} together with 
diagram~\ref{eq:DandC}.
\end{proof}
\end{corollary}

We want to generalize this corollary to the case where $\Gamma$ is not $\ZZ/p^s$-reduced. 
That is, we want to prove: 
\begin{thm}
\label{generation}
  The group $H^0(\Gamma,\vk;\ZZ/p^s)$ is generated by the images of the inclusion maps
$p^dD^d_s \subset H^0(\Gamma,\vk;\ZZ/p^s)$.
\end{thm}
    
The proof of this theorem will take up the rest of this section. The argument will be by induction on $s$.
We introduce further notation which will be helpful when we do the induction step.

For the duration of this proof we will write
$G_s$ for the group generated by the  subgroups $\{p^dD^{d+1}_{s}\}_d\subset C^0(\Gamma,\vk;\ZZ/p^s)$. 
The cohomology $H^0(\Gamma,\vk;\ZZ/p^s)$ equals the cycles in 
$\{p^dD^{d+1}_{s}\}_d\subset C^0(\Gamma,\vk;\ZZ/p^s)$, so we can consider
it as subgroup of the chains. We have a diagram of inclusions
\begin{equation}
  \label{intersection}
\begin{CD}
  pG_s @>>> G_s\\
@V\tilde q_s VV @V q_s VV\\  
H^0(\Gamma,\vk;\ZZ/p^s) @>>> C^0(\Gamma,\vk;\ZZ/p^s)
\end{CD} 
\end{equation}
and if $\iota,\rho$ are the maps induced by inclusion respectively reduction of coefficients, we have 
commutative diagrams
\begin{equation}
\begin{CD}
  G_{s-1} @>>> G_s\\
@V q_{s-1} VV @V  q_s VV\\
C^0(\Gamma,\vk;\ZZ/p^{s-1}) @>\iota >> C^0(\Gamma,\vk;\ZZ/p^s)
\end{CD}
\quad
\begin{CD}
  G_s @>>> pG_{s-1}\\
@V q_s VV @V q_{s-1} VV\\
C^0(\Gamma,\vk;\ZZ/p^s) @>\rho>> C^0(\Gamma,\vk;\ZZ/p^{s-1})
\end{CD}
\end{equation}
 
\begin{remark}
  \label{Gs.and.points}
  For a one-vertex subgraph of a connected graph $\Delta_v\subset \Gamma$ we have that $r(\Delta_v)-m(\Delta_v)\geq 0$, so that $\Delta_v\in D^s_s$ and $G_s$ contains $p^{s-1}v\in C^0(\Gamma,\vk;\ZZ/p^s)$. It follows that for every $s$, $G_s$ contains $p^{s-1}C^0(\Gamma,\vk;\ZZ/p^s)$.
\end{remark}
In this notation, theorem~\ref{generation} is the statement that the map $\tilde q_s$ is surjective
for all $s$. 
The main work will be in the proof of the following lemma.
In order to make the proof easier to follow, we will cut the argument up into eight steps.

\begin{lemma}
\label{main.lemma.generation}
 Assume that $\tilde q_{s-1}$ is surjective. Then, $q_s(pG_s)=q_s(G_s)\cap H^0(\Gamma,\vk;\ZZ/p^s)\subset C^0(\Gamma,\vk;\ZZ/p^s)$. That is, diagram (\ref{intersection}) is a pull-back diagram.
 \begin{proof}

\noindent\emph{1. Reduction of the lemma to a ``main claim''}   
  If $p$ is odd it is possible to shorten the proof somewhat, using that we can
characterize orientability by bipartiteness. We will give a proof which is a little more involved, but
has the virtue that it works in the same way for $p=2$ as for odd primes.

Obviously $q_s(pG_s)\subset q_s(G_s)$, and from diagram (\ref{intersection}) we see that  $q_s(pG_s)\subset  H^0(\Gamma,\vk;\ZZ/p^s)$, so that $q_s(pG_s)\subset q_s(G_s)\cap H^0(\Gamma,\vk;\ZZ/p^s)$ and it is the
opposite inclusion  $q_s(pG_s)\supset q_s(G_s)\cap H^0(\Gamma,\vk;\ZZ/p^s)$ that we have to prove.
Let $z$ by a cocycle in the subgroup $q_s(G_{s})\subset C^0(\Gamma,\vk;\ZZ/p^s)$. 
We can write
\begin{equation}
\label{*}
z=\sum_\Delta \zeta_\Delta \funddivZ \Delta,
\end{equation}
where each $\Delta$ in the sum satisfies the two conditions that each $\Delta$ is an oriented component of some 
$\red {r(\Delta)}(\Gamma,\vk)$, and that for each $\Delta$ we have the 
inequalities
\begin{equation}
\label{**}
\val p {\zeta_\Delta}\geq s-r(\Delta)+m(\Delta)-1. 
\end{equation}
We have to prove 
\[
\tag{Main Claim}
z\in q_s(pG_s),
\] 
since the above argument shows that the main claim will prove the lemma.

\noindent\emph{2. Discussion of $\val p {\zeta_\Delta}$ in a minimal counter example.}

We are going to argue by contradiction. Let us fix a counter example
$z=\sum_\Delta \zeta_\Delta \funddivZ \Delta$ to the main claim 
involving as few $\Delta$ in the corresponding sum (\ref{*}) as possible.
We will refer to the
formal sum $\sum_\Delta \zeta_\Delta \funddivZ \Delta$
as a minimal counterexample to the main claim.
For the rest of the proof, we will work with this particular minimal counter example to
deduce a contradiction.

Let
$X$ denote the set of all $\Delta$ that occur in this minimal counter example. 
To each $\Delta\in X$ we associate the number $\zeta'_{\Delta}$ defined by that
$\zeta'_{\Delta}=p^{-\val p{\zeta_\Delta}}\zeta_\Delta$. Then 
$\zeta'_\Delta$ is an integer which is relatively prime to $p$. 
The minimality of the sum (\ref{*}) has some serious consequences.
We claim that the minimality implies that we can strengthen (\ref{**}) to the
statement that
each $\Delta\in X$ actually satisfies the equality
$\val p {\zeta_\Delta}= s-r(\Delta)+m(\Delta)-1$.
To validate the claim, assume to the contrary that
some $\Delta_0\in D^{d+1}_{s+1}(\Gamma,\vk)$ satisfies the strict inequality
$\val p {\zeta_{\Delta_0}} > s-r(\Delta_0)+m(\Delta_0)-1$ so that
$\val p {\zeta_{\Delta_0}} \geq s+1-r(\Delta_0)+m(\Delta_0)-1$ and  
$\zeta_\Delta \funddivZ \Delta_0\in p^{d+1}D^{d+1}_{s+1}(\Gamma,\vk)\subset pG_{s+1}$. Then 
$\zeta_\Delta \funddivZ \Delta_0$ is a cocycle, and
the formal sum 
\[
  z-\zeta_\Delta \funddivZ \Delta=\sum_{\Delta\in X;\Delta\not=\Delta_0}\zeta_\Delta\funddivZ{\Delta}
\]
  would be a counterexample
to the lemma involving fewer $\Delta$, against the minimality assumption.

That is, we can write   
\begin{equation}
\label{zeta.equality} 
\zeta_{\Delta}=p^{s-r(\Delta)+m(\Delta)-1}\zeta_\Delta'.  
\end{equation}

Note that in particular the equality
$\val p {\zeta_\Delta}= s-r(\Delta)+m(\Delta)-1$ implies that $r(\Delta)<\infty$.

\noindent\emph{3. The consequence of the vanishing of the coefficient of an edge $e$ in $dz$.}
Let $z$ be a minimal counterexample as discussed above.
The fact that $z$ is assumed to be a cocycle imposes further conditions on the coefficients $\zeta_\Delta$. Let $v,w\in V(\Gamma)$ and $e=e(v,w)$ an edge of $\Gamma$. We write $d({\funddivZ\Delta})_e$ for the coefficient
of $e$ in $d({\funddivZ\Delta})$. This coefficient is 0, unless $e$ is in the edge boundary of 
$\Delta$. If $e\in B(\Delta)$, there are three cases:
\[
d({\funddivZ\Delta})_{e(v,w)}=
\begin{cases}
  \alpha_{\Delta}(v)p^{-m(\Delta)}k_vk_w&\text{ if $v\in V(\Delta)$, $w\not\in V(\Delta)$}\\
   \alpha_{\Delta}(w)p^{-m(\Delta)}k_vk_w&\text{ if $v\not\in V(\Delta)$, $w\in V(\Delta)$}\\
 \alpha_{\Delta}(v)p^{-m(\Delta)}k_vk_w+\alpha_{\Delta}(w)p^{-m(\Delta)}k_vk_w&\text{ if $v\in V(\Delta)$, $w\in V(\Delta)$}\\
\end{cases}
\] 
Let $I_v=\{\Delta\in X\mid v\in V(\Delta), e\not\in E(\Delta)\}$ respectively $I_w=\{\Delta\in X\mid w\in V(\Delta), e\not\in E(\Delta)\}$.
Then we can rewrite the coefficient of $e(v,w)$ in $dz$ as follows:
\begin{align*}
(dz)_e&=\sum_{\Delta \in I_v}\zeta_\Delta \alpha_{\Delta}(v)p^{-m(\Delta)}k_vk_w
+\sum_{\Delta \in I_w}\zeta_\Delta \alpha_{\Delta}(w)p^{-m(\Delta)}k_vk_w\\
\end{align*}

Since $z$ is a cocycle
the coefficient of $e$ in $dz$ vanishes. 
But $\zeta_{\Delta}=p^{s-r(\Delta)+m(\Delta)-1}\zeta_\Delta'$, so we are left with:
\begin{equation}
  \label{***}
k_vk_w
     \Big(\sum_{\Delta \in I_v}\alpha_\Delta(v)p^{s-r(\Delta)-1}\zeta_{\Delta}'+\sum_{\Delta \in I_w}\alpha_\Delta(w)p^{s-r(\Delta)-1}\zeta_{\Delta}'\Big)
=0\in \ZZ/p^s  
\end{equation}

\noindent\emph{4. Cutting down the index sets.}
We simplify equation~(\ref{***}) by showing that most of the terms in the sums vanish.
The set $I_v$ is contained in the set of $\Delta$ which are oriented components of $\red r(\Gamma,\vk)$ for various $r=r(\Delta)$.
For each $r$, there is at most one such component $\Delta$ which contains $v$, and similarly for $w$.
That is, the numbers $r(\Delta)$ for $\Delta\in I_v$ are all distinct.
It follows that 
if $I_v$ is non-empty, there is some $\Delta_v\in I_v$ such that if $\Delta\in I_v$ then $r(\Delta)\leq r(\Delta_v)$
and such that equality occurs if and only if $\Delta=\Delta_v$. 

\emph{Claim:} For
$\Delta\not=\Delta_v\in I_v$, the coefficient $k_vk_wp^{s-r(\Delta)-1}\zeta_{\Delta}'$ vanishes.

If $\Delta\not=\Delta_v$, we have that  $r(\Delta_v)> r(\Delta)$.
Since $e$ is in the edge boundary $\Delta_v$, we also have that $\val p{k_vk_w}\geq r(\Delta_v)$. 
We conclude from this that $\val p{k_vk_w}\geq r(\Delta)+1$.

Now we recall that $\val p {\zeta_\Delta}= s-r(\Delta)+m(\Delta)-1$.
 
We see that 
\begin{equation*}
  \val p {k_vk_w p^{s-r(\Delta)-1} \zeta_\Delta'}=\val p {k_vk_w} +s -r(\Delta)-1\geq (r(\Delta)+1)+s -r(\Delta)-1=s
\end{equation*}
Since we are working in the ring $\ZZ/p^{s}$, this proves the claim.
Purging the vanishing terms from the equality~(\ref{***}) we get that
\begin{equation}
\label{****}
k_vk_w p^{s-r(\Delta_v)-1}\alpha_{\Delta_v}(v)\zeta_{\Delta_v}'+k_vk_w p^{s-r(\Delta_w)-1}\alpha_{\Delta_w}(w)\zeta_{\Delta_w}'=0\in \ZZ/p^{s}.  
\end{equation}

\noindent\emph{5. Discussion of the implications of equation (\ref{****}).}
For an edge $e\in E(\Gamma)$ equation~(\ref{****}) allows two possibilities. We say that the edge $e$ is of type I if any of the following
equivalent conditions are satisfied:
\begin{itemize}
\item $k_vk_w p^{s-r(\Delta_v)-1}\zeta_{\Delta_v}'=0\in \ZZ/p^{s}$
\item $k_vk_w p^{s-r(\Delta_w)-1}\zeta_{\Delta_w}'=0\in \ZZ/p^{s}$
\item $\val p {k_vk_w} > s-1-(s-r(\Delta_v)-1)=r(\Delta_v)$
 \item  $\val p {k_vk_w} > r(\Delta_w)$.
\end{itemize}
We say that the edge is of type II if any of the following equivalent conditions hold:

\begin{itemize}
\item $k_vk_w p^{s-r(\Delta_v)-1}\zeta_{\Delta_v}'\not=0\in \ZZ/p^{s}$
\item $k_vk_w p^{s-r(\Delta_w)-1}\zeta_{\Delta_w}'\not=0\in \ZZ/p^{s}$
\item $r(\Delta_v)=\val p {k_vk_w}$
 \item $r(\Delta_w)=\val p {k_vk_w}$   
\end{itemize}
If $e$ is an edge of type II, it follows from (\ref{****}) we that
\[
\alpha_{\Delta_v}(v)\zeta'_{\Delta_v}+\alpha_{\Delta_w}(w)\zeta'_{\Delta_w}=0\in \ZZ/p.
\]
The classification of edges into types depends on the particular minimal expression for
$z$, since the sets $I_v$ and $I_w$ depend on that expression. 
  
\noindent\emph{6. $X_{max}$ and the definition of the subgraph $\Psi$.}
The strategy for  the rest of the proof of the lemma is to  show that some of the terms
in the expression for $z$ can be collected as $\funddivZ{\Psi}$ for some
new oriented subgraph $\Psi$. Then we will argue that this
contradicts the minimality of the expression for $z$.
Our first task will be to construct the subgraph $\Psi$.
 
Recall that for every $\Delta\in X$, $r(\Delta)<\infty$.
Let $r_{max}$ be the maximal value of $r(\Delta)$ for $\Delta\in X$.
We consider the set $X_{max}$  of oriented components of $\red {r_{max}}(\Gamma,\vk)$
which occur with non-trivial coefficient in the expression for $z$.

Now choose $\Delta_0\in X_{max}$ such that $m(\Delta_0)\leq m(\Delta)$ for all $\Delta\in X_{max}$.
We define the subgraph $\Psi\subset \Gamma$ to be the component of $\red{r_{max}+1}(\Gamma,\vk)$ containing  $\Delta_0$.

The reduction $\red {r_{max}} \Psi$ has components $\Delta_0,\Delta_1,\dots,\Delta_k$.
If $\Delta_0$ were the only component of $\Psi$, it would be a component of
 $\red{r_{max}+1}(\Gamma,\vk)$, so that $r(\Delta_0)\geq r_{mzx+1}$
However, $r(\Delta_0)=r_{max}$, and we conclude that 
$\Delta_0$ cannot be the only component of $\Psi$.
Let $\Delta_i$ ($i\not=0$) be another component of
$\red {r_{max}} \Psi$. In particular $r(\Delta_i)\geq r_{max}$. On the other hand,
we cannot have that $r(\Delta_i)\geq r_{max}+1$, because if it were, then
$\Delta_i$ would equal $\Psi$ and $\Delta_0\subset \Delta_i$. 
It follows that $r(\Delta_i)=r_{max}$ for each $i$.

\noindent\emph{7. The restriction of $z$ to $\Psi$ is a cocycle.}

Now we consider the following cochain.
\[
\bar z= \sum_i \zeta_{\Delta_i} \funddivZ{\Delta_i}\in C^0(\Psi,\vk;\ZZ/p^s)  
\]
We compute the coboundary of this cochain. Let $\bar z_e$ be the coefficient of $e$ in $d\bar z$. We want to prove that $\bar z_e=0$ for all  $e\in E(\Psi)$.

The edges $e\in E(\Psi)$ are either
edges of one of the subgraphs $\Delta_i$, or they connect two different
subgraphs $\Delta_i$ and $\Delta_j$. If $e\in E(\Delta_i)$, the
$d\bar z_e=d(\funddivZ{\Delta_i})_e$ is zero for all $i$.
If $e=e(v,w)$ with $v\in V(\Delta_i)$ and $w\in V(\Delta_j)$,
then $\val p {k_vk_w}=r_{max}+1$. Since $r(\Delta_i)=r(\Delta_j)=r_{max}$ by step 5, we
have that $\Delta_i=\Delta_v$ and $\Delta_j=\Delta_w$. According to
equation~(\ref{****})
\begin{align*}
  d(\bar z)_e&=k_vk_wp^{-m(\Delta_i)}\alpha(\Delta_i)\zeta_{\Delta_i}+
               k_vk_wp^{-m(\Delta_j)}\alpha(\Delta_j)\zeta_{\Delta_j}\\
  &=0
\end{align*}
We conclude that $\bar z$ is a cocycle in $C^0(\Psi,\vk;\ZZ/p^s)$.
 
\noindent\emph{7. $\Psi\in D^{d}_{s}$, where $d=s-r(\Psi)+m(\Psi)$.}

There are two conditions to check.  That $r(\Psi)-m(\Psi)\geq s-d$ follows trivially
from the definition of $d$. But we also have to check that $\Psi$ is $\ZZ/p^{s-d}$--orientable.
 
 According to
lemma~\ref{one.more.edge}, to prove this it is sufficient to find a cocycle
representing a non--trivial element in
$C(\Psi,\ZZ/p^{s-d})$, that is, a cocycle $u=\sum u_v v\in C^0(\Psi,\ZZ/p^{s-d})$ 
with at least one coefficient $u_v$ prime to $p$.
We propose to obtain $u$ by dividing $\bar z$ by $p^d$.  
To be able to do so, we have to check that if we write 
$\bar z=\sum_{v\in V(\Psi)}\bar z_v v$, then $\val p{\bar z_v}\geq d$.
We compute this valuation for $v\in V(\Delta_i)$

\begin{align*}
\val p {\bar z_v}&=\val p{\zeta_{\Delta_i} p^{-m(\Delta_i)}k_v}\\
           &=(s-r(\Delta_i)+m(\Delta_i)-1)-m(\Delta_i)+\val p{k_v}\\
           &=s-r_{max}-1 + \val p{k_v}\\
           &=s-r(\Psi)+\val p{k_v}\\
           &\geq s -r(\Psi)+m(\Psi)\\
           &=d  
\end{align*}
It follows that $u=p^{-d}(\bar z)$ is indeed defined. If $v\in V(\Psi)$ is a vertex
such that $k_v=m(\Psi)$, we see from the same computation that
$\val p{u_v}=\val p{k_v}-m(\Psi)=0$. That is, $\Psi$ is
$\ZZ/p^{s-d}$ orientable and $u$ defines an orientation.

Using these orientations, if $e(v,w)$ is an edge in $E(\Psi)$ connecting $v\in V(\Delta_i)$ to $w\in V(\Delta_j)$, 
then 
\begin{equation}
  \label{eq:inducedOrientation}
  \alpha_{\Delta_i}(v)=-\alpha_{\Delta_j}(w)
\end{equation}

\noindent\emph{8. Using $\Psi$ to get a counterexample contradicting the minimality of the expression for $z$.}

Since $\Psi\in D^{s-r_{max}-1+m(\Psi)}_s$, by definition  $p^{s-r_{max}+m(\Psi)-1}\funddivZ\Psi\in q_s(pG_s)$. In particular, this class is a cocycle in $q_s(G_s)$.
Recall that $z$ is a counter example to the main claim, that is $z\not \in q_s(pG_s)$,
but $z$ is a cocycle in $q_s(G_s)$. 
 Consider
 \[
v=z- \zeta_{\Delta_0}'p^{s-r_{max}+m(\Psi)-1}\funddivZ\Psi\in pG_s\not \in q_s(pG_s).
\]
This is a cocycle in $q_s(G_s)$ and $v\not\in q_s(pG_s)$, so that $v$ is also a counter example to the main claim. Now recall that the expansion
$z=\sum_{\Delta\in X} \zeta_\Delta\funddivZ{\Delta}$
  with $\Delta$ an $r(\Delta)$-oriented component of $\red {r(\Delta)}(\Gamma,\vk)$ and
$\val p{\zeta_\Delta}\geq s-r(\Delta)+m(\Delta)-1$, and the assumtion is that the   
the number of terms in this sum is as small as possible. We complete the argument for the main claim by showing that there is an expansion of $v$ satisfying the same conditions, but with a smaller number of terms, contradicting the minimality assumption. 

We can write $v$ as a sum
\[
\left(\sum_{\Delta\in X,\Delta\not=\Delta_0} u_{\Delta}\funddivZ\Delta\right)+(\zeta_{\Delta_0}-\zeta_{\Delta_0}' p^{s-r_{max}+m(\Psi)-1})\funddivZ{\Delta_0}
\]
According to (\ref{zeta.equality}), $\zeta_{\Delta_0}-\zeta_{\Delta_0}' p^{s-r_{max}+m(\Psi)-1}=0$,
so that $u=\sum_{\Delta\in X,\Delta\not=\Delta_0} u_{\Delta}\funddivZ\Delta$. This sum has strictly fewer terms, and it's easy to check that this contradicts the minimality assumption.

We have finished the proof of the lemma.
\end{proof}
\end{lemma}

\begin{proof}[Proof of theorem~\ref{generation}.]
We assume inductively that the map
$\tilde q_{s-1}:pG_{s-1}=\oplus p^dD^d_{s-1} \to H^0(\Gamma,\vk;\ZZ/p^{s-1})$ is surjective, and need to prove that
$\tilde q_s : pG_{s}=\oplus p^dD^d_{s}\to H^0(\Gamma,\vk;\ZZ/p^{s})$ is also surjective.

Let $z\in H^0(\Gamma,\vk;\ZZ/p^{s})\subset C^0(\Gamma,\vk;\ZZ/p^{s})$.
Our first step is to prove that $z\in q(G_s)$.
The reduction $z'$ of $z$ modulo $p^{s-1}$ is also a cocycle, $z'\in H^0(\Gamma,\vk;\ZZ/p^{s-1})\subset C^0(\Gamma,\vk;\ZZ/p^{s-1})$.
By the induction assumption, $z'\in \tilde q_{s-1}(pG_{s-1})$. So we can write
\[
z'=\sum_\Delta \zeta_\Delta \funddivZ \Delta \in C^0(\Gamma,\vk;\ZZ/p^{s-1}),
\]
where each term $\zeta_\Delta \funddivZ \Delta$ in this sum satisfies that for
$d=\val p {\zeta_\Delta}$ we have that $\Delta\in D^d_{s-1}=D^{d+1}_{s}$. 
For each $\Delta$, choose a $\bar \zeta_\Delta\in \ZZ/p^s$ which reduces to $\zeta_\Delta$ modulo $p^{s-1}$.   
Since $\val p {\bar\zeta_\Delta}\geq \val p {\zeta_\Delta}=d$ as above, we have that
$\bar\zeta_\Delta \funddivZ \Delta\in p^dD^{d+1}_s\subset q_s(G_s)$, and consequently
$\sum_\Delta \bar\zeta_\Delta \funddivZ \Delta\in q_s(G_s)$.
Now notice that the difference
\[
u=z-\sum_\Delta \bar\zeta_\Delta \funddivZ \Delta\in C^0(\Gamma,\vk;\ZZ/p^s)
\]  
is divisible by $p^{s-1}$. However,
by remark~\ref{Gs.and.points} this means that $u$ is in $q_s(G_{s})$.
so we can conclude that $z\in q_s(G_s)$.

This means that $z$ is a cocycle
in $q_s(G_s)$, so by lemma~\ref{main.lemma.generation}
$z\in \tilde q_s(pG_s)$, which completes the proof of the theorem.
\end{proof}
 
\section{The fundamental forest}
\label{sec:forest}
\subsection{The structure of the forest}
\label{structure.forest}
Let $(\Gamma,\vk)$ be a negative color scheme.
For each natural number $r$ we consider the reduced negative color scheme $\red{p^r}(\Gamma,\vk)$. 
\begin{defn}
$\HB_r(\Gamma,\vk)$ is the set of the $p^r$-orientable components $\Delta$ of $\red{p^r}(\Gamma,\vk)$ such that $m(\Delta)=\min_{v\in V(\Delta)}(\val p {k_v})<r$.  
\end{defn}
We reformulate this slightly. Consider the following five conditions which a subgraph $\Delta\subset \Gamma$ might or might not satisfy. The conditions depend on the prime $p$ and also on a natural number $r$.
\begin{enumerate}[(H1)]
\item $\Delta$ is connected.
\item  Each edge $e(v,w)\in E(\Delta)$ satisfies the inequality $\val p{k_v}+\val p {k_w}<r$.
\item Let $e$ be an  edge in the edge boundary of $\Delta$, that is $e=e(v,w)\in E(\Gamma)\setminus E(\Delta)$ but  $v\in V(\Delta)$. Then  $\val p{k_v}+\val p {k_w}\geq r$.
\item $\Delta$ is $p^r$ orientable.
\item  If $\Delta$ has one 
vertex $v$ and no edges, then $\val p{k_v}<r$.
\end{enumerate}
\begin{lemma}
\label{H.characterization}Assume that the graph $\Gamma$ does not contain any one vertex components.
$\Delta\subset \Gamma$  represents an element of $\HB_r(\Gamma,\vk)$ if and only if it satisfies H1,H2 H3,H4 and H5 with respect to the number $r$. 
\begin{proof}
Assume that $\Delta\in \HB_r(\Gamma,\vk)$. It satisfies the conditions
H1,H2,H3 and H4 by definition.
Assume that $\Delta=\Delta[v]$ has the single vertex $v$. Since $\Gamma$ does not contain
one vertex components, there is an 
edge $e(v,w)\in E(\Gamma)$ incident to $v$. 
We estimate that $\val p {k_v}\leq \val p {k_v}+\val p {k_w}<r$, so that H5 is also satisfied.

 To prove the ``if'' part,  assume that $\Delta$ satisfies the conditions of the lemma.
It follows from H1, H3 and H4 that $\Delta$  is a $p^r$-oriented  component of $\red {p^r}(\Gamma,\vk)$, so we only have to check that
$\min_{v\in\Delta_v}k_v<r$. In case $\Delta=\Delta[v]$ is a one vertex graph, we have that
$m(\Delta)=\val p {k_v}<r$ by H5. If on the other hand
$\Delta$ has at least two vertices, any $v\in V(\Delta)$ is on some edge $e(v,w)\in E(\Delta)$.  
We get from H2 that $\val p{k_v} \leq   \val p{k_v}+\val p{k_w} <r$.
\end{proof}
\end{lemma}

\begin{defn}
If $(\Delta,r)\in \HB_r(\Gamma,\vk)$ for some $r$, we define $r(\Delta)$ to be  the supremum of all $r$ for which
$(\Delta,r)\in \HB_r(\Gamma,\vk)$.  
\end{defn}
In particular, $r(\Gamma)=\infty$ if and only if $\Gamma$ is bipartite, . 

If we have chosen an $\ZZ/p^r$--orientation of $(\Delta,\vk)$, we write  $a=(\Delta,r)\in \HB_r(\Gamma,\vk)$, $\fund a =\fund {(\Delta,\vk)}$ and $\funddivZ a=\funddivZ {(\Delta,\vk)}$. These chains will depend on the $\ZZ/p^r$--orientation of $\Delta$, but we will not emphasize that in the notation. We will get back to how we pick the orientation of subgraphs of $\Gamma$ in a systematic way.

Consider one of the components $(\Delta,\vk) \in \red{p^r}(\Gamma,\vk)$. Suppose that it is $\ZZ/p^r$--oriented. If we truncate $\Delta$ further, we obtain a graph $\red{p^{r-1}}(\Delta,\vk) \subset \red{p^{r-1}}(\Gamma,\vk)$. This graph is not necessarily connected. On the other hand, it is bipartite, also in the case $p=2$, and the orientation of $(\Delta,\vk)$ determines a bipartitioning of $\red{p^{r-1}}(\Delta,\vk)$ (lemma~\ref{lemma.pn.orientation} respectively lemma~\ref{lemma.two.orientation}). In particular it determines an orientation of $\red{p^{r-1}}(\Delta,\vk)$. It follows that 
 to every element $a\in \HB_n(\Gamma,\vk)$ we can associate the set $\Phi(a)\subset \HB_{n-1}(\Gamma,\vk)$ of the components of $a$, and that an orientation of $a$ induces an orientation on each member of $\Phi(a)$, 
Note also that each component $\Delta$ of $\red{p^{r-1}}(\Gamma,\vk)$ will be contained in a unique component of $\red{p^{r}}(\Gamma,\vk)$, but even if $\Delta$ is $\ZZ/p^{r-1}$--oriented, this component might not be $\ZZ/p^{r}$ oriented.

We now consider the subgraphs $\Delta$ which occur as oriented components of some
$\red{p^r}(\Gamma)$. These graphs form a partially ordered set by inclusion. In addition to this structure,
given a subgraph $\Delta\subset \Gamma$ we also want to keep track of for which numbers $r$
$\Delta$ is actually a component in $\red{p^r}(\Gamma)$. We collect this information in a graph.

\begin{defn}
  The fundamental forest $F(\Gamma,\vk)$ is the directed graph whose set of vertices is $\HB_*(\Gamma,\vk):=\coprod_r \HB_r(\Gamma,\vk)$. It has an edge going from $(\Delta,r)$ to $(\Delta',r')$ if and only if the graph $\Delta$ of is a subgraph of the graph $\Delta'$ and $r\leq r'$.
The set of fundamental subgraphs $\FS(\Gamma,\vk)$ is the set of all graphs $\Delta\subset \Gamma$ such that for some $r$, there is an element
$x=(\Delta,r)\in \HB_r(\Gamma,\vk)$.
\end{defn}

\begin{remark}
  \label{odd.fact}
Given the graph $\Gamma$, the fundamental forest depends on the prime, but only on whether the prime is even or odd. If necessary, we will distinguish the two cases by referring to the odd respectively the even fundamental forest.
\end{remark}
The fundamental forest  of a connected graph $\Gamma$ has a single component if and only if  $(\Gamma,\vk)$ is itself $\ZZ/p^r$ oriented for some $r$. It has infinitely many vertices if and only if $\Gamma$ is bipartite. In this case, $\HB_r(\Gamma,\vk)$ for all sufficiently large $r$ consists of the single element $(\Gamma,r)$.
In this case, the fundamental forest is a single, infinite tree. This tree contains the infinitely many vertices of the form $(\Gamma,r)$ and finitely many vertices which are not of this form,
If $\Gamma$ is not bipartite there are finitely many vertices in the fundamental forest. Each component of the fundamental forest is a tree, containing a unique maximal vertex. 

We now specify the choices of orientations. Assume first that $\Gamma$ is connected.
If $\Gamma$ is bipartitioned, we choose a bipartitioning of $\Gamma$. This induces a bipartitioning and therefore an orientation on all vertices. 
If $\Gamma$ is not bipartitioned, every tree in the fundamental forest  has a maximal element $(\Delta,\vk)$.  We choose an orientation for each of these maximal elements. For any vertex $a$ in the fundamental forest, we chose the induced orientation, defined by restriction from the unique maximal element bigger than $a$.
If $\Gamma$ is not connected, we use the above methods on every component of $\Gamma$, and end up with an
orientation on every $\Delta\subset \Gamma$ such that $(\Delta,r)\in \HB_r(\Gamma,\vk)$ for some $r$.

\begin{defn}
Let $\{\Gamma_j\}_{j\in J}$ be the set of bipartite components of $\Gamma$. We define
$\FS(\Gamma,\vk)^f=\{\Delta\in \FS(\Gamma,\vk)\mid \Delta\not=\Gamma_j\text{ for all $j\in J$\}}$.   
\end{defn}

For any $p$, we see that $\Delta\in \FS(\Gamma,\vk)^f$ if and only if
there is a strictly positive, finite number of choices of $r$ such that
$(\Delta,r)\in \HB_r(\Gamma,\vk)$. 

We will now discuss the structure of the fundamental forest. We will introduce a number of definitions. In order to understand the definitions, it might be helpful to compare them to the examples at the end of the section.

There is an obvious surjective map $P:\HB_*(\Gamma, \vk)\to \FS(\Gamma,\vk)$ given as $P(\Delta,r)=\Delta$.  
The direction of the directed graph $F(\Gamma,\vk)$ induces a partial order on the set of vertices $\HB_*(\Gamma,\vk)$ of the fundamental forest.
The map $P$ is an order preserving map to the fundamental subgraphs $\FS(\Gamma,\vk)$ ordered by inclusion. 

A maximal vertex in the fundamental forest is a pair $(\Delta,r)$ satisfying that the component of $\red{p^{r+1}}(\Gamma,\vk)$ containing $\Delta$ is not $\ZZ/p^{r+1}$ -- oriented.
Let $\Min *(\Gamma,\vk)\subset \HB_*(\Gamma,\vk)$ be the set of minimal vertices.
Using lemma~\ref{H.characterization} we see that a minimal vertices is given either by a one vertex subgraphs $\Delta[v]$ as $a =(\Delta[v],k_v+1)$, or as $(\Delta,1)$ where
all vertices of $\Delta$ have weight 0. Equivalently, a minimal vertex is a vertex of the fundamental forest of the form $(\Delta,m(\Delta)+1)$.

For $\Delta\in \FS(\Gamma,\vk)$ the set $P^{-1}(\Delta)$ is totally ordered. There is a unique minimal vertex $(\Delta,r^L(\Delta)+1)$, where $r^L(\Delta)$  is the largest weight of an edge in $\Delta$. If $\Delta\in \FS^0(\Gamma,\vk)$ there is also a maximal vertex in $P^{-1}(\Delta)$, and it is  $(\Delta,r(\Delta))$. But these vertices are not necessary maximal and minimal vertices in  $\HB_*(\Gamma,\vk)$. 

\begin{defn}
  For $\Delta\in \HB_r(\Gamma,\vk)$ let
\[
  \Phi((\Delta,r))=\{(\Omega,r-1)\mid (\Omega,r-1)\in \HB_{r-1}(\Gamma,\vk), \Omega\subset \Delta\}
\]
For $\Delta\in \FS(\Gamma,\vk)$, we define $\Phi(\Delta)=\{\Omega\mid (\Omega,r^L(\Delta))\in \Phi((\Delta,r^L(\Delta)+1))$
\end{defn}
\begin{remark}
  \label{rL}
  If $\Delta\not\in P(\Min *(\Gamma,\vk))$, the cardinality of $P^{-1}(\Delta)$ is $r(\Delta)-r^L(\Delta)$. For every $\Omega\in \Phi(\Delta)$ we have that $r(\Omega)=r^L(\Delta)$.
  On the other hand, if $\Delta\in P(\Min *(\Gamma,\vk))$, the cardinality of $P^{-1}(\Delta)$ is $r(\Delta)-m(\Delta)$.
\end{remark}

So far everything has been rather natural. In order to make certain computations later, it will later be useful to introduce a few choices, and some more notation. We discuss how to make these choices, but for the moment we are not giving any motivation for them.

For each element $(\Delta,r)\in \HB_r(\Gamma,\vk)\setminus \Min *(\Gamma,k)$ we choose
an element $s((\Delta,r))=(\Omega,r-1)\in \Phi((\Delta,r))$ such that
$m(\Omega)=m(\Delta)$. This amounts to choosing a vertex $v\in V(\Delta)$ such that
$k_v=\min_{w\in V(\Delta)} k_w = m(\Delta)$, and letting $\Omega$ be the component of
$\red {r-1}(\Delta)$ containing $v$.
This gives a map $s:\HB_*(\Gamma,\vk)\setminus \Min *(\Gamma,k)\to \HB_r(\Gamma,\vk)$.
We also define a map $s:\FS(\Gamma,\vk)\to \FS(\Gamma,\vk)$ by
$s(\Delta)=P(s(\Delta,r^L(\Delta)+1))$.

Given $a\in \HB_r(\Gamma,\vk)$ we can apply $s$ repeatedly on it until we hit an element of $\Min *(\Gamma,\vk)$. That is, there is some $m_a$ such that $s^{m_a}a\in \Min *(\Gamma,\vk)$. This defines a retraction
$B:\HB_*(\Gamma,\vk)\to \Min *(\Gamma,\vk)$ by $B(a)=s^{m_a}a$.
\begin{lemma}
  \label{BandP}
  The map $B$ factors over $P$, so that there is 
  commutative diagram
  \begin{equation}
    \label{B.definition}
  \begin{CD}
    \HB_*(\Gamma,\vk)@>B>> \Min *(\Gamma,\vk)\\
    @VPVV @VPV{\cong}V\\
    \FS(\Gamma,\vk) @>B>> P(\Min *(\Gamma,k))
  \end{CD}
\end{equation}
Moreover, $V(B(\Delta))\subset V(\Delta)$ and $m(B(\Delta))=m(\Delta)$.
\begin{proof}
  If $P(a)=P(b)$, then either $a=s^k(b)$ or $b=s^k(a)$. In either case, $s^{m_a}(a)=s^{m_b}(b)$, so that
$B(a)=B(b)$, and the lower horizontal map $B$ in the diagram is well defined. The last sentence of the lemma follows directly from the definitions of $B$ and $s$. 

To prove that  $P:\Min * (\Gamma,\vk) \to P(\Min * (\Gamma,\vk) )$ is a bijection 
we only have to prove that the map is injective. But this follows from that 
if $a=(\Delta,r)\in \Min *(\Gamma,\vk)$, then $r=m(\Delta)+1$, so that $a$ is determined by
$\Delta=P(a)$.   
\end{proof}
\end{lemma}

Let $\{\Gamma_j\}$ be the set of bipartite components of $\Gamma$.
\begin{defn}
$P(\Min *(\Gamma,\vk))^0=P(\Min *(\Gamma,\vk))\setminus \{B(\Gamma_j)\}$.
We pull this definition back over diagram~\ref{B.definition} to define
\begin{align*}
\FS_*(\Gamma,\vk)^0&=B^{-1}(P(\Min *(\Gamma,\vk))^0),\\
  \Min *(\Gamma,\vk)^0&=P^{-1}(P(\Min *(\Gamma,\vk))^0),\\
  \HB_*(\Gamma,\vk)^0&=B^{-1}(P(\Min *(\Gamma,\vk))^0).
\end{align*}
\end{defn}
We see that
\[
\HB_*(\Gamma,\vk)= \cup_{a\in \Min *(\Gamma,\vk)}B^{-1}(a),
\]
and also that the set $(B\circ P)^{-1}(a)$ is a finite set if and only if
$a\in \Min *(\Gamma,\vk)^0$. The union
$\cup_{a\in \Min *(\Gamma,\vk)^0}(B\circ P)^{-1}(a)$ equals $\HB_*(\Gamma,\vk)^0$, and its complement
$\HB_*(\Gamma,\vk)\setminus \HB_*(\Gamma,\vk)^0$ equals the intersection
$\cap_i \mathrm{Im}(s^i:\HB_{*+i}(\Gamma,\vk)\to \HB_{*}(\Gamma,\vk))$.
Similarly, $\FS(\Gamma,\vk)\setminus \FS(\Gamma,\vk)^0$ equals
the set of all subgraphs $s^i(\Delta)$ where $\Delta$ ranges over the bipartite components of $\Gamma$.
\begin{defn}
If $a\in \Min *(\Gamma,\vk)^0$, we let
$T(a)=(\Delta,r)$ be the element with maximal $r$ such that
$a=B((\Delta,r))$.
Let $\Max *(\Gamma,\vk)^0=T(\Min *(\Gamma,\vk)^0))$. 
\end{defn}
We see that for $a\in \Min *(\Gamma,\vk)^0$, the set $B^{-1}(a)$ consists of the classes
$\{s^n(T(a))\}$. The restricted map $B:\Max *(\Gamma,\vk)^0\to \Min *(\Gamma,\vk)^0$ is a bijection, with inverse $T$.
\begin{remark}
  \label{inclusions}
  We also note that by lemma~\ref{BandP} the restriction $P:\Min *(\Gamma,\vk)\to \FS(\Gamma,\vk)$ is injective.  We also note that  $P(\Min *(\Gamma,\vk)^0)\subset \FS(\Gamma,\vk)^0$.
\end{remark}
\begin{lemma}
  \label{NotMax}
  If $(\Delta,r)\in \HB_r(\Gamma,\vk)$ but $(\Delta,r)\not\in \Max r(\Gamma,\vk)^0$, then there is a
  $(\Psi,r+1)\in \HB_{r+1}(\Gamma,\vk)$ such that $\Delta\in \Phi(\Psi)$, and if
  $\Omega\in \Phi(\Psi)\setminus \{\Delta\}$, then $(\Omega, r)\in \Max r(\Gamma,\vk)^0$.
  \begin{proof}
    Since $(\Delta,r)\not\in \Max r(\Gamma,\vk)^0$, there is a $(\Psi,r+1)\in \HB_{r+1}(\Gamma,\vk)$
    such that $s((\Psi,r+1))=(\Delta,r)$. This means that $\Delta$ is a component of
    $\red r (\Psi,\vk)$, that is, $\Delta\in \Phi(\Psi)$. The other elements $\Omega\in \Phi(\Psi)$ are not
    in the image of $s$, therefore they are in $\Max r (\Gamma,\vk)^0$.
  \end{proof}
\end{lemma}

The final choice we want to do is a map  we will use to prove an injectivity statement later. 
\begin{defn} 
  The witness map is a map $w:B(\Min * (\Gamma,\vk))\to V(\Gamma)$ which for each 
$\Delta \in B(\Min * (\Gamma,\vk))$ chooses a vertex $w(\Delta)\in V(\Delta)$ such that
$k_{w(\Delta)}=m(\Delta)$.
\end{defn}

\begin{remark}
  \label{witness}
   By the definition of $B$, we easily see that for any $\Delta\in \FS(\Gamma,\vk)$ the vertex
   $w(B(\Delta))$ is a vertex in $\Delta$ such that $k_v=m(\Delta)$.  In this case we also say that $w(B(\Delta))$ is a witness for $\Delta$.  In particular, this is true if $\Delta\in \Max *(\Gamma,\vk)^0$. 
  Since the elements of $B(\Min *(\Gamma,\vk))$ are disjoint, the witness map restricts to an injective map $w:\Max *(\Gamma,\vk)^0\to V(\Gamma)$. Moreover, if $\Gamma$ is bipartite, the witness of $B(\Gamma)$ is not contained int $w(\Max *(\Gamma,\vk)^0)$. 
\end{remark}
\newpage
We draw two examples in order to explain the definitions above.

\includegraphics[height=4cm]{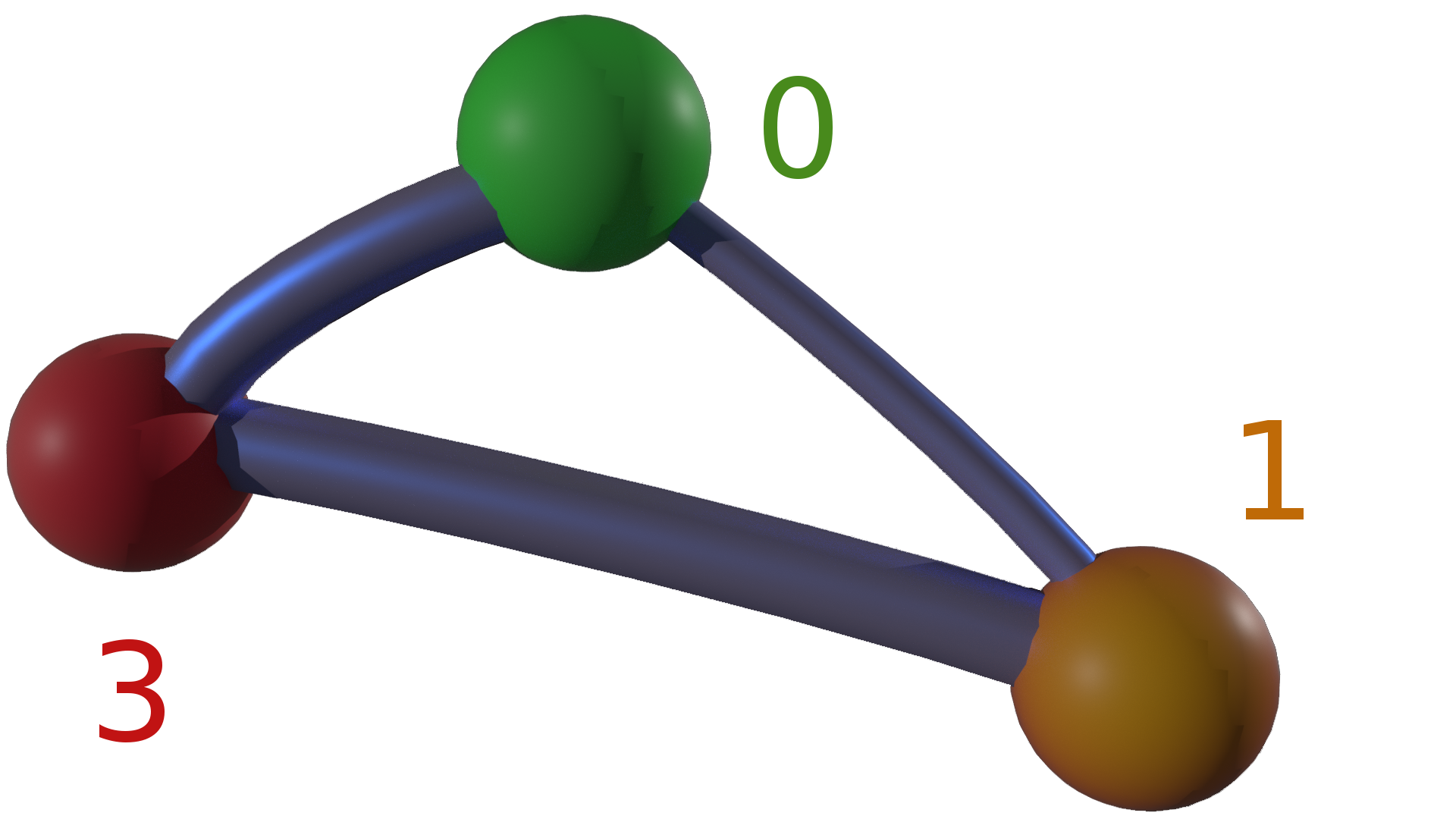}
\includegraphics[height=4cm]{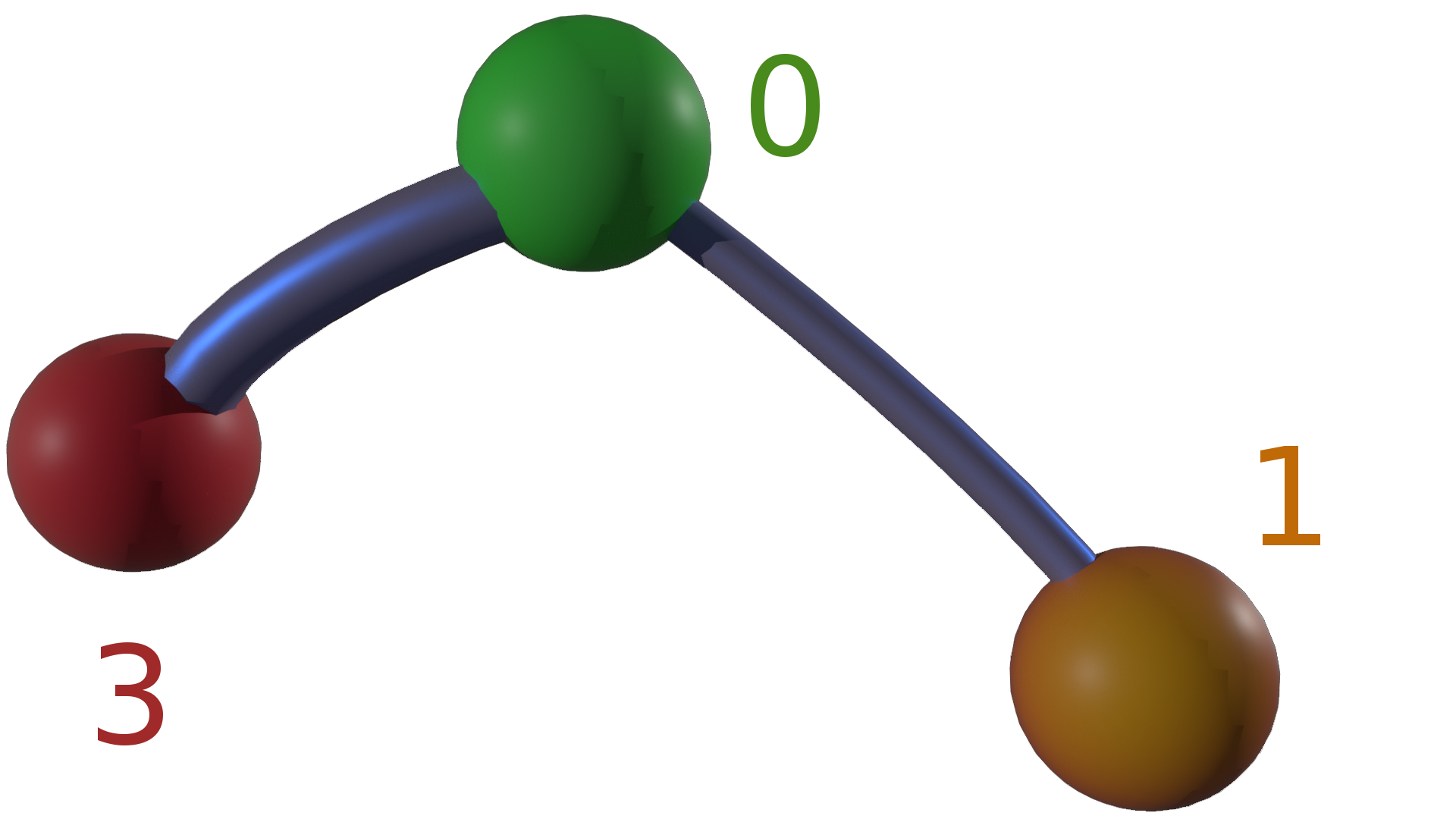}
\begin{tikzpicture}
  \begin{scope}[every node/.style={circle,thick,draw}]
    \node (TOP) at (2,0) {(RGB,4)};
    \node (RED) at (0,-2) {(R,3)};
    \node (GREENBROWN_TOP) at (4,-2){(GB,3)};
    \node (GREENBROWN_BOTTOM) at (4,-5){(GB,2)};
    \node (GREEN) at (6,-7){(G,1)};
  \end{scope}
\path[red, very thick] (TOP) edge node {}  (RED);  
  \begin{scope}[>={Stealth[black]},
    every node/.style={fill=white,circle},
    every edge/.style={draw=red,very thick}
    ]
    \path[->] (TOP) edge node {$s$} (GREENBROWN_TOP);
    \path[->] (GREENBROWN_TOP) edge node {$s$} (GREENBROWN_BOTTOM);
    \path[->] (GREENBROWN_BOTTOM) edge node {$s$} (GREEN);
  \end{scope}
  \begin{scope}[every node/.style={circle,thick,draw}]
    \node (OVER_TOP) at (8,3) {(RGB,5)};
    \node (TOP) at (8,0) {(RGB,4)};    
    \node (RED) at (6,-2) {(R,3)};
    \node (GREENBROWN_TOP) at (10,-2){(GB,3)};
    \node (GREENBROWN_BOTTOM) at (10,-5){(GB,2)};
    \node (GREEN) at (12,-7){(G,1)};
  \end{scope}
\path[red, very thick] (TOP) edge node {}  (RED);  
  \begin{scope}[>={Stealth[black]},
    every node/.style={fill=white,circle},
    every edge/.style={draw=red,very thick}
    ]
    \path[->] (OVER_TOP) edge node {$s$} (TOP);
    \path[->] (TOP) edge node {$s$} (GREENBROWN_TOP);
    \path[->] (GREENBROWN_TOP) edge node {$s$} (GREENBROWN_BOTTOM);
    \path[->] (GREENBROWN_BOTTOM) edge node {$s$} (GREEN);
  \end{scope}  
\end{tikzpicture}

The left graph $\Gamma_0$ is the complete graph on the vertices $R,G,B$. It is not bipartite. The picture shows the odd fundamental forest, corresponding to the weights $v_R=p^3, v_G=1,v_B=p$ for an odd prime $p$. The fundamental forest contains a single tree containing five vertices. There is only one possible choice for $s$, which is shown in the picture above. The set $\HB_*(\Gamma_0,\vk)^0$ consists of all the vertices of the fundamental forest. $\Min *(\Gamma,\vk)$ is the set $\{(R,3),(G;3)\}$. and  $\Max *(\Gamma_0,\vk)^0$ is the set $\{(R,3),(RGB;4)\}$. The map
$T$ is given by $T(R,3)=(R,3)$ and $T(G,1)=(RGB,4)$. 
The witness for $(R,3)$ is $R$, and the witness for $(RGB,4)$ is $G$.

The right graph $\Gamma_1$ has the same vertices as $\Gamma_0$. We use the same vertex weights $\vk$. But this graph is bipartite, and the fundamental forest has infinitely many vertices of the form $(\Gamma_0,r)$. Obviously, not all of those are shown in the picture.   The set $\Min *(\Gamma_1,\vk)$ is still $\{(R,3),(G,3)\}$, but the set $\HB_*(\Gamma_1,\vk)^0$ only consists of $(R,3)$. The set $\Min *(\Gamma_1,\vk)^0$ is now $\{(R,3)\}$, and $T(R,3)=(R,3)$, so that $\Max *(\Gamma_0,\vk)^0$ is $\{(R,3)\}$.  The witness of $B(R)$ is different from the witness of $B(\Gamma_1)$.

\subsection{The fundamental chain complex}

The fundamental class of $x\in \HB_r(\Gamma,\vk)$  defines a cochain $\fund x :=\fund{\Delta(x)}\in C^0(\Gamma,\vk)$. 
For any $x$, we have that
\[
\fund x =\sum_{y\in \Phi(x)} \fund y.
\]
 
To each vertex $x=(\Delta,r)$ in the fundamental forest we associate a weight 
$m(x)$, which is defined as the $p$-valuation of the greatest common
divisor of $k_v$ for $v\in V(\Delta(x))$. If $y$ is another vertex which is smaller than $x$ in the partial ordering, then $m(x)\leq m(y)$. Similarly, we define $\funddivZ{x}$ so that
$p^{m(x)}\funddivZ{x}=\fund{x}$. 

For the rest of this section, we will fix a prime $p$. We will be concerned with the $p$-primary part of 
the torsion in $H^1(\Gamma,\vk)$. This means that we can as well assume that each weight is a power of $p$, and we will occasionally
write the weights as $\vk=p^{\lao k}=\{p^{{\la kv}}\}_{v\in V(\Gamma)}$, where  each $\la kv$ is a non-negative integer.
In particular, all of the the torsion in the graph cohomology groups is $p$ primary torsion. 

The strategy is to define a chain complex $\FK^*(\Gamma,\vk)$ that only depends on $p$ and the fundamental forest. 
Then we will show that the cohomology $H^1(\FK(\Gamma,\vk))$ equals the torsion of the
graph cohomology.

The cochain groups $\FK^*(\Gamma,\vk)$ are free Abelian groups generated by certain symbols. 

  The chain complex $\FK^{*}(\Gamma,\vk)$ is concentrated in the degrees
$-1,0,1$. There is a case distinction between the case of a bipartite component $\Delta\subset \Gamma$, and the case that
$\Delta$ is not a bipartite component, that is $\Delta\in \FS(\Gamma,\vk)^f$.
Recall that if $\Phi(\Delta)$ is non-empty, the number  $r^L(\Delta)$ is the common $r(\Omega)$ for all $\Omega\in \Phi(\Delta)$.
Since $r(\Delta)> r(\Omega) > m(\Delta)$, we have that
$r(\Delta)> r^L(\Delta)>m(\Delta)$. 
\begin{itemize}
\item $\FK^{-1}(\Gamma,\vk)$ is freely generated by symbols $\rho_{-1}(\Delta)$ for $\Delta\in \FS(\Gamma,\vk)\setminus P(\Min *(\Gamma,\vk))$.
\item $\FK^{0}(\Gamma,\vk)$ is freely generated by symbols $\rho_{0}(\Delta)$ for $\Delta\in \FS(\Gamma,\vk)\setminus P(\Min *(\Gamma,\vk))$, 
  together with symbols $\alpha_0(\Delta)$ for $\Delta\in \FS(\Gamma,\vk)$.
\item $\FK^{1}(\Gamma,\vk)$ is generated by symbols 
  $\alpha_1(\Delta)$ for all $\Delta\in \FS(\Gamma,\vk)$. If
  $\Delta\not\in \FS(\Gamma,\vk)^f$, we 
  divide out by the relation  $\alpha_1(\Delta)=0$.
\end{itemize}
Here, the injection $P:\Min *(\Gamma,\vk)\to \FS(\Gamma,\vk)$ is as in remark~\ref{inclusions}.
The boundary maps are given by

\begin{align*}
  d(\rho_{-1}(\Delta))&=\alpha_0(\Delta)-p^{r^L(\Delta)-m(\Delta)}\rho_0(\Delta)-\sum_{\Omega\in \Phi(\Delta)} p^{m(\Omega)-m(\Delta)}\alpha_0(\Omega),\\
 d(\rho_0(\Delta)&=p^{r(\Delta)-r^L(\Delta)}\alpha_1(\Delta)-\sum_{\Omega\in \Phi(\Delta)} \alpha_1(\Omega),\\
  d(\alpha_0(\Delta))&=   p^{r(\Delta)-m(\Delta)}\alpha_1(\Delta)\\
 d(\alpha_1(\Delta))&=0.   
\end{align*}
We check that we have defined a co-chain complex:

\begin{align*}
  dd(\rho_{-1}\Delta)&=d\left( \alpha_0(\Delta)-p^{r^L(\Delta)-m(\Delta)}\rho_0(\Delta)-\sum_{\Omega\in \Phi(\Delta)} p^{m(\Omega)-m(\Delta)}\alpha_0(\Omega)\right)\\
&=p^{r(\Delta)-m(\Delta)}\alpha_1(\Delta)-p^{r(\Delta)-m(\Delta)}\alpha_1(\Delta)+\sum_{\Omega\in \Phi(\Delta)} p^{r(\Omega) -m(\Delta)}\alpha_1(\Omega)\\
&\phantom{X}-\sum_{\Omega\in \Phi(\Delta)} p^{r(\Omega)-m(\Delta)}\alpha_1(\Omega)\\
&=0.
\end{align*}

\begin{defn}
  The fundamental complex is the complex $\FK^*(\Gamma,\vk)$ defined above. 
\end{defn}

  
We are ready to compute the cohomology of the complex $\FK^*(\Gamma,\vk)$.

\begin{lemma}
  \label{order.HFK}
Assume that $\Gamma$ is connected.   
The only non-trivial cohomology groups of   $\FK^{*}(\Gamma,\vk)$ are in degrees 0 and 1.
\[
  H^0(\FK^{*}(\Gamma,\vk))=
  \begin{cases}
    \ZZ&\text{ with generator $\alpha_0(\Gamma)$ if $\Gamma$ is bipartite,}\\
    0&\text{ if $\Gamma$ is not bipartite}.\\
  \end{cases}
\]
The order of the group $H^1(\FK^{*}(\Gamma,\vk))$ is $p^N$ where $N$ is the number of
elements in the set $\HB_*(\Gamma,\vk)^0$. 
 \begin{proof}

  Let $A^*\subset \FK^{*}(\Gamma,\vk)$ be the subcomplex generated by
  the classes $\alpha_0(\Delta)$ and $\alpha_1(\Delta)$ for
  $\Delta\in \FS(\Gamma,\vk)$. 
The cohomology group $H^0(A^*)$  is 0 if $\Gamma$ is not bipartite, and isomorphic to $\ZZ$ generated by
$\alpha_0(\Gamma)$ if  $\Gamma$ is bipartite. 
The cohomology  $H^1(A^*)$ 
is a direct sum of cyclic groups, indexed by
$\Delta\in \FS(\Gamma,\vk)^f$. The summand indexed by $\Delta$ is isomorphic to $\ZZ/p^{r(\Delta)-m(\Delta)}$, generated by
$\alpha_1(\Delta)$. 

  There is a short exact sequence of
chain complexes
\[
0 \to A^*\to \FK^{*}(\Gamma,\vk) \xrightarrow{\pi} R^*\to 0.
\]
$R^*$ is freely generated by the classes $\pi(\rho_{-1}(\Delta))$ and $\pi(\rho_0(\Delta))$ for $\Delta\in \FS(\Gamma,\vk)\setminus P(\Min *(\Gamma,\vk))$.
The differential in $R^*$ is given by $d(\pi(\rho_{-1}(\Delta)))=-p^{r^L(\Delta)-r(\Delta)}\pi(\rho_0(\Delta))$.
The cohomology of $R^*$ is concentrated in degree 0. It is a direct sum of cyclic groups
indexed by $\Delta\in \FS(\Gamma,\vk)\setminus P(\Min *(\Gamma,\vk))$. The summand indexed by $\Delta$ is isomorphic to $\ZZ/p^{r(\Delta)-m(\Delta)}$, generated by
$\rho_0(\Delta)$. 
The long exact sequence of cohomology groups takes the form
\[
0\to H^0(A^*) \to H^0(\FK^{*}(\Gamma,\vk)) \xrightarrow{\pi^*} H^0(R^*) \xrightarrow\delta H^1(A^*) \to  H^1(\FK^{*}(\Gamma,\vk)) \to 0.
\]  
We see that $H^{-1}(\FK^{*}(\Gamma,\vk))\cong 0$.

We can now identify the boundary map $\delta$ as the following map:
\begin{align*}
  \delta:\bigoplus_{\Delta\in \FS(\Gamma,\vk)\setminus P(\Min *(\Gamma,\vk))}& \ZZ/p^{r^L(\Delta)-m(\Delta)} \to \bigoplus_{\Delta\in \FS(\Gamma,\vk)^f} \ZZ/p^{r(\Delta)-m(\Delta)}\\ 
      \delta(\pi(\rho_0(\Delta))) \quad&=\quad p^{r(\Delta)-r^L(\Delta)}\alpha_1(\Delta) - \sum_{\Omega\in \Phi(\Delta)} \alpha_1(\Omega). 
\end{align*}
As before, if $\Delta=\Gamma$ we interpret the term $p^{r(\Delta)-r^L(\Delta)}\alpha_1(\Delta)$ as 0.

We claim that the map $\delta$ is injective. To show this, for each
$\Delta\in \FS(\Gamma,\vk)\setminus P(\Min *(\Gamma,\vk))$
we choose
$s(\Delta)\in \Phi(\Delta)$ such that $m(s(\Delta))=m(\Delta)$.
Then $s(\Delta)\in \FS(\Gamma,\vk)^0$, and $r(s(\Delta))=r^L(\Delta)$.
Order the set $\FS(\Gamma,\vk)\setminus P(\Min *(\Gamma,\vk))$
so that $\Delta > s(\Delta)$. 
The projection of $\delta(\pi(\rho_0(\Delta)))$ to the component
indexed by $s(\Omega)$ is
$0$ if $\Delta<\Omega$
, and
$-1\in \ZZ/p^{r(s(\Omega))-m(s(\Omega))}=\ZZ/p^{r^L(\Delta)-m(\Delta)}$  if $\Delta=\Omega$.
A standard filtration argument now shows the injectivity of $\delta$. 

We conclude from this and from the exact sequence above  that the map
$H^0(A^*)\to H^0(\FK^{*}(\Gamma,\vk))$ is an isomorphism. The statements about
$H^0(\FK^{*}(\Gamma,\vk))$ follows from this.

It also follows from the exact sequence that $H^1(\FK^{*}(\Gamma,\vk))$
is isomorphic to the
cokernel of $\delta$. 
Since $\delta$ is injective, the cardinality of the group $\mathrm{coker}(\delta)$ is the quotient of the
cardinality of the target of $\delta$ and the cardinality of the source of $\delta$.
That is, the cardinality of $\mathrm{coker}(\delta)$ is $p^N$
where

\begin{align*}
  N&=\sum_{\Delta\in \FS(\Gamma,\vk)^f}(r(\Delta)-m(\Delta))
     -\sum_{\Delta\in \FS(\Gamma,\vk)\setminus P(\Min *(\Gamma,\vk))}{(r^L(\Delta)-m(\Delta))} 
\end{align*}
Now we notice that $\FS(\Gamma,\vk)^0\subset\FS(\Gamma,\vk)^f$
and that $\FS(\Gamma,\vk)\setminus P(\Min *(\Gamma,\vk))$ is the disjoint union of
$\FS(\Gamma,\vk)^0\setminus P(\Min *(\Gamma,\vk)^0)$ and
\[
  X=\left(\FS(\Gamma,\vk)\setminus P(\Min *(\Gamma,\vk))\right)\setminus
  \left(\FS(\Gamma,\vk)^0\setminus P(\Min *(\Gamma,\vk))^0\right)
\]
We can rewrite the sum above as $N=A+B$ where
\begin{align*}
  A&=\sum_{\Delta\in \FS(\Gamma,\vk)^0\setminus P(\Min *(\Gamma,\vk)^0)}{(r(\Delta)-r^L(\Delta))}
       +\sum_{\Delta\in P(\Min *(\Gamma,\vk)^0)}(r(\Delta)-m(\Delta))\\
    B&=\sum_{\Delta\in \FS(\Gamma,\vk)^f\setminus \FS(\Gamma,\vk)^0}(r(\Delta)-m(\Delta))
     -\sum_{\Delta\in X}{r^L(\Delta)-m(\Delta)}  
\end{align*}
There map $s$ restricts to a bijection $s:X\to \FS(\Gamma,\vk)^f\setminus \FS(\Gamma,\vk)^0$. Since
$r^L(\Delta)=r(s(\Delta))$ and $m(s(\Delta))=m(\Delta)$, this shows that $B=0$.

Using remark~\ref{rL} we see that $A$ is the cardinality of
$\cup_{\Delta\in \FS(\Gamma,\vk)^0} P^{-1}(\Delta)=\HB_*(\Gamma,\vk)^0$.
This completes the proof.
\end{proof}
\end{lemma}

We can also give a description of the group $H^1(\FK(\Gamma,\vk)$ up to isomorphism.

We define a map
\[
f:\bigoplus_{\Delta\in P(\Min *(\Gamma,\vk)^0)} \ZZ/p^{r(\Delta)-m(\Delta)} \to \bigoplus_{\Delta\in \FS(\Gamma,\vk)^0} \ZZ/p^{r(\Delta)-m(\Delta)} \to \mathrm{coker}(\delta).
\]
as follows. Recall that for every $\Delta\in \FS(\Gamma,\vk)$, the element $TP(\Delta)\in \Max *(\Gamma,\vk)$ is the maximal element in $\HB_*(\Gamma,\vk)$ such that $TP(\Delta)\in B^{-1}(P(\Delta))$. Let $\{n_\Delta\}\in \bigoplus_{\Delta\in P(\Min *(\Gamma,\vk)^0)} \ZZ/p^{r(\Delta)-m(\Delta)}$ be the vector with components $n_\Delta$. Then
\[
  f(\{n_\Delta\})=\sum_{\Delta\in P(\Min *(\Gamma,\vk)^0)} n_{\Delta} \alpha_1(TP(\Delta)).
\]
Note that every element $\alpha_1(\Delta)$ for $\Delta\in P(\Max *(\Gamma,\vk))$ is in the image of $f$.
\begin{lemma} The map $f$ is an isomorphism.
  \label{computation.HFK}
  \begin{proof}
The cardinality of $\bigoplus_{\Delta\in P(\Min *(\Gamma,\vk)^0)} \ZZ/p^{r(\Delta)-m(\Delta)}$ is
$p^M$ where
$M=\sum_{\Delta\in P(\Min *(\Gamma,\ vk))}r(\Delta)-m(\Delta)$.
But
$r(\Delta)-m(\Delta)$ is also the cardinality of $B^{-1}(B(\Delta))$, so that
\[    
  M=\sum_{\Delta\in \Max *(\Gamma,\vk)}\mid B^{-1}(B(\Delta))\mid=
  \mid
    B^{-1} \Min *(\Gamma,\vk)
  \mid
  =\mid \FS(\Gamma,\vk)^0\mid
\]
It follows this and from lemma \ref{order.HFK} that the source and the target of $f$ have the same order. In order to prove that $f$ is an isomorphism, it suffices to show that $f$ is surjective.

  Assume inductively that the image of the composite
  \[
   \bigoplus_{\Delta\in \FS(\Gamma,\vk)^0\mid r(\Delta)>n} \ZZ/p^{r(\Delta)-m(\Delta)}\subset \bigoplus_{\Delta\in \FS(\Gamma,\vk)^0} \ZZ/p^{r(\Delta)-m(\Delta)}\to \mathrm{coker}(\delta)
 \]
 is contained in the image of $f$. We have to show the downwards induction step that if $\Delta\in \FS(\Gamma,\vk)^0$ and $r(\Delta)= n$,
 then the image of $\alpha_1(\Delta)$ in $\mathrm{coker}(\Delta)$ is also contained in the image of $f$. 
The induction start is trivial since $\FS(\Gamma,\vk)^0$ is a finite set.
We have already noticed that if   $\Delta\in \Max *(\Gamma,\vk)^0\subset \FS(\Gamma,\vk)^0$ then $\alpha_1(\Delta)$ is in the image of $f$.
 If $\Delta\not\in \Max *(\Gamma,\vk)$, we use lemma~\ref{NotMax}. We can write $(\Delta,n)=s(x)$ for some
 $(\Psi,n+1)\in \HB_{n+1}(\Gamma,\vk)$, and if $\Omega\in \Phi(\Psi)$ but $\Omega\not=\Delta$, then $\Omega\in \Max *(\Gamma,\vk)$.
    In $\mathrm{coker}(\delta)$ we have the relation
 \[
   p^{r(\Delta)-r^L(\Delta)}[\alpha_1(\Psi)]-\sum_{\Omega\in \Phi(\Psi)\mid \Omega\not= \Delta}[\alpha_1(\Omega)]= [\alpha_1(\Delta)].
 \]
 The left hand side is in the image of $f$ by the induction assumption
and since each $\Omega\in \Max *(\Gamma,\vk)$. 
  \end{proof}  
\end{lemma}

\subsection{The fundamental complex and graph cohomology}
The next step is to use $\FK^*(\Gamma,\vk)$ to study $C^*(\Gamma,\vk)$
We define a map $\chi:\FK^*(\Gamma,\vk)\to C^*(\Gamma,\vk)$ by
$\chi(\rho_{-1}(\Delta))=0$, $\chi(\rho_{0}(\Delta))=0$,
$\chi(\alpha_{0}(\Delta))=\funddivZ{\Delta}$ and $\chi(\alpha_{1}(\Delta))=p^{-r}d\fund{\Delta}$ for $\Delta\in \FS(\Gamma,\vk)$.
If $\Gamma$ is bipartite, we define 
$ \phi(\alpha_0(\Gamma))=\funddivZ \Gamma$.
It is easy to check that $\chi$ is a chain map, inducing a map
$\chi^*$ on cohomology.  

\begin{lemma}
\label{chi}
  $\chi^*:H^0(\FK^*(\Gamma,\vk);\ZZ/p)\to H^0((\Gamma,\vk);\ZZ/p)$ is injective.
If $\Gamma$ is not bipartite,
  $\chi^*:H^0(\FK^*(\Gamma,\vk);\ZZ/p^s)\to H^0((\Gamma,\vk);\ZZ/p^s)$ is surjective for all $s$.
If $\Gamma$ is bipartite,
\[
\chi^*:H^0(\FK^*(\Gamma,\vk);\ZZ/p^s)\to H^0((\Gamma,\vk);\ZZ/p^s)
\]
 is injective for $s=1$ and surjective for all $s$.

  \begin{proof}
    According to the computation of lemma~\ref{computation.HFK},
$H^0(\FK^*(\Gamma,\vk);\ZZ/p)$ is a $\ZZ/p$ -- vector space with a bases
consisting of the classes $\alpha_0(\Delta)$ for $\Delta\in \Max *(\Gamma,\vk)$.
In order to prove injectivity
we have to prove that the classes
$\funddivZ{\Delta}\in C^0(\Gamma,\vk;\ZZ/p)$
for $\Delta\in \Max *(\Gamma,\vk)$ 
are linearly independent.
Let's assume that we have a linear dependence
\[
\sum_{\Delta\in \Max *(\Gamma,\vk)}\lambda_\Delta \funddivZ{\Delta}=0.
\]  
For each $\Delta$ we consider the coefficient of the 
witness $w(\Delta)$ in $\sum_{\Delta\in \Max *(\Gamma,\vk)}\lambda_\Delta \funddivZ{\Delta}$.
According to remark~\ref{witness}, this coefficient is
$\pm \lambda_\Delta$, so that $\lambda_\Delta=0$. This proves the injectivity.
 
The surjectivity is a corollary of theorem~\ref{generation}. According to this theorem, it is 
sufficient to prove that all classes $p^d\funddivZ \Delta\in C^0(\Gamma,\vk;\ZZ/p^s)$ 
for $d\geq s-r(\Delta)+m(\Delta)$ are images of cycles. By the definition of $\chi$, 
$p^d\funddivZ \Delta=\chi(p^d\alpha_0(\Delta))$, so we have to check that
$p^d\alpha_0(\Delta)$ is a cycle. But
$d(p^d\alpha_0(\Delta))=p^{d+r(\Delta)-m(\Delta)}\alpha_1(\Delta)$. Since $s\leq d+r(\Delta)-m(\Delta)$, 
this is indeed trivial modulo $p^s$. This completes the proof of surjectivity.   
\end{proof}
\end{lemma}

\begin{corollary}
  \label{injectivity.chi}
  $\chi^*:H^1(\FK^*(\Gamma,\vk))\to H^1((\Gamma,\vk))$ is injective. It's image is the torsion subgroup of $H^1((\Gamma,\vk))$.
  \begin{proof}
    We can without restriction of the generality assume that $\Gamma$ is connected. We first deal with the injectivity statement. It suffices to show that
    $\chi^*$ is injective on the $p$-torsion subgroup (as opposed to the $p$-primary torsion subgroup).
    There is a map of long exact sequences
\[    
   \begin{CD}
   H^0(\FK^*(\Gamma,\vk))@>>>H^0(\FK^*(\Gamma,\vk);\ZZ/p)@>\beta_F >>H^1(\FK^*(\Gamma,\vk))@>{p\cdot}>>H^1(\FK^*(\Gamma,\vk))\\  
@V{\chi^*} V{\cong}V @V{\chi^*}V{\cong}V @V{\chi^*}VV @V{\chi^*}VV\\
   H^0((\Gamma,\vk))@>>>H^0((\Gamma,\vk);\ZZ/p)@>\beta_C >>H^1((\Gamma,\vk))@>{p\cdot}>>H^1((\Gamma,\vk)).\\  
   \end{CD}
 \]
In this diagram, the left vertical map $\chi^*:H^0(\FK^*(\Gamma,\vk))\to H^0((\Gamma,\vk))$ is an isomorphism because of lemma~\ref{computation.HFK}.
 The middle left vertical map $\chi^*:H^0(\FK^*(\Gamma,\vk);\ZZ/p)\to H^0((\Gamma,\vk);\ZZ/p)$ is an isomorphism by lemma~\ref{chi}.
 
 The $p$-torsion subgroup of $H^1(\FK^*(\Gamma,\vk))$ is the image of $\beta_F$, so we only have to prove that $\chi^*$ is injective on that image. 
Assume first that $\Gamma$ is not bipartite. Then $H^0((\Gamma,\vk))=0$, so that the composition
$\beta_C \circ \chi^*:H^0(\FK^*(\Gamma,\vk);\ZZ/p)\to  H^1((\Gamma,\vk))$ is injective.
It follows that $\chi^*$ is injective in $\mathrm{im}(\beta_F)$. 

If $(\Gamma,\vk)$ is connected and bipartite, it follows from the diagram that
$\mathrm{ker}(\beta_C)=\chi^*(\mathrm{ker}(\beta_F)$. Using this and diagram chasing, we
can argue exactly as in the not bipartite case that
 $\chi^*$ is injective on the image of $\beta_F$. The injectivity statement of the lemma now follows as in the non-bipartite case.

So far we have proved that $\chi^*:H^1(\FK^*(\Gamma,\vk))\to H^1((\Gamma,\vk))$ is injective
with image contained in the torsion subgroup of $H^1((\Gamma,\vk))$. To complete the proof
we also need to show that every $p$--primary torsion element of $H^1((\Gamma,\vk))$ is contained in the image. But every $p$--primary torsion element is in the image of some
Bockstein map $\beta^s_C:H^1((\Gamma,\vk);\ZZ/p^s)\to H^1((\Gamma,\vk))$.

By lemma~\ref{chi} the map
$\chi^*:H^0(\FK^*(\Gamma,\vk);\ZZ/p^s)\to H^0((\Gamma,\vk);\ZZ/p^s)$
is surjective for all $s$,
so every $p$-primary torsion element is in the image of
$\beta^s_C\circ \chi^*$. 
It follows from the commutative diagram
\[
\begin{tikzcd}
  H^0(\FK^*(\Gamma,\vk);\ZZ/p^s)\arrow[r,"\beta_F^s"]
\arrow[d,"\chi^*"]
&H^1(\FK^*(\Gamma,\vk))\arrow[d,"\chi^*"]\\  
 H^0((\Gamma,\vk);\ZZ/p^s)\arrow[r,"\beta_C^s"]&H^1((\Gamma,\vk))\\  
\end{tikzcd}
\]
that every $p$--primary torsion element of $H^1((\Gamma,\vk))$ is in the image of
the map $\chi^*:H^1(\FK^*(\Gamma,\vk))\to H^1((\Gamma,\vk))$
\end{proof}
\end{corollary}

 We can combine this with lemma~\ref{computation.HFK} and obtain:
\begin{thm}
\label{main.theorem}
  For any prime $p$, 
  the $p$-torsion subgroup of $H^1((\Gamma,\vk))$ is isomorphic to
  $\oplus_{\Delta\in \Max *(\Gamma,\vk)}\ZZ/p^{r(\Delta)-m(\Delta)}$.
\end{thm}

We also record the following weaker statement will be useful later.
\begin{corollary}
\label{cor.main.theorem}
  For any prime $p$, 
  The $p$-torsion subgroup of $H^1((\Gamma,\vk))$ has order
  $p^N$ where $N$ is the cardinality of the set of elements of $\HB_*$
  which are not in $\cap_n \mathrm{Im}(s^n)$.
  \begin{proof}
    This follows from theorem~\ref{main.theorem} and lemma \ref{order.HFK}.
  \end{proof}
\end{corollary}

\subsection{Functoriality properties}
\label{functoriality.properties}

The graph cohomology $H^*(\Gamma,\vk)$ is a functor on the category of subgraphs
$\Delta\subset \Gamma$, taken with the induced weighing. The purpose of this section
is to describe $\FK^*(\Gamma,\vk)$ as a functor on the same category,
such that the homomorphism 
$\chi^*:\FK^*(\Gamma,\vk)\to C^*(\Gamma,\vk)$ of corollary~\ref{injectivity.chi} becomes a natural transformation. Since we will now sometimes consider a graph $\Omega$ as subgraph of different supergraphs, 
for $\Omega\subset \Gamma$ we will write $r(\Omega)$ and $r^L(\Omega)$ as $r_\Gamma(\Omega)$ respectively $r^L_\Gamma(\Omega)$ to emphasize that we are computing $r(\Delta)$ with respect to the supergraph $\Gamma$.

Suppose that $j:\Delta\subset \Gamma$ is a subgraph.
If $\Omega$ is a connected, $\ZZ/p^r$--oriented subgraph of $\Gamma$, then
$\Omega\cap \Delta$ is a  $\ZZ/p^r$--oriented subgraph of $\Delta$.
It will not necessarily be connected. We write the set of components of $\Omega\cap \Delta$
as $C(\Omega\cap \Delta)$, so that
$\Omega\cap \Delta = \cup_{\Psi\in C(\Omega\cap \Delta)} \Psi$.
 In this notation, for every $\Psi\in C(\Omega\cap\Delta)$ we have the inequalities
$m(\Omega)\leq m(\Psi)$ and $r_\Delta(\Psi)\geq r_\Gamma(\Omega)$.

\begin{remark}
\label{double.intersection}
Since $\Phi(\Omega)$ are the components of $\red {r(\Omega)} \Omega$, we can
decompose the graph $\left(\cup_{\Theta\in\Phi(\Omega)}\Theta\right)\cap \Delta$ into components as follows:
\[
  \left(\bigcup_{\Theta\in\Phi(\Omega)}\Theta\right)\cap \Delta = (\red{p^{r(\Omega)}} \Omega)\cap \Delta =
  \red{p^{r(\Omega)}} (\Omega \cap \Delta)=
  \bigcup_{\Psi\in C(\Omega\cap \Delta)}\bigcup_{\Lambda\in\Phi(\Psi)}\Lambda
\]    
\end{remark}

We define the restriction map $j^*:\FK^*(\Gamma,\vk)\to \FK^*(\Delta,\vk)$ as follows.

\begin{align*}
  \label{functoriality.F}
  j^*(\rho_{-1}(\Omega))&=\sum_{\Psi\in C(\Omega\cap \Delta)}\rho_{-1}(\Psi),\\
  j^*(\rho_{0}(\Omega))&=\sum_{\Psi\in C(\Omega\cap \Delta)}p^{r^L_\Delta(\Psi)-r^L_\Gamma(\Omega)}\rho_0(\Psi),\\
  j^*(\alpha_{0}(\Omega))&= \sum_{\Psi\in  C(\Omega\cap \Delta)}\alpha_0(\Psi),\\
  j^*(\alpha_{1}(\Omega))&= \sum_{\Psi\in C(\Omega\cap \Delta)}p^{r_\Delta(\Psi)-r_\Gamma(\Omega)}\alpha_1(\Psi).
\end{align*}
In the expression above for $j^*(\alpha_{1}(\Omega))$, it can happen that
$r_\Delta(\Psi)=\infty$, namely if $\Psi$ is a bipartite component of
$\Delta$. In this case, $\alpha_1(\Psi)=0$ and we interpret the term $p^{r_\Delta(\Psi)-r_\Gamma(\Omega)}\alpha_1(\Psi)$ in the sum as 0.

\begin{lemma}
\label{FK.functoriality}  
  The map $j^*$ is a chain map. There is a commutative diagram of chain complexes
\[
  \begin{CD}
      \FK^*(\Gamma,\vk)@>{j^*}>> \FK^*(\Delta,\vk)\\
      @V\chi VV @V\chi VV \\
      C^*(\Gamma,\vk)@>{j^*}>> C^*(\Delta,\vk),
    \end{CD}
    \]
  where the vertical maps are the weak equivalences.
  \begin{proof}
    That $j^*$ is a functor and that $j^*\circ \chi=\chi\circ j^*$ follows immediately from the definitions. 
    That $j^*$ is a chain map is less obvious, but straightforward. We check this using the power of mindless computation.

     \noindent{\emph{ The case for $\alpha_0(\Omega)$:}}

      \begin{align*}
        j^*(d\alpha_0(\Omega))&=j^*(p^{r_\Gamma(\Omega)-m(\Omega)}\alpha_1(\Omega))\\
                              &=\sum_{\Psi\in C(\Omega\cap \Delta)} p^{r_\Gamma(\Omega)-m(\Omega)+r_\Delta(\Psi)-r_\Gamma(\Omega)}\alpha_1(\Psi)\\
        d(j^*\alpha_0(\Omega))&=d(\sum_{\Psi\in C(\Omega\cap \Delta)}\alpha_0(\Psi))\\
                             &=\sum_{\Psi\in C(\Omega\cap \Delta)} p^{r_\Delta(\Psi)-m(\Psi)}\alpha_1(\Psi))
      \end{align*}

      These obviously agree, as they should.

     \noindent{\emph{ The case for $\rho_0(\Omega)$:}}

    \begin{align*}
      j^*(d\rho_0(\Omega))&=j^*\left (p^{r_\Gamma(\Omega)-r^L_\Gamma(\Omega)}\alpha_1(\Omega)-\sum_{\Theta\in \Phi(\Omega)}\alpha_1(\Theta)\right)\\
                          &=\sum_{\Psi\in C(\Omega\cap\Delta)} p^{r_\Gamma(\Omega)-r^L_\Gamma(\Omega)+r_\Delta(\Psi)-r_\Gamma(\Omega)}\alpha_1(\Psi)\\
      &\phantom{==}-\sum_{\Theta\in \Phi(\Omega)}\sum_{\Lambda\in C(\Theta\cap \Delta)} p^{r_\Delta(\Lambda)-r_\Gamma(\Theta)}\alpha_1(\Lambda)\\
      d(j^*\rho_0(\Omega))&=d\left (\sum_{\Psi\in C(\Omega\cap \Delta)}p^{r^L_\Delta(\Psi)-r^L_\Gamma( \Omega)}\rho_0(\Psi)\right)\\
                          &
                           =\sum_{\Psi\in C(\Omega\cap \Delta)}p^{r_\Delta^L(\Psi)-r_\Gamma^L(\Omega)}
                            \left(p^{r_\Delta(\Psi)-r^L_\Delta(\Psi)}\alpha_1(\Psi)-\sum_{\Lambda\in \Phi(\Psi)}\alpha_1(\Lambda)\right)
     \end{align*}

     What is immediately obvious here is that the coefficients for the various $\alpha_1(\Psi)$ agree. We also have to check that
 \[    
   \sum_{\Theta\in \Phi(\Omega)}\sum_{\Lambda\in C(\Theta\cap \Delta)} p^{r_\Delta(\Lambda)-r_\Gamma(\Theta)}\alpha_1(\Lambda)=
   \sum_{\Psi\in C(\Omega\cap \Delta)}\sum_{\Lambda\in \Phi(\Psi)}p^{r_\Delta^L(\Psi)-r_\Gamma^L(\Omega)}\alpha_1(\Lambda)
\]
For this, we first notice that by remark~\ref{double.intersection} the index sets of $\Lambda$ in these two double sums agree. We also have to make sure that $r_\Delta(\Lambda)-r_\Gamma(\Theta)=r_\Delta^L(\Psi)-r^L_\Gamma(\Omega)$. But this follows from that $\Lambda\in \Phi(\Psi)$ and
$\Theta\in \Phi(\Omega)$, so that
$r_\Delta(\Lambda)=r^L_\Delta(\Psi)$ and
$r_\Gamma(\Theta)=r^L_\Gamma(\Omega)$.

     \noindent{\emph{ The case for $\rho_{-1}(\Omega)$:}}

\begin{align*}
  j^*(d\rho_{-1}(\Omega))&=j^*\left( \alpha_0(\Omega)-p^{r^L_\Gamma(\Omega)-m(\Omega)}\rho_0(\Omega)-
                          \sum_{\Theta\in \Phi(\Omega)}p^{m(\Theta)-m(\Omega)}\alpha_0(\Theta)\right)\\
                         &=\sum_{\Psi\in C(\Omega\cap\Delta)}\alpha_0(\Psi)
                           -\sum_{\Psi\in C(\Omega\cap \Delta)}p^{r^L_\Gamma(\Omega)-m(\Omega)}\left( p^{r^L_\Delta(\Omega)-r^L_\Gamma(\Omega)}\rho_0(\Psi) \right)\\
                         &\phantom{Zog}-\sum_{\Theta\in \Phi(\Omega)}\sum_{\Lambda\in C(\Theta\cap \Delta)}p^{m(\Theta)-m(\Omega)+r_{\Delta}(\Theta)-m(\Theta)}\alpha_0(\Lambda)\\
  d(j^*(\rho_{-1}(\Omega)) &= d(\sum_{\Psi\in C(\Omega\cap \Delta)}\rho_{-1}(\Psi))\\
                           &=\sum_{\Psi\in C(\Omega\cap \Delta)} \alpha_0(\Psi) - \sum_{\Psi\in C(\Omega\cap \Delta)}p^{r^L_\Delta(\Omega)-m(\Omega)}\rho_0(\Omega)\\
              &\phantom{Zog}  -\sum_{\Psi\in C(\Omega\cap \Delta)}\sum_{\Lambda\in \Phi(\Psi)}p^{m(\Theta)-m(\Omega)+r_\Delta(\Theta)-m(\Theta)}\alpha_0(\Lambda)
\end{align*}

The coefficients of $\alpha_0(\Psi)$ and $\rho_0(\Omega)$ in these sums are obviously equal. 
As in the previous case, we use remark~\ref{double.intersection} to check that the index sets in the double sums agree.  
This concludes the proof that $j^*$ is a chain map.    
\end{proof}
\end{lemma}

\section{Varying the weights}
\label{sec:varying}
In this paragraph, we will initiate a study of how the 
order $t(\Gamma,\vk)$ of the torsion subgroup of $H^1(\Gamma,\vk)$ varies with the weights $\vk$.

That is, we are fixing the graph $\Gamma$, and varying the weights $\vk$. In this section, we will give a preliminary answer in the form that this order is essentially determined by the cardinality of the fundamental forest. In the next section we will continue this discussion.

The general setup is as follows.
We fix a prime $p$ and a graph
 $\Gamma$. We also assuming that the weights are
powers of $p$. Given a vector of non-negative integers $\lao k=\{\la k v\}_{v\in V(\Gamma)}$, we write 
such a weight vector as $\lvk=\{p^{k_v}\}_{v\in V(\Gamma)}$.
We want to study the order 
$t(\Gamma,\lvk)$ of the torsion subgroup of $H^1(\Gamma,\lvk)$.  

This is given by a function $\phi_{\Gamma,p}$ which to a set of weights $\lao k$ orders a non-negative integer  $\phi_{\Gamma,p}(\lvk)$ such that
\[
t(\Gamma,\lvk)=p^{\phi_{\Gamma,p}(\lao k)}.
\]

\begin{remark}
If $p,q$ are odd primes, $\phi_{(\Gamma,p)}=\phi_{(\Gamma,q)}$.
\begin{proof}
  According to remark~\ref{cor.main.theorem} we can express $\phi_{(\Gamma,p)}$ in terms of the
  structure of the fundamental forest, with no reference to $p$. But the fundamental forest is independent of $p$ as long as $p$ is odd by remark~\ref{odd.fact}.
\end{proof}
\end{remark}

We will first consider the cases where $\Gamma$ is a tree.
In these case, we do not need to deal with the fundamental forest.  
\begin{lemma}
\label{tree.order}
Let $\Gamma$ be a tree. Let $u(v)$ be the valence of the vertex $v\in V(\Gamma)$.
Then $t(\Gamma,\vk)=GCD(k_v\vert v\in V(\Gamma))\prod_{v\in V(\Gamma)} (k_v)^{u(v)-1}$.
\begin{proof}
Since a tree is bipartite, $(\Gamma,\vk)$ is oriented for any $\vk$,
so that $H^0(\Gamma,\vk)$ is a copy of the integers, generated by 
$\funddivZ{\Gamma,\vk}$.
We can find a filtration $\Gamma_1 \subset \Gamma_2\subset \dots \subset\Gamma_r=\Gamma$ where 
$\Gamma_i$ is a tree with $i$ vertices. 

We now write 
$V(\Gamma_i)=V(\Gamma_{i-1})\cup \{v_i\}$ and 
    $E(\Gamma_i)=E(\Gamma_{i-1})\cup\{e_i\}$. The edge $e_i$ has the
end points $v_i\not\in V(\Gamma_{i-1})$ and $w_i\in V(\Gamma_{i-1})$. The relative chain
complex $C^*(\Gamma_i,\Gamma_{i-1})$ is
$\ZZ[v_i]\xrightarrow{k_{w_i}} \ZZ[e_i]$ with homology 
$\ZZ/k_{w_i}$ concentrated in dimension 1. 

On cohomology in degree zero the induced map
\[
\ZZ \cong H^0(\Gamma_i,\vk)\to H^0(\Gamma_{i-1},\vk)\cong \ZZ
\]
is multiplication by 
$a_i=GCD(k_{v_1},k_{v_2},\dots k_{v_{i-1}})/GCD(k_{v_1},k_{v_2},\dots k_{v_i})$.

There is a long exact 
 cohomology sequence of the pair $(\Gamma_i,\Gamma_{i-1})$:
\[
0 \to \ZZ \xrightarrow{a_i} \ZZ \to \ZZ/k_{w_i}\to H^1(\Gamma_i,\vk) \to H^1(\Gamma_{i-1},\vk)\to 0.
\]
We see from this sequence that $H^1(\Gamma_i,\vk)$ is a finite group for all $i$ ,
and that 
\[
\vert H^1(\Gamma_i,\vk) \vert= \vert H^1(\Gamma_{i-1},\vk) \vert\cdot \frac{k_{w_i}}{a_i}.
\]

It follows that
\[
\vert H^1(\Gamma,\vk) \vert=\prod_i k_{w_i} \prod_i 
\frac{GCD(k_{v_1},k_{v_2},\dots k_{v_i})}
{GCD(k_{v_1},k_{v_2},\dots k_{v_{i-1}})}
=GCD(k_v\vert v\in V(\Gamma))\prod_{v\in V(\Gamma)} (k_v)^{u(v)-1}.
\]
\end{proof}
\end{lemma}

We see that if $\Gamma$ is a tree, it follows from lemma \ref{tree.order} that 
$\phi_{\Gamma,p}$ is a function which is  independent of $p$, namely
\[
\phi_{\Gamma,p}(\lao k)=\min_{v\in V(\Gamma)} \la kv +\sum_{v\in V(\Gamma)} (u(v)-1)\la kv.
\]
Moreover, this formula defines a concave function $\phi_{\Gamma,p}:\RR^{E(\Gamma)}\to \RR$.

\subsection{The edge weighted version of the theory}
Up to now, we have been working with vertex weighted graphs. There is a version of the theory for edge weighted graphs. We discuss this version now.

Let $\{k_e\}_e$ be a set of weights on the edges of a graph $\Gamma$. We assume that each $k_e$ is a natural number (in particular, non-zero).   
We consider a chain complex $C^*(\Gamma,\vk^E)$ given
as
\[
C^0=\ZZ[V(\Gamma)]\xrightarrow{d^0_E} C^1=\ZZ[E(\Gamma)],
\]
where $d^0_E(v)=\sum_{e(v,w)} k_{e(v,w)}v$. We define the edge-weighted graph cohomology $H^0_E(\Gamma,\vk)$ to be the cohomology of this complex.

Given a vertex weighing $k_v$ of a graph, we can define a corresponding edge weighing 
by $k_{e(v,w)}^E=k_vk_w$. Not every edge weighing can be obtained in this way, but we are only going to consider such edge weightings that are obtained from vertex weightings.

The short exact sequence of chain complexes
\[
  \begin{CD}
    0 @>>>    \ZZ[V(\Gamma)] @>{\oplus_v k_v}>> \ZZ[V(\Gamma)] @>>> \oplus_v \ZZ/k_v @>>> 0\\
    @. @V{d_E^0}VV @V{d^0}VV  @VVV\\\
     0 @>>>    \ZZ[E(\Gamma)] @= \ZZ[E(\Gamma)] @>>> 0 @>>> 0\\   
  \end{CD}
\]  
leads to an exact sequence
\begin{equation}
\label{eq:edge}
0\to H^0_E(\Gamma,\vk) \to H^0(\Gamma,\vk) \to \oplus_v  \ZZ/k_v 
\to  H^1_E(\Gamma,\vk) \to H^1(\Gamma,\vk)\to 0.
\end{equation}
If $\Gamma$ is not bipartite, the group $H^0(\Gamma,\vk)$ is trivial, and
we have a short exact sequence
\[
0\to
\oplus_v  \ZZ/k_v 
\to  \mathrm{Torsion}(H^1_E(\Gamma,\vk)) \to \mathrm{Torsion}(H^1(\Gamma,\vk))\to 0.
\]
If $\Gamma$ is bipartite, the image of the map 
$H^0(\Gamma,\vk) \to \oplus_v  \ZZ/k_v$ is cyclic of order equal to the greatest common
divisor $GCD(\vk)$ of the numbers $k_i$, so that we obtain an exact sequence
\[
0\to \ZZ/GCD(\vk)\to 
\oplus_v  \ZZ/k_v 
\to  \mathrm{Torsion}(H^1_E(\Gamma,\vk)) \to \mathrm{Torsion}(H^1(\Gamma,\vk))\to 0.
\]
These exact sequences do not necessarily split, so that even if we know
the torsion subgroup of $H^1_E(\Gamma,\vk)$, this does not guarantee that we
know the torsion subgroup of $H^1(\Gamma,\vk)$. However, we have now proved the following result.
\begin{lemma}
\label{euler.characteristic}
If  $\Gamma$ is not bipartite, 
\[
 \vert\mathrm{Torsion}(H^1(\Gamma,\vk))\vert= \frac{\vert   \mathrm{Torsion}(H^1_E(\Gamma,\vk))   \vert}{\prod_v k_v}.
\] 
If $\Gamma$ is bipartite,
\[ 
\vert\mathrm{Torsion}(H^1(\Gamma,\vk))\vert= \frac{GCD(\vk)\vert   \mathrm{Torsion}(H^1_E(\Gamma,\vk))   \vert}{\prod_v k_v}.
\]
\end{lemma}
We introduce the following notation
\begin{align*}
  C_0(\Gamma,\vk)&=
       \begin{cases}
         GCD(\vk) &\text{if $\Gamma$ is bipartite,}\\
         1 &\text{if $\Gamma$ is not bipartite.}
       \end{cases}\\
C_1(\Gamma,\vk)&=\prod_{v\in V(\Gamma)}k_v.\\
C_2(\Gamma,\vk)&=\vert   \mathrm{Torsion}(H^1_E(\Gamma,\vk))   \vert.
\end{align*}
In this language  we can formulate the previous lemma as 
\begin{equation}
\label{Euler}
\vert\mathrm{Torsion}(H^1(\Gamma,\vk))\vert=\frac{C_0(\Gamma,\vk)C_2(\Gamma,\vk)}{C_1(\Gamma,\vk)}.
\end{equation}
The most interesting quantity here seems to be $C_2(\Gamma,\vk)$, and we turn to a short discussion of it. 
Let $(\Gamma,\vk)$ be a negative color scheme.
\begin{defn}
  $\HBE r (\Gamma,\vk)$ is the set of subgraphs $\Delta\subset \Gamma$ satisfying the conditions
  H1,H2 H3 and H4 of section~\ref{sec:forest}, but not necessarily condition H5.
\end{defn}
We could define an edge fundamental forest with vertices $\HBE *(\Gamma,\vk)$.
\begin{lemma}
\label{C2.and.HBE}
$C_2(\Gamma,p^\vk)=p^{\vert \coprod_r \HBE r {(\Gamma,p^\vk)}\vert}$   
\begin{proof}
By lemma~\ref{H.characterization} there is an injective map $\HBE r (\Gamma,\vk) \to \HB_r (\Gamma,\vk)$
which we will think of as an inclusion. It also follows from the same lemma that
the elements of $\HBE r ( \Gamma,\vk)$ which are not contained in $\HB_r (\Gamma,\vk)$
are given by the one element graphs $(\Delta(v),r)$ with $\val p {k_v}\geq r$. In particular, 
the number of these elements is  $\sum_v k_v$.
The lemma follows from this and from lemma~\ref{euler.characteristic}.
\end{proof}
\end{lemma}

\subsection{The oriented core}

In this section we start to approach the function $\phi_{(\Gamma,p)}$. The philosophy is to
cut $\RR^{V(\Gamma)}$ into a finite number of convex cones in such a way that $\phi_{(\Gamma,p)}$
behaves nicely after restricting to one of the pieces. In other words, we classify the
possible values of $\lvk$ into finitely many cases.

Let $(\Gamma,\lvk)$ be a negative color scheme.
Let $\{x_i\}$ be the set of maximal elements of its fundamental forest.
The corresponding subgraphs of $\Gamma$ are $\Delta_i=B(x_i)$.
We put $\Delta=\coprod_i\Delta_i$, so that $\Delta_i$ is the set of components
of $\Delta$. Let $i:\Delta \subset \Gamma$ be the inclusion.
\begin{defn}
  The oriented core $OC(\Gamma,\lvk)$ of $(\Gamma,\lvk)$ at the prime $p$ is the negative color scheme
$(\Delta,i^*\lvk)$.
\end{defn}
In the example at the end of section~\ref{structure.forest} the graph $\Gamma_1$ is the oriented core of $\Gamma_0$.

The first step is to classify the possible oriented cores into finitely many classes.
Consider first the case that $p$ is odd.
The underlying graph $\Delta$ of the oriented core satisfies the following odd OC conditions. 
\begin{enumerate}[(OC1)]
\item $\Delta\subset  \Gamma$ is a disjoint union of the
  connected subgraphs $\Delta_i$. Each $\Delta_i$ is bipartite.
\item $V(\Delta)=V(\Gamma)$.
\end{enumerate}
If a subgraph $\Delta\subset \Gamma$ satisfies (OC1) and (OC2), we say that $\Delta$ is an odd oriented core graph in $\Gamma$. Evidently, there are at most finitely many such. If $p$ is an odd prime, the oriented core of $(\Gamma,\lvk)$ is an oriented core graph.

If $p=2$, the graphs $\Delta_i$ come equipped with distinguished subsets of the edges
$S_i =\{e(v,w)\in E(\Delta_i)\mid \la kv+\la kw=r(\Delta)-1\}$.
If $\Delta_i$ is bipartite, then $S_i$ is the empty set. 
If $\Delta_i$ is not bipartite, the set $S_i$ consists of the edges of $\Delta_i$ of
highest valuation, and according to lemma~\ref{orientation.two.conditions}, if we
remove these edges from $\Delta_i$, what remains is a disjoint union of connected bipartite
subgraphs. Let $S=\cup_i S_i\subset E(\Delta)$.

 For a pair $(\Delta,S)$ where
$\Delta\subset \Gamma$ and $S\subset E(\Delta)$ we formulate the even oriented core conditions:
\begin{enumerate}[(OCE1)]
\item $\Delta\subset  \Gamma$ is a disjoint union of the
  connected subgraphs $\Delta_i$. 
\item There is a possibly not connected bipartite subgraph $\Omega_i\subset \Delta_i$ such that
  $V(\Omega_i)=V(\Delta_i)$ and
  $E(\Omega_i)=E(\Delta_i)\setminus S\cap E(\Delta_i)$. In particular, if $\Omega_i$ is empty, then
  $\Delta_i$ is bipartite.
\item $V(\Delta)=V(\Gamma)$.
\end{enumerate}
If a pair $(\Delta,S)$  satisfies (OCE1), (OCE2)  and (OCE3), we say that $(\Delta,S)$
is an oriented even core graph in $\Gamma$. If $(\Gamma,2^{\lao k})$ is a negative color scheme,
the even oriented core graph $(\Delta,S)$ satisfies the even oriented core conditions.

\begin{lemma}
\label{ScopeOfCore}  
  Assume either that $p$ is odd and that $\Delta$ is an oriented odd core graph in $\Gamma$
  or $p=2$ and $(\Delta,S)$ is an oriented even core graph in $\Gamma$.
There is a convex, non-empty polyhedral cone $A(\Delta)\in \RR^{E(\Gamma)}$  so that
$\Delta$ (respectively $(\Delta,S)$) is the oriented core graph of $(\Gamma,p^{\lao k})$ if and only if $\lvk\in A(\Delta)$. 
\begin{proof}
  For each $i$, let $a_i=\max_{e\in E(\Delta_i)}\lao k_e$. Similarly, let
$r_i=\min_{e\in B(\Delta_i)} \lao k_e$ where $B(\Delta_i)$ is the edge boundary of $\Delta_i$ in $\Gamma$. $\Delta_i$ is a component of $\red {p^r}(\Gamma,\lvk)$ 
if and only if $a_i<r\leq r_i$. There is such an $r$ if and only if $a_i<r_i$. Since $a_i$ is a maximum and $r_i$ is a minimum, these equations determine a convex polyhedral cone set in $\RR^{E(\Gamma)}$. 

In the case $p=2$, $(\Delta,S)$ is the oriented core graph if in addition to the above condition we also have that $\lao k_e=r$ for all $e$ in $S$.
In this case, such an $r$ exists if all $\lao k_e$ agree (for $e\in S$) and in addition to this
$a_i<\lao k_e\leq r_i$. Again, these equalities and inequalities describe a convex polyhedral cone, since the set of solutions is the intersection of (finitely many) linear subspaces and half spaces.    
\end{proof}
\end{lemma}

A graph $\Gamma$ possesses a finite set of odd or even oriented core graphs of the graph $\Gamma$ at a prime $p$. We write these oriented core graphs as $(\Delta,S)$ respectively $\Delta$. To each oriented core graph there is a convex set
$A\subset \RR^{V(\Gamma)}$, such that $\Delta$ respectively $(\Delta,S)$ is an oriented core for $(\Gamma,\la kv)$ if and only if $\lao k\in A$.
We consider the set of weights $\la kv$ that belong to $A$. We want to reduce the study of
$t(\Gamma,\lvk)$ to a study of $t(\Delta,\lvk)$,
so we need to discuss the relation between
$t(\Gamma,\lvk)$ and $t(\Delta,\lvk)$ under the assumption that $\lao k\in A$.

The inclusion $j:\Delta\subset \Gamma$ induces a map of chain complexes $j^*:\FK^*(\Gamma,\vk)\to \FK^*(\Delta,\vk)$. By definition of $j^*$ in section~\ref{functoriality.properties} we see that $j^*(\rho_{-1}(\Omega))=\rho_{-1}(\Omega)$,
$j^*(\rho_{0}(\Omega))=\rho_{0}(\Omega)$ for $\Omega\in \FS(\Gamma)=\FS(\Delta)$. 
$j^*(\alpha_{0}(\Omega))=\alpha_{0}(\Omega)$ and finally 
$j^*(\alpha_{1}(\Omega))=\alpha_{1}(\Omega)$

Recall from lemma~\ref{FK.functoriality} that
we have a commutative diagram of chain complexes
\[
  \begin{CD}
    \FK^*(\Gamma)@>{j^*}>> \FK(\Delta)\\
    @V\chi VV @V\chi VV\\
  C^*(\Gamma)@>{j^*}>> C^*(\Delta).\\
\end{CD}
\]
The vertical maps $\chi$ are weak equivalences.
The map 
$j^*={\FK(j^*)}$ in the top row is surjective, but it is not necessarily injective.
\begin{lemma}
  \label{using.oriented.core}
  Let $\Delta$ be an oriented core of $(\Gamma,\lvk)$. Assume that  $\Gamma$ is not bipartite, and  let $\{\Delta_i\}_{i\in I}$ be the components of $\Delta$ that are bipartite. Let $r_i=r(\Delta_i)$ and $m_i=m(\Delta_i)$.
Then $t(\Gamma,\lvk)=t(\Delta,\lvk)\prod_{i\in I}p^{r_i-m_i}$.
\begin{proof}
There is a short exact sequence of chain complexes
\[
  0 \to  \mathrm{ker}(\FK^*(j)) \to\FK^*(\Gamma,\lvk) \xrightarrow{\FK^*(j)} \FK^*(\Delta,\lvk) \to  0.
\]
Our first task is to determine the cohomology of the chain complex $\mathrm{ker}(\FK^*(j))$.
This chain complex will split up into summands indexed by the components $\Delta_i$. for $i\in I$:
\[
\mathrm{ker}(\FK^*(j))\cong \oplus_i \mathrm{ker}(\FK^*(j))_i.
\]
We now describe the summands $\mathrm{ker}(\FK^*(j))_i$.
The cochain complex $\mathrm{ker}(j^*)_i$ has the following generators: 
  \begin{itemize}
   \item $\mathrm{ker}(\FK^*(j))_i^{0}(\Gamma,p^{\lao k})$ is generated by symbols $\alpha_0(\Delta_i)$.
   \item $\mathrm{ker}(\FK^*(j))_i^{1}(\Gamma,p^{\lao k}
     )$  is generated by symbols $\alpha_{1}(\Delta_i)$. 
\end{itemize}
The boundary maps are given by
 $ d(\alpha_0(\Delta_i))=p^{r_i-m_i}\alpha_1(\Delta_i)$ and
$d(\alpha_1(\Delta_i))=0$

The cohomology of $\mathrm{ker}(j^*)$ is concentrated in dimension 1. and
$H^1(\mathrm{ker}(j^*))\cong \ZZ^I$ with generators $\alpha_1(\Delta_i)$. There is an exact sequence of cohomology
\[
  0\to H^0(\FK(\Delta,p^{\lao k})\xrightarrow{\delta} H^1(\mathrm{ker}j^*) \to
  H^0(\FK(\Gamma,p^{\lao k})\to H^0(\FK(\Delta,p^{\lao k}) \to 0
\]
Since $H^0(\FK(\Delta,p^{\lao k})\cong \ZZ^I$ generated by the classes $\alpha_0(\Delta_i)$
and $\delta(\alpha_0(\Delta_i))=p^{r_i-m_i}\alpha_1(\Delta_i)$, the cokernel of $\delta$
is a finite group of order $\prod_{i\in I}p^{r_i-m_i}$. The lemma follows from this computation and the exact sequence.
\end{proof}
\end{lemma}

\subsection{The oriented case}
In this subsection, we assume that $(\Gamma,p^{\lao k})$ is connected, $r$--reduced and $\ZZ/p^r$--orientable.

Let $u,v\in V(\Gamma)$.  
There is a biggest number $q_{\Gamma}(u,v)$ such that $u,v$ are not in the same component
of $\red {p^{q_\Gamma(u,v)}} \Gamma$. 

Consider subgraphs  $\Delta\subset \Gamma$ with the properties (T): 

\begin{enumerate}
\label{oriented.core}
\item $V(\Delta)=V(\Gamma)$.
\item For every pair of vertices $v,w\in V(\Delta)$, we have that $q_\Delta(v,w)=q_\Gamma(v,w)$.  
\end{enumerate}

\begin{defn}
  A weighted spanning tree of $(\Gamma,p^{\lao k})$ is a minimal subgraph satisfying the properties (T).
\end{defn}

Obviously such a subgraph exists. It is not quite obvious that it is a tree.

\begin{lemma}
\label{spanning.tree}  
Assume that $(\Gamma,p^{\lao k})$ is connected, $r$--reduced and $\ZZ/p^r$--orientable.
A weighted spanning tree $\Delta$ is a spanning tree of $\Gamma$ in the usual sense.
Moreover, the restriction $H^1(\Gamma,p^{\lao k})\to H^1(\Delta,p^{\lao k})$ induces an isomorphism of torsion subgroups. 
\begin{proof}
Every subgraph $\Omega\subset \Gamma$ is
also orientable, so that the set $\HB r (\Gamma,p^\vk)$ is simply the set of components of 
$\red {p^r} \Gamma$.

Now consider any subgraph $\Delta\subset \Gamma$ satisfying the properties (T).
Note that any graph $\Omega$ such that $\Delta\subset\Omega\subset\Gamma$ also satisfies
the conditions (T).
Since $r_\Delta(v,w)=r_\Gamma(v,w)<\infty$, the graph $\Delta$ is connected.
We claim that for every $r$, the inclusion
$\red r \Delta\subset \red r \Gamma$.
induces a bijection of the sets of components. By induction, it suffices to
consider the case that 
$\Delta$ only differs from $\Gamma$ by one edge, that is,
$E(\Gamma)=E(\Delta)\cup\{e\}$.
Let $v,w$ be the vertices on $e$.
Our claim is that if the vertices $v,w$ are in the same component of $\red r \Gamma$, they are also in the same components in $\red r \Delta$. If $v,w$ are in the same component of  $\red r \Gamma$, by the assumption on $\Delta$, $q_\Delta(v,w)=q_\Gamma(v,w)\geq r$. It follows from this that $v$ and $w$ are in the same component of $\red r\Delta$.

Now assume that $\Delta$  is a weighted spanning tree in $\Gamma$.
We show that $\Delta$ does not contain 
any cycle $C\subset \Delta$. Suppose it does. Let $e(v,w)\in C$ be an edge in $C$ of maximal weight.  
Let $\Delta'\subset \Delta$ be the subgraph we get by removing $e(v,w)$ from $E(\Delta)$. We claim that
$r_{\Delta'}(w_1,w_2)=r_{\Delta}(w_1,w_2)$ for all $w_1,w_2\in V(\Delta)$ in contradiction to the minimality of $\Delta$. For this is suffices to see that for every $r$, the inclusion
$\red {p^r}\Delta'\subset \red {p^r}\Delta$ induces a bijection of components.
  
If $r<r_\Delta(u,v)$ this is because the map is an isomorphism of graphs.
If $r\geq r_\Delta$, the graphs $\red {p^r} \Delta'$ and $\red {p^r} \Delta$ differ by the single
edge $e(u,v)$. The inclusion of this edge does not change the number of components, since
$C\setminus \{e(u,v)\}\subset \red {p^r} \Delta'$, so that the end points of $e(u,v)$ are
already in the same component in $\red {p^r} \Delta'$.
Since $\Delta$ is connected, contains no cycles and contains all the vertices of $\Gamma$, it is a spanning tree.

Finally we turn to the statement about the cohomology.
We prove that if $\Delta$ satisfies the conditions (T), then the map $H^1(\Gamma,p^{\lao k})\to H^1(\Delta,p^{\lao k})$ induces an isomorphism on cohomology.

Since the inclusion $\Delta\subset \Gamma$ induces a bijection on components of
$\red r \Delta\subset \red r \Gamma$, the sets $\HB_r(\Delta)$ and
$\HB_r(\Gamma)$ have the same cardinality.
Using corollary~\ref{cor.main.theorem} we conclude that
the orders of the torsion subgroups in
$H^1(\Gamma,p^{\lao k})$ and $H^1( \Delta,p^{\lao k})$ agree.

There is an exact sequence
\[
  H^0(\Gamma,p^{\lao k}) \to H^0(\Delta,p^{\lao k})\to
H^1(\Gamma,\Delta,p^{\lao k})\to H^1(\Gamma,p^{\lao k}) \to H^1(\Delta,p^{\lao k})\to 0
\]
The cokernel of $H^0(\Gamma,p^{\lao k}) \to H^0(\Delta,p^{\lao k})$ is a finite cyclic group.
Since we also know that $H^1(\Gamma,\Delta,p^{\lao k})$ is a free Abelian group,
we can split off a short exact sequence
\[
0\to H^1(\Gamma,\Delta,p^{\lao k})\to H^1(\Gamma,p^{\lao k}) \to H^1(\Delta,p^{\lao k})\to 0
\]
It follows formally from this short exact sequence that 
the map
\[
\mathrm{torsion}\left(H^1(\Gamma,p^{\lao k})\right) \to \mathrm{torsion}\left(H^1(\Delta,p^{\lao k})\right)
\]
is injective. Since the orders of these groups are the same, the map is an isomorphism.

\end{proof}
\end{lemma}

We can now give a procedure for determining the order of the torsion subgroup in
$H^1(\Gamma,p^\lvk)$ for an oriented graph $\Gamma$. First, we find the oriented
core $\Delta$ in the even case or $(\Delta,S)$ in the case $p=2$ by checking a finite number of equalities and inequalities between linear combinations of the weights.
In this oriented core we find a weighted spanning tree $T$. We can for instance do this inductively, where the induction step is to remove an edge which is part of a cycle, and has maximal weight among such edges. Again, this weighted spanning tree is determined by a finite number of inequalities. The set of weights corresponding to a given weighted spanning tree is the integral points of a certain convex subset of $\RR^{V(\Gamma)}$.  
For $v\in V(\Gamma)$, let $u(v)$ be the valence of $v$ in the spanning tress $T$. 

Combining  lemma~\ref{spanning.tree}  and lemma~\ref{tree.order} we obtain
\begin{thm}
  \label{order.oriented.graph}
  The order of the $p$-torsion in $H^1(\Gamma,p^{\lao kv})$ is $p^N$ where
  \[
N=\min_v {\lao k}_v+ \sum_v (u(v)-1) {\lao k}_v.
  \]
\end{thm}
A more abstract but slightly less precise result is the following immediate consequence:
\begin{corollary}
  Given a graph $\Gamma$ and a prime $p$, there is a finite number of convex cones in
  $\RR^{V(\Gamma)}$ such that restriction of the function $\phi_{(\Gamma,p)}$ is linear on each of these convex cones.
\end{corollary}

\section{Tropical numbers}
\label{sec:tropical-cohomology}

In this section we reformulate the description of the order of the torsion of $H^1(\Gamma,\lvk)$ using the language of max-plus rings, or equivalently, by using tropical polynomials.
For simplicity we will restrict ourselves to the case of odd primes $p$, because it makes the situation a little easier that in this case $\ZZ/p^r$--orientability
is equivalent to bipartiteness.
This is not an essential restriction. The situation at the prime 2 is more complicated, but presumable manageable. We will not deal with that case in this paper.

We recall the tropical language that we will use.
We consider two tropical commutative semi-rings. There is a  tropical commutative semi-rings structure on
$\ZZ\cup\{\infty\}$ given by the tropical sum 
$a\oplus b = \min(a,b)$ and tropical multiplication $a\odot b = a+b$.
There is also a
 tropical commutative semi-rings structure on
non-negative integers $\NN\cup\{0\}$ and the non-negative rational numbers $\QQ_+\cup\{0\}$ given by tropical sum 
$a\oplus b = ab$ and tropical multiplication $a\odot b = GCD(a,b)$.
The zero elements of these semi-rings  are $\{\infty\}$ respectively $\{0\}$. The obvious inclusion
$\NN\cup\{0\}\subset  \QQ_+\cup\{0\}$ is obviously a homomorphism of tropical semirings.

For each prime $p$ there are
tropical commutative semi-ring homomorphisms $\log_p:\QQ_+\cup\{0\}\to\ZZ\cup\{\infty\}$
given by the $p$--valuation $\log_p(k)=\val p k$ and $\exp_p:\ZZ\cup\{\infty\}\to\QQ_+\cup\{0\}$ given by $p$--exponentiation
$\exp_p(\lao k)=\lvk$. The product of the logarithm maps
\[
\prod_p \log_p: \QQ_+\cup\{0\}\to \prod_p \ZZ\cup\{\infty\}
\]
is injective .

In these semi-rings
the tropical addition does not have cancellation, that is,  there is no way to define subtraction.
However, in $\ZZ\cup\{\infty\}$ and $\QQ_+\cup\{0\}$ every element except the zero element is multiplicatively invertible, and we can define the tropical quotient $a\oslash b$ for any $b$ which is not the zero element. 

Fix a finite set  $A$.
A tropical monomial $\lambda$ in $A$ is an iterated tropical product of the variables
$a$, where $a$ ranges over the elements of $A$. 
It  is given by the non-negative integers $I=\{i_a\}_{a\in A}$, counting the exponents  
of the variables $a$. The degree of the monomial is $\sum_a i_a $ (not a tropical sum this time).
 A tropical polynomial is the tropical sum of a set
of tropical monomials. A tropical rational function $f$ is the tropical quotient of two polynomials.

An $A$--integer is a family  $\lao k =\{k_a\}_{i\in A}$ of
elements of $\ZZ\cup\{\infty\}$
parametrized  by the elements of $V$. Similarly, an $A$--natural number $\vk$ is a family of
elements of $\NN\cup\{0\}$, parametrized by $A$. 
We are interested in the cohomology cochain complexes parametrized by $A$--natural numbers.
To describe how numerical invariants of the cohomology of these complexes 
change according to the parameter, it will be useful to describe certain functions of the $A$ integers $\lao k$
as evaluations of tropical rational functions. 

We can evaluate a tropical rational function $f$ in $A$ 
on the $A$--integer $\lao k$ using the semi ring structures.
Since $\log_p$ and $\exp_p$ are tropical homomorphism,
$f(\log_p(\vk))=\log_p(f(\vk))$, and similarly $f(\exp_p(\lao k))=\exp_p(f(\lao k))$.  

We will use a few special examples of the evaluations of tropical functions on $A$--integers. 
For instance, if $X\subset A$, the tropical monomial $\chi^{}_X=\sum_{a\in A} a$ evaluates
to $\min_{a\in X} \lao k_a$. The tropical rational function
\[
\chi_X^*=\frac{\chi^{}_X}{\sum_{a\in X}\chi^{}_{X\setminus \{a\}}}
\]
evaluates to $\max_{a\in X} \lao k_a$. In this formula, the sum in the denominator is
to be interpreted as the tropical sum.


We also remark that we can form the tropical elementary symmetric functions
in $A$ using the usual formulas:
\[
\sigma_i(A)=\sum_{X\subset A\mid \text{card}(X)=i}\left( \prod_{a\in X} a\right)
\]
In this formula, both the product and the sum are to be taken in the tropical sense.
If we evaluate the tropical polynomial $\sigma_i$ on an $A$-integer $\lao k$, we obtain the sum of the $i$ smallest numbers
from $\{\lao k_a\}_{a\in A}$.

\begin{thm}
  \label{tropical.interpretation}
 For any graph $\Gamma$ there is a tropical rational function $Z_\Gamma$
  in indeterminates indexed by the vertices of $\Gamma$ such that the
  for any vertex weighing $\lao k$, for any odd prime $p$ the order of the $p$-torsion of
  $H^1(\Gamma,\lvk)$ equals the evaluation $Z_\Gamma(p^\lvk)$ in the tropical ring $\NN\cup \{0\}$.
  \begin{proof}
    Using the K\"u{}nneth theorem, we easily reduce the theorem to the case that $\Gamma$ is connected.
    
    According to corollary~\ref{cor.main.theorem}, the order of the $p$-torsion equals $p^N$ where $N$ 
is the cardinality of the set of elements if $\HB_*(\Gamma,\lvk)$ that are not in $\cap_n\mathrm{Im}(s^n)$.
In order to prove the theorem, we need to show find $Z_\Gamma$ so that this cardinality is the evaluation
$Z_\Gamma(\lao k)$ in the tropical ring $\ZZ\cup \{\infty\}$. 

   The next remark is that the vertex weights determine edge weights, and that the edge weights are
tropical degree 2 monomials evaluated at the vertex weights: $\lao k_{e(v,w)}=\lao k_v \odot \lao k_w$. So a tropical
rational function in the edge weight can by composition with these monomials be written as a tropical
rational function in the vertex weights, and it is sufficient to give $Z_\Gamma$ as a tropical
rational function in edge weights and vertex weights.

Assume that the subgraph $\Delta\subset \Gamma$ is bipartite and connected, but not equal to $\Gamma$.  
It follows from 
lemma~\ref{H.characterization} that $(\Delta,r)\in \HB_r$ if 
and only if the following conditions are satisfied:
\begin{align*}
  \max_{e\in E(\Delta)} \lao k_e &< r \leq \min_{b\in B(\Delta)}\lao k_b,\\
  \min_{v\in V(\Delta)} \lao k_v &< r.
\end{align*}
That is, if
\[
\max(\max_{e\in E(\Delta)}\lao k_e,\min_{v\in V(\Delta)}\lao k_v)<r\leq \min_{b\in B(\Delta)}\lao k_b.
\]
Since we can express a maximum as a tropical rational function, this can be written as
$
f_1(\lao k)<r\leq f_2(\lao k)
$,
where $f_1$ and $f_2$ are tropical rational functions in $V(\Gamma)$.
Since $\Delta\not=\Gamma$ and $\Gamma$ is connected, the set $B(\Delta)$ is non-empty, and 
$\min_{b\in B(\Delta)}\lao k_b\not =\infty$.

This means that $f_2(\lao k)\not=\infty$, and the number of such integers $r$ equals
\[
  g_\Delta(\lao k)=
    \begin{cases}
    f_2(\lao k)-f_1(\lao k)&\text{ if $f_2(\lao k)-f_1(\lao k)>0$}\\
    0&\text{ else}
  \end{cases}
  \bigg \}
  =f_2(\lao k)\oslash(f_1(\lao k)\oplus f_2\lao k)).
\]
Form the tropical product $Z_\Gamma=\prod_{\Delta} g_\Delta$, where the sum is taken over all
bipartite $\Delta$ strictly contained in $\Gamma$. The evaluation
$Z_\Gamma(\lao k)$ 
of this tropical function equals the cardinality of
\[
  \bigcup_{\Delta\not=\Gamma}\{(\Delta,r)\in HB_r(\Gamma,\vk)\}=
  \HB_*\setminus \{(\Gamma,r)\in \HB_r(\Gamma,\vk)\}=
  \HB_*\setminus \cap_r\im(s^n).
\]  
This completes the proof of the theorem.
  \end{proof}
\end{thm}

Actually, the proof of the theorem gives a formula for $Z_\Gamma$.
 For a bipartite. connected
subgraph $\Delta\in K_n$, we write $g_\Delta$ for the tropical rational function in the vertices of
$\Delta$ given by the formula
\[
g_\Delta(\lao k)= \max(\min_{b\in B(\Delta)}\lao k_b - \max_{e\in E(\Delta)}\lao k_e,0)
\]
According to the proof of the preceding theorem,
\begin{equation}
  \label{eq:Z}
Z_\Gamma=\bigodot_{\Delta\not=\Gamma}g_\Delta.  
\end{equation}
The product is to be interpreted in the tropical sense. In the evaluation, this product becomes an (ordinary) summation
\[
Z_\Gamma(\lao k)=\sum_{\Delta\not=\Gamma}g_\Delta(\lao k).  
\]

\subsection{The complete graphs}

A particularly simple example is the case when $\Gamma=K_n$ is a complete graph on $n$ vertices.
We will compute the corresponding tropical function $Z_{K_n}$.
We will not use formula~\ref{eq:Z}. Instead we will compute the cardinality of
$\HB_r(\Gamma,\vk)$ directly. 
To do this, we first determine which bipartite, connected
subgraphs $\Delta\subset \Gamma$  can occur as a components of
$\red r (\Gamma,\lao k)$ for a given $r$.

\begin{lemma}
  \label{complete.bipartites}
  If a bipartite connected subgraph $\Delta\subset K_n$ is a component of $\red r (K_n,\lao k)$,
  then either $\Delta=\Delta_v$ is a one vertex graph, or $\Delta$ contains a vertex $v$ such that
  $E(\Delta)=\big\{e(v,w)\mid w\in V(\Delta)\setminus \{v\}\big\}$. Moreover,
  there is an $r$ such that $V(\Delta)=\{w\in V(K_n)\mid \lao k_v<r\}$.
  \begin{proof}
    Assume that $\Delta$ is a connected component of $\red r (K_n,\lao k)$.
    Let us order the vertices of $K_n$ according to the weights, so that $V(K_n)=\{v_i\mid 1\leq i\leq n\}$ and
    $k_{v_1}\leq k_{v_2}\leq \dots\leq k_{v_n}$. Put $v=v_1$. If $\Delta=\Delta_{v_1}$, we are in the first case of the lemma, so we can
    without restriction assume that
    there is some $v_i\in V(\Delta)$ such that $e(v,v_i)\in E(\Delta)$.  Let $i_{max}>1$ be the maximal $i$
    such that $e(v,v_i)\in E(\Delta)$. If $i\leq v_{max}$, $\lao k_{v_1}+ \lao k_{v_i}\leq \lao k_{v_1}+ \lao k_{max}<r(\Delta)$, so that $e(v,v_i)\in E(\Delta)$.
    If $i>i_{max}$, then $e(v,v_i)\not\in E(\Delta)$, and $\lao k_{v_1}+ \lao k_{v_i}\geq r(\Delta)$. It follows that for any $j$,
    $\lao k_{v_j}+\lao k_{v_i}\geq \lao k_{v_1}+\lao k_{v_i}\geq r(\Delta)$, so that $v_j$ is a vertex incident to no edges in $\red r{K_n}$. It follows that none of these vertices are in $V(\Delta)$.

    Next, we check that there are no edges $e(v_i,v_j)\in E(\Delta)$ for $i,j\not=1$. However, since
    $e(v_1,v_i)$ and $e(v_1,v_j)$ are edges of $\Delta$, such an edge  would contradict the bipartiteness of $\Delta$.
    This means that every edge in $E(\Delta)$ is incident to $v$, which proves the statement of the lemma about the structure of $\Delta$. The statement about $V(\Delta)$ is also clear after choosing $r=r(\Delta)-\lao kv$.
  \end{proof}
\end{lemma}

We  can now describe the fundamental forest of $(K_n,\lao k)$. Let us assume that we have ordered the vertices as in the proof of lemma~\ref{complete.bipartites}. First, there are the
one vertex subgraphs $\Delta_{v_i}$ for $v_i\in V(K_n)$. If $i\geq 2$, the  
$(\Delta_{v_i},r)\in \HB_r(K_n)$ for $\lao k_{v_i} < r \leq \lao k_{v_1}+\lao k_{v_i}$. There are
$(n-1)\lao k_{v_1}$ of these elements of $HB_*(\Gamma,\vk)$.
The one vertex graph $(\Delta_{v_1},r)$ is in $\HB_r(K_n)$ for
$\lao k_{v_1}<r\leq \lao k_{v_1}+\lao k_{v_2}$. There are $\lao k_{v_2}$ of those possible values of $r$.

Then we have the graphs which have at least two vertices.
If we order the edges incident to  $v_1$ according to increasing  weights we get
\[
  e_{v_1,v_2} \leq e_{v_1,v_3} \dots \leq e_{v_1,v_n}
\]
The edge with smallest weight which is not in this list is $e_{v_2,v_3}$. There are two possibilities.
Either, for some
$m$ we will have that $\lao k_{v_1}+\lao k_{v_m}<\lao k_{v_2}+\lao k_{v_3}\leq \lao k_{v_1}+\lao k_{v_{m+1}}$, or
$\lao k_{v_1}+\lao k_{v_i}<\lao k_{v_2}+\lao k_{v_3}$ for all $i$. In the last case, we set $m=n$, and see that
the list of edges ordered after weight up to and including $e_{v_2,v_3}$ are
\[
e_{v_1,v_2} \leq e_{v_1,v_3} \dots \leq e_{v_1,v_m}\leq e_{v_2,v_3}.
\]

For
$i\leq m$, there graph $\Delta_i$ with vertices $\{v_j\mid j\leq i\}$.
If
$i<m$,
there are
$(\lao k_{v_{i+1}}+\lao k_{v_1})-(\lao k_{v_{i}}+\lao k_{v_1})=\lao k_{v_{i+1}}-\lao k_{v_{i}}$  elements
of the form $(\Delta_i,r)$ .
Finally, there are $\lao k_{v_2}+\lao k_{v_3}-(\lao k_{v_1}+\lao k_{v_{m}})$ elements of the
  form $(\Delta_{m},r)$.
  
To get the number of elements of $\HB_*$, we sum up:
\begin{align*}
  (n-1)\lao k_1+\lao k_2&+\left(\sum_{2\leq i\leq m-1}\lao k_{v_{i+1}}-\lao k_{v_{i}}\right)
+\left(\lao k_{v_2}+\lao k_{v_3}-(\lao k_{v_1}+\lao k_{v_{m}})\right)
  \\
  &=(n-2)\lao k_1+\lao k_{v_2}+\lao k_{v_3}
\end{align*}

We can formulate the computation of $Z_{K_n}$ in the tropical language as follows:
\begin{thm}
  For $n\geq 3$, 
\[
Z_{K_n}(\lao k)=(\sigma_1(\lao k))^{n-3}\sigma_3(\lao k)
\]
where $\sigma_i$ is the $i^{\text{th}}$ tropical elementary symmetric functions in the vertices of $K_n$. The product is to be evaluated as a tropical product.
\begin{proof}
  $\sigma_3$ is the tropical product of the three smallest numbers in $\lao k_v$, that is
  $\lao k_1+\lao k_2 + \lao k_3$. Similarly, $\sigma_1$ is minimal number in $\lao k_v$, that is $\lao k_1$. The theorem follows. 
\end{proof}
\end{thm}
Presumably, there is also a more direct linear algebra proof of this result.

One can use the description of the fundamental forest above to determine the complete structure of
the torsion subgroup of $H^1(K_n,\vk)$. For instance, if for some $v$ the weight $k_v$ is prime to $p$, the $p$--torsion subgroup is cyclic.
\bibliography{Graphs}

\end{document}